\documentclass[12pt]{article}
\usepackage[latin1]{inputenc}
\usepackage[english]{babel}
\usepackage{graphicx}
\usepackage{amsmath,amsfonts,amssymb,amsthm,amsopn,amscd}
\usepackage{bbm,wasysym,stmaryrd}
\usepackage{subfigure}
\usepackage{color}
\usepackage{setspace}
\usepackage{authblk}
\usepackage{hyperref}
\usepackage[normalem]{ulem} 

\newcommand{\R}{\mathbbm{R}}

\newcommand{\mN}{\mathcal{N}}
\newcommand{\Z}{\mathbbm{Z}}
\newcommand{\B}{\mathcal{B}}

\newcommand{\F}{\mathcal{F}}
\newcommand{\mM}{\mathcal{M}}
\newcommand{\mone}{1}

\newcommand{\Exp}{\mathbb{E}} 

\newcommand{\Expt}[1]{\mathbb{E}\left [ #1 \right]} 
\newcommand{\T}{\mathcal{T}}  
\newcommand{\M}{\mathcal{M}}

\newcommand{\norm}[1]{\| #1 \|}
\newcommand{\lsup}[1]{\underset{#1\to\infty}{\overline{\lim}}}
\newcommand{\linf}[1]{\underset{#1\to\infty}{\underline{\lim}}}
\newcommand{\limunder}[1]{\underset{#1\to\infty}{\lim}}
\newcommand{\J}{J}
\newcommand{\undt}[1]{\underline{#1}}
\renewcommand{\Re}{{\rm Re}}
\renewcommand{\Im}{{\rm Im}}
\newcommand{\meas}{\mathcal{M}_{1}^+}
\newcommand{\mmeas}{\mathcal{M}_{1,s}^+}

\theoremstyle{plain}
\newtheorem{theorem}{Theorem}
\newtheorem{corollary}[theorem]{Corollary}
\newtheorem{proposition}[theorem]{Proposition}
\newtheorem{lemma}[theorem]{Lemma}

\theoremstyle{definition}
\newtheorem{definition}{Definition}

\newcommand{\olivier}[1]{#1}

\begin{document}

\title{A large deviation principle  for 
  networks of rate neurons with correlated synaptic weights}

\author[*]{Olivier Faugeras}
\author[*]{James Maclaurin}
\affil[*]{NeuroMathComp Laboratory, INRIA Sophia-Antipolis, France
  \hspace{2cm} {\tt \small firstname.name@inria.fr}}

\maketitle

\hrule
\medskip
\tableofcontents
\medskip
\hrule
\vspace{0.5cm}

\noindent
{\bf AMS Mathematics Subject Classification:}\\
28C20, 34F05, 34K50, 37L55, 60B11, 60F10, 60G10, 60G15, 60G57, 60G60, 60H10, 62M45, 92C20.
\section*{Abstract}
We study the asymptotic law of a network of interacting neurons when the number of neurons becomes infinite. Given a
completely connected network of firing rate neurons in which the
synaptic weights are Gaussian {\emph correlated} random variables, we describe the 
asymptotic law of the network when the number of neurons goes to
infinity. We introduce the process-level empirical measure
 of the trajectories of the
solutions to the equations of the finite network of  neurons  and the averaged law (with respect to
the synaptic weights)  of the trajectories of the
solutions to the equations of the network of neurons. The main result of this article  is that the image law through the empirical measure satisfies a large deviation principle  with a good rate function which is shown to have a unique global minimum. 
Our analysis of the rate function allows us also to characterize the limit measure as the image of a stationary Gaussian measure defined on a transformed set of trajectories. This is potentially very useful for applications in neuroscience since  the Gaussian measure can be completely characterized by its mean and spectral density. It also facilitates the assessment of the probability of finite-size effects.

\section{Introduction}   

The goal of this paper is to study the asymptotic behaviour and large deviations of a network of
interacting neurons when the number of neurons becomes
infinite. Our network may be thought of as a network of weakly-interacting diffusions: thus before we begin we briefly overview other asymptotic analyses of such systems. In particular, a lot of work has been done on spin glass dynamics, including Ben Arous and Guionnet on the
mathematical side \cite{guionnet:95,ben-arous-guionnet:95,ben-arous-guionnet:97,guionnet:97} and Sompolinsky and his co-workers on the
theoretical physics side
\cite{sompolinsky-zippelius:81,sompolinsky-zippelius:82,crisanti-sompolinsky:87,crisanti-sompolinsky:87b}. Furthermore the large deviations of weakly interacting diffusions has been extensively studied by Dawson, Gartner and co-workers \cite{dawson-gartner:87,dawson-gartner:94,dawson-del-moral:05}. More
references to previous work on this particular subject can be found in
these references. 

Because the dynamics of spin glasses is not too far
from that of networks of interacting neurons, Sompolinsky also
succesfully explored this particular topic
\cite{sompolinsky-crisanti-etal:88} for fully connected networks of
rate neurons, i.e. neurons represented by the time variation of their
firing rates (the number of spikes they emit per unit of time), as
opposed to spiking neurons, i.e. neurons represented by the time
variation of their membrane potential (including the individual
spikes). For an introduction to these notions, the interested
reader is referred to such textbooks as
\cite{gerstner-kistler:02b,izhikevich:07,ermentrout-terman:10}. In his
study of the continuous time dynamics of networks of rate neurons,
Sompolinsky and his colleagues assumed, as in the work on spin
glasses, that the coupling coefficients, called the synaptic
weights in neuroscience, were random variables 
i.i.d. with zero mean Gaussian laws. The main result
obtained by Ben Arous and Guionnet 
for spin glass networks using a large deviations approach (resp. by Sompolinsky and his colleagues for networks of rate neurons using the local chaos hypothesis) under the
previous hypotheses is that the averaged law of Langevin spin glass
(resp. rate neurons)
dynamics is chaotic in the sense that the averaged law of a finite
number of spins (resp. of neurons) converges to a product measure.

The next theoretical efforts in the direction of understanding the
averaged law of rate neurons are those of Cessac, Moynot and Samuelides
\cite{cessac:95,moynot:99,moynot-samuelides:02,cessac-samuelides:07,samuelides-cessac:07}. From the
technical viewpoint, the study of the collective dynamics is done in discrete
time, assuming no leak  (this term is explained below) in the individual dynamics of each of the rate
neurons. Moynot and Samuelides obtained a large deviation principle and were able to describe in detail the
limit averaged law that had been obtained by Cessac using the local chaos hypothesis and to  prove rigourously the propagation of chaos
property. Moynot extended these results to the more general case where the
neurons can belong to two populations, the synaptic weights are
non-Gaussian (with some restrictions)  but still i.i.d., and the
network is not fully connected (with some restrictions) \cite{moynot:99}.

One of the next challenges is to incorporate in the network model the fact
that the synaptic weights are not independent and in effect often
highly correlated. One of the reasons for this is the plasticity processes at
work at the levels of the synaptic connections between neurons; see
for example \cite{kandel-schwartz-etal:00} for a biological viewpoint,
and \cite{dayan-abbott:01,gerstner-kistler:02b,ermentrout-terman:10}
for a more computational and mathematical account of these phenomena.

The problem we solve in this paper is the following. Given a
completely connected network of firing rate neurons in which the
synaptic weights are Gaussian {\emph correlated} random variables, we describe the 
asymptotic law of the network when the number of neurons goes to
infinity. Like in  \cite{moynot:99,moynot-samuelides:02}  we study a
discrete time dynamics but unlike these authors we cope with more
complex intrinsic dynamics of the neurons, in particular we allow for
a leak (to be explained in more detail below). The structure of our proof is broadly similar to these authors; we have generalised their results. Indeed one may directly obtain the LDP in \cite{moynot:99} by applying a contraction principle to the LDP to be proved below.

To be complete, let us mention the fact that this problem has already partially been explored in Physics by Sompolinsky and Zippelius \cite{sompolinsky-zippelius:81,sompolinsky-zippelius:82} and in Mathematics by Alice Guionnet \cite{guionnet:97} who analysed symmetric spin glass dynamics, i.e. the case where the matrix of the coupling coefficients (the synaptic weights in our case) is symmetric. This is a very special case of correlation. The work in \cite{cugliandolo-kurchan-etal:97} is also an important step forward in the direction of understanding the spin glass dynamics when more general correlations are present.

Let us also mention very briefly another class of approaches toward the description of very large populations of neurons where the individual spikes generated by the neurons are considered. The model for individual neurons is usually of the class of Integrate and Fire (IF) neurons \cite{lapicque:07} and the underlying mathematical tools are those of the theory of point-processes \cite{daley-vere-jones:07}. Important  results have been obtained in this framework by Gerstner and his collaborators, e.g.  \cite{gerstner-hemmen:93,gerstner:95} in the case of deterministic synaptic weights. Related to this approach but from a more mathematical viewpoint, important results on the solutions of the mean-field equations have been obtained in \cite{caceres-carillo-etal:11}. In the case of spiking neurons but with a continuous dynamics (unlike that of IF neurons), the first author and collaborators have recently obtained some limit equations that describe the asymptotic dynamics of fully connected networks of neurons \cite{baladron-fasoli-etal:12} with independent synaptic weights.

Because of the correlation of the synaptic weights, the natural space to work in is the infinite dimensional space of the trajectories, noted $\T^\Z$, of a countably-infinite set of neurons and the set of stationary probability measures defined on this set, noted $\mmeas(\T^\Z)$.

We introduce the process-level empirical measure, noted $\hat{\mu}^N$, of the $N$ trajectories of the
solutions to the equations of the network of $N$ neurons  and the averaged (with respect to
the synaptic weights) law $Q^N$ of the $N$ trajectories of the
solutions to the equations of the network of $N$ neurons. The main result of this article (theorem \ref{theo:LDP}) is that the image law $\Pi^N$ of  $Q^N$ through $\olivier{\hat{\mu}}^N$ satisfies a large deviation principle (LDP) with a good rate function $H$ which is shown to have a unique global minimum, $\mu_e$. Thus, with respect to the measure $\Pi^N$ on $\mmeas(\T^\Z)$, if the set $X$ contains the measure $\delta_{\mu_e}$, then $\Pi^N(X) \to 1$ as $N \to \infty$, whereas if $\delta_{\mu_e}$ is not in the closure of $X$, $\Pi^N(X) \to 0$ as $N \to \infty$ exponentially fast and the constant in the exponential rate is determined by the rate function. Our analysis of the rate function allows us also to characterize the limit measure $\mu_e$ as the image of a stationary Gaussian measure $\undt{\mu_e}$ defined on a transformed set of trajectories $\T^\Z$. This is potentially very useful for applications since $\undt{\mu_e}$ can be completely characterized by its mean and spectral density. Furthermore the rate function allows us to quantify the probability of finite-size effects.

The paper is organized as follows. In section \ref{sect:model} we describe the equations of our network of neurons, the type of correlation between the synaptic weights, define the proper state spaces and introduce the different probability measures that are necessary for establishing our results, in particular the level-3 empirical measure, $\hat{\mu}^N$, $\Pi^N$ and the image $R^N$ through $\hat{\mu}^N$ of the law of the uncoupled neurons. We state the principle result of this paper in Theorem \ref{theo:LDP}. In section \ref{sect:average} we motivate our approach by showing that when computing the Radon-Nikodym derivative of $Q^N$ with respect to the law of the uncoupled neurons, one is led to consider certain Gaussian processes which are directly related to the synaptic weights and can be described with the help of the empirical measure $\hat{\mu}^N$.

In section \ref{sect:PiN} we extend the definition of the previous Gaussian processes to be valid for any stationary measure, not only the empirical one. This allows us to compute the Radon-Nikodym derivative of $\Pi^N$ with respect to $R^N$ for any measure in $\mmeas(\T^\Z)$. Using these results, section \ref{Sect:GoodRateFunction} is dedicated to the proof of the existence of a strong LDP for the measure $\Pi^N$. In section \ref{sect:Hminimum} we show that the good rate function obtained in the previous section has a unique global minimum and we characterize it as the image of a stationary Gaussian measure. We conclude with section \ref{sect:conclusion} by  stating some important consequences and sketching a number of possible generalisations of our work as well as discussing some further connections with other approaches.

\section{The neural network model}\label{sect:model}
We consider a fully connected network of $N$ rate neurons. For
simplicity but without loss of generality, we assume $N$ odd\footnote{When $N$ is even the formulae are slightly more  complicated but all the results we prove below in the case $N$ odd are still valid.} and write
$N=2n+1$, $n \geq 0$. In the course of this paper, we will asymptote $N$ to $\infty$. The state of the neurons is described
either by the rate variables $(X^j_t),\,
j=-n,\cdots,n,\,t=0,\cdots,T$ or the potential variables $(U^j_t),\,
j=-n,\cdots,n,\,t=0,\cdots,T$. 
These variables are related as follows
\[
X^j_t=f(U^j_t)\quad j=-n,\ldots,n \quad t=0,\ldots,T-1,
\]
where $f: \R \to ]0,\,1[$ is a monotonic bijection \olivier{which we assume to be Lipschitz continuous. Its Lipschitz constant is noted $k_f$.} We could for example employ $f(x)=(1+\tanh(gx))/2$, where the parameter $g$ can be used to control
the slope of the ``sigmoid'' $f$ at the origin $x=0$.


\subsection{The model equations}\label{sect:modelequations}
The equation describing the time variation of the membrane potential
$U^j$ of the $j$th neuron writes
\begin{equation}\label{eq:U}
U^j_{t}=\gamma U^j_{t-1}+\sum_{i=-n}^{n} J_{ji}^{\olivier{N}}
f(U^i_{t-1})+\theta_j+B^j_{t-1}, \quad j=-n,\ldots,n \quad t=1,\ldots,T.
\end{equation}
This equation involves the parameters $\gamma$, $J_{ij}^{\olivier{N}}$, $\theta_j$, and $B^j_t$,
$i,\,j=-n,\ldots,n$, $t=0,\ldots,T-1$. \olivier{The initial conditions are discussed at the beginning of section \ref{subsection:laws-uncoupled-coupled}.}

$\gamma$ is a positive real between 0 and 1 that determines the time
scale of the intrinsic dynamics, i.e. without interactions, of the
neurons. If $\gamma=0$ the dynamics is said to have no leak.

The $\theta_j$s are the thresholds: they change the value of the
potential of the neuron $j$ at which the sigmoid $f$ takes the value
$1/2$. Like the $J_{ij}^{\olivier{N}}$s they are {\em random variables} that we
assume to be i.i.d. as $\mathcal{N}_1(\bar{\theta},\,\theta^2)$, and
independent of the $J_{ij}$s\footnote{We note $\mN_p(m,\Sigma)$ the law of the $p$-dimensional
Gaussian variable with mean $m$ and covariance matrix $\Sigma$.}. \olivier{We note $\Theta$ the $N$-dimensional random vector of the thresholds:
\[
\Theta=\,^t(\theta_{-n},\cdots,\theta_n)
\]
}

The $B^j_t$s represent random fluctuations of the membrane
potential of neuron $j$. They are independent {\em random
  processes} with the same law. We assume that at each time instant
$t$, the $B^j_t$s are i.i.d. random variables distributed as
$\mathcal{N}_1(0,\sigma^2)$. They are also independent of the synaptic
weights and the thresholds.

The $J_{ij}^{\olivier{N}}$s are the synaptic weights. $J_{ij}^{\olivier{N}}$ represents the
strength with which the `presynaptic' neuron $j$ influences the
`postsynaptic' neuron $i$. They are {\em Gaussian random
variables} \olivier{whose mean is given by
\[
\Exp[J^N_{ij}]=\frac{\bar{J}}{N}
\]
}
We note $\J^{\olivier{N}}$ the $N \times N$ matrix of the synaptic weights:
\[
\J^N=(J_{ij}^{\olivier{N}})_{i,j=-n,\cdots,n}.
\]
Their covariance is assumed to satisfy the following shift invariance property,
\[
cov(J_{ij}^{\olivier{N}}J_{kl}^{\olivier{N}})=cov(J_{i+m,j+n}^{\olivier{N}}J_{k+m,l+n}^{\olivier{N}})
\]
for all indexes
$i,\,j,\,k,\,l$ $=-n,\cdots,n$ and all integers $m$ and $n$,
the indexes being taken modulo $N$. Here, and throughout this paper, $i \text{ mod }N$ is taken to lie between $-n$ and $n$. 

We stipulate the covariances through a covariance function $\Lambda: \Z^2 \to \R$ and assume that they scale as $1/N$. We write
\begin{equation}\label{eq:cov}
cov(J_{ij}^{\olivier{N}}J_{kl}^{\olivier{N}}) =
   \frac{1}{N}\Lambda\left((k-i)\text{ mod }N,(l-j)\text{ mod }N\right). 
\end{equation}

The function $\Lambda$ is even:
\begin{equation}\label{eq:Lambdasymmetries}
\Lambda(- k,- l)=\Lambda(k,l),
\end{equation}
corresponding to the simultaneous exchange of the two presynaptic and postsynaptic neurons ($cov(J_{ij}^{\olivier{N}}J_{kl}^{\olivier{N}})=cov(J_{kl}^{\olivier{N}}J_{ij}^{\olivier{N}})$!).
It is important to note that the covariance function $\Lambda$ and mean $\bar{J}$ are independent of $N$, so that these remain fixed when we asymptote $N$ to infinity later on.  

We must make further assumptions on $\Lambda$ to ensure that the system is well-behaved as the number of neurons $N$ asymptotes to infinity. We assume that the series $(\Lambda(k,l))_{k,\,l \in \Z})$ is
absolutely convergent, i.e.
\begin{equation}
\Lambda^{sum} = \sum_{k,l=-\infty}^\infty \left| \Lambda(k,l) \right| < \infty,\label{eq:lambdasumassumption}
\end{equation}
and furthermore that
\begin{equation}\label{eq:Lambdamin}
\Lambda^{\rm min}=\sum_{k,l=-\infty}^\infty \Lambda(k,l) > 0
\end{equation}
In practice one might expect there to exist a maximal correlation distance $d$ such that $\Lambda(k,l) = 0$ if $|k|+|l| > d$ (especially since in practice there is only a finite number of neurons). The existence of such a maximal correlation distance would be sufficient to guarantee the requirement \eqref{eq:lambdasumassumption}, however we refrain from explicitly making this assumption as it is not necessary \textit{per se}.

We let $\Lambda^N$ be the restriction of $\Lambda$ to $[-n,n]^2$, i.e. $\Lambda^N(i,j) = \Lambda(i,j)$ for $-n\leq i,j\leq n$. 

We next introduce the spectral properties of $\Lambda$ that are crucial for  the results in this paper. 
We use the notation that if $x$ is some quantity, $\tilde{x}$ represents its Fourier transform in a sense that depends on the particular space where $x$ is defined. For example $\tilde{\Lambda}$ is the $2\pi$ doubly periodic Fourier transform of the function $\Lambda$ whose properties are described in the next proposition. \olivier{Similarly, $\tilde{\Lambda}^N$ is the two-dimensional Discrete Fourier Transform (DFT) of the doubly periodic sequence $\Lambda^N$.}
\begin{proposition}\label{prop:lambdabehaviour}
The sum $\tilde{\Lambda}(\omega_1,\omega_2)$ of the absolutely convergent series \\
$(\Lambda(k,l) e^{-i(k \omega_1+l \omega_2)})_{k,\,l \in \Z}$ is continuous on $[-\pi,\,\pi[^2$ and positive. 
\olivier{The covariance function $\Lambda$ is recovered from the inverse Fourier transform of $\tilde{\Lambda}$:
\[
\Lambda(k,l)=\frac{1}{(2\pi)^2} \int_{-\pi}^\pi \int_{-\pi}^\pi	\tilde{\Lambda}(\omega_1,\omega_2) e^{i(k \omega_1+l \omega_2)}\,d\omega_1d\omega_2
\]
}
Moreover there exists $\tilde{\Lambda}^{\text{min}} > 0$ such that
\begin{equation}
\tilde{\Lambda}^N(0,0) \geq \tilde{\Lambda}^{\text{min}} > 0.\label{eq:Lambdabound1},
\end{equation}
for all $N$s sufficiently large.
\end{proposition}
\begin{proof}
The existence and continuity of $\tilde{\Lambda}(\omega_1,\omega_2)$ are a consequence of the absolute convergence of the series $(\Lambda(k,l))_{k,\,l \in \Z})$ \olivier{as is the inverse Fourier relation between $\Lambda$ and $\tilde{\Lambda}$.}
Being a covariance function (up to the scale factor $1/N$)
$\Lambda^N(i,j)$ must be a positive function, i.e.
\[
\tilde{\Lambda}^N(p,q) \geq 0 \ \forall p,\,q =-n,\cdots,n,
\]
where $\tilde{\Lambda}^N$ is the discrete Fourier transform (DFT) of the sequence $(\Lambda^N(k,l))_{k,\,l=-n,\cdots,n}$ (since the eigenvalues of $\Lambda^N$ are given by its DFT, see lemma \ref{lem:preserveeigenvalues}). That is,
\begin{equation}\label{eq:tildeLambdapq}
\tilde{\Lambda}^N(p,q)=\sum_{k,\,l=-n}^n \Lambda^\olivier{N}(k,l) e^{-\frac{2
    \pi i}{N} (pk+ql)}\quad p,\,q=-n,\cdots,n,
\end{equation}
with the inverse relation
\[
\Lambda^N(k,l)=\frac{1}{N^2}\sum_{p,\,q=-n}^n \tilde{\Lambda}^N(p,q) e^{\frac{2
    \pi i}{N} (pk+ql)}
\]
It is clear from the definitions that, if the series of integers $(p_N)$ and $(q_N)$ satisfy $\lim_{N \to \infty} 2\pi p_{N} / N =\omega_1$ and $\lim_{N \to \infty} 2\pi q_{N} / N = \omega_2$, then
\[
\tilde{\Lambda}^N(p_N,q_N)\to\tilde{\Lambda}(\omega_1,\omega_2),
\]
and the positivity of $\tilde{\Lambda}$ follows from that of $\tilde{\Lambda}^N$ for all $N$s.

The existence of $\tilde{\Lambda}^{\text{min}}$ and the relation \eqref{eq:Lambdabound1} follow from \eqref{eq:Lambdamin} and the convergence of $\tilde{\Lambda}^N(0,0)$ to $\tilde{\Lambda}(0,0)=\Lambda^{\rm min}$.
\end{proof}

\subsection{The laws of the uncoupled and coupled processes}
\label{subsection:laws-uncoupled-coupled}

\subsubsection{Preliminaries}

We define $\T$ (resp. $\T_T$) to be the set $\R^{[0\cdots T]}$ (resp. $\R^{[1\cdots T]}$) of finite sequences $(u_t)_{t=0,\cdots,T}$ (resp. $(u_t)_{t=1,\cdots,T}$) of length $T+1$ (resp. $T$) of real numbers.
$\T^N$ is the set of sequences $(u^{-n},\cdots,u^n)$ ($N=2n+1$) of elements of $\T$.

Similarly we note $\T^\Z$ the set of doubly infinite sequences of elements of $\T$. If $u$ is in $\T^\Z$ we note $u^i$ its ith coordinate, an element of $\T$. Hence $u=(u^i)_{i=-\infty \cdots \infty}$.

\olivier{
The shift operator
$S: \T^\Z \to \T^\Z$ is defined by
\[
(Su)^i=u^{i+1}, \quad i \in \Z.
\]
Given the element $u=(u^{-n},\ldots,u^{n})$ of $\T^N$ we form the doubly
infinite periodic sequence
\[
u(N)=(\ldots,u^{n-1},u^{n},\,u^{-n},\ldots,\,u^{n},\,u^{-n},\,u^{-n+1},\ldots)
\]
which is an element of $\T^\Z$. We have
$(u(N))^i=u^{(i \mod N)}$. 
}

We define the projection $\pi_N: \T^\Z\rightarrow\T^N$ ($N=2n+1$) to be $\pi_N(u) = (u^{-n},\ldots,u^n)$. The $N$-dimensional marginal $\mu^N$ of a measure $\mu$ in $\mmeas(\T^\Z)$ is such that $\mu^N=\mu \circ \pi_N^{-1}$.

\olivier{
We equip $\T^\Z$ with the projective topology, i.e. the topology generated by the following metric. For $u,\, v \in \T^{N}$, we define their distance $d_N(u,v)$ to be
\[
d_N(u,v) = \sup_{|j|\leq n,0\leq s\leq T}\left|u^j_s-v^j_s\right|.
\]
This allows us to define the following metric over $\T^{\Z}$, whereby if $u,v\in\T^\Z$, then
\begin{equation}
d(u,v)=\sum_{N=1}^\infty 2^{-N}(d_N(\pi_N u,\pi_N v) \wedge 1).\label{eq:Wassersteinearlier}
\end{equation}
Equiped with this topology, $\T^Z$ is Polish (a complete, separable metric space).
}

\olivier{
The metrics $d_N$ and $d$ generate, respectively, the Borelian $\sigma$-algebras $\B(\T^N)$ and $\B(\T^{\Z})$. These Borelian $\sigma$-algebras, also noted $\F_{T}^N$ and $\F_{T}$, respectively, are generated by the coordinate functions $(u^i_t)_{i=-n \cdots n,\,t=0 \cdots T}$ and $(u^i_t)_{i \in \Z,\,t=0 \cdots T}$, respectively. 
We require later on the $\sigma$-algebras $\F_t$, $t=0 \cdots T$ generated by the coordinate functions $(u^i_s)_{i \in \Z,\,s=0 \cdots t}$, and $\F_{1,t}$, $t=1 \cdots T$ generated by the coordinate functions $(u^i_s)_{i \in \Z,\,s=1 \cdots t}$. Similar definitions apply to $\F_t^N$ and $\F_{1,t}^N$.
}

\olivier{
We note $\M_1^+(\T^\Z)$  (resp. $\M_1^+(\T^N)$) the set of probability measures on $(\T^\Z,\F_T)$ (resp. $(\T^N,\F^N_T)$).
The marginal  $\mu_t$, $t=0 \cdots T$  (respectively $\mu_{1,t}$, $t=1 \cdots T$) of a probability  measure  $\mu \in \M_1^+(\T^\Z)$ with respect to the variables  $(u^i_r)_{i \in \Z,\,r=0,\cdots,t}$ (respectively $(u^i_r)_{i \in \Z,\,r=1,\cdots,t}$) is the restriction of  $\mu$ to  $\F_t$ (respectively to $\F_{1,t}$). Similar definitions apply to 
$\mu_t^N$ and $\mu_{1,t}^N$ .
}

A strictly stationary measure $\mu$ on $(\T^\Z,\F_T)$ satisfies
\[
\mu(S(B))=\mu(B) \quad \forall B \in \B(\T^{\Z}).
\]
We note $\mmeas(\T^\Z)$ the set of strictly stationary probability measures on $\T^\Z$.

\olivier{
We now introduce the following empirical measure. Given an
element $(u^{-n},\ldots,u^{n})$ in $\T^N$ we associate with it the
measure, noted $\hat{\mu}^N(u^{-n},\ldots,u^{n})$, in
$\mathcal{M}^+_{1,s}(\mathcal{T}^\Z)$ defined by
\begin{equation}\label{eq:hatmuN}
\hat{\mu}^N: \T^N \to \mathcal{M}^+_{1,s}(\mathcal{T}^\Z) \quad \text{such
  that} \quad d\hat{\mu}^N(u^{-n},\cdots,u^{n})(y)=\frac{1}{N} \sum_{i=-n}^{n}
\delta_{S^i u(N)}(y).
\end{equation}
}

We next equip $\meas(\T^\Z)$ with the topology of weak convergence, as follows. This can be defined in many ways, but the following definition is the most convenient for our paper. For $\mu^N,\nu^N \in \mM_{1}^+(\T^N)$, we note the Wasserstein distance induced by the metric $k_f d_N(u,v) \wedge 1$,
\begin{equation}
D_N(\mu^N,\nu^N) = \inf_{\mathcal{L}\in\mathcal{J}} \left\lbrace E^{\mathcal{L}}(k_f d_N(u,v) \wedge 1) \right\rbrace,\label{eq:Wasserstein}
\end{equation}
where $k_f$ is a positive constant defined at the start of section \ref{sect:model} and $\mathcal{J}$ is the set of all measures in
$\mathcal{M}_1^+(\T^{N} \times \T^N)$  with $N$-dimensional marginals $\mu^N$ and $\nu^N$. For $\mu,\nu\in\mM_{1}^+(\T^Z)$, we define\begin{equation}
D(\mu,\nu) = 2\sum_{n=0}^\infty \kappa_n D_N(\mu^N,\nu^N),\label{eq:metricdefn}
\end{equation}
where $N = 2n+1$. Here $\kappa_n = \text{max}(\lambda_n,2^{-N})$ and $\lambda_n = \sum_{k=-\infty}^{\infty} |\Lambda(k,n)|$. We note that this metric is well-defined because $D_N(\mu^N,\nu^N) \leq 1$ and \newline$\sum_{n=0}^\infty\kappa_n < \infty$. It can be shown that $\mM_1^+(\T^\Z)$ is Polish. The topology corresponding to this metric generates a Borelian $\sigma$-algebra which we denote by $\mathcal{B}(\mathcal{M}_1^+(\mathcal{T}^{\Z}))$. The Borelian $\sigma$-algebra on the set of stationary probability measures is denoted by $\mathcal{B}(\mathcal{M}_{1,s}^+(\mathcal{T}^{\Z}))$.

\subsubsection{Coupled and uncoupled processes}
We specify the initial conditions for \eqref{eq:U} as $N$ i.i.d. random
variables \newline$(U^j_0)_{j=-n,\cdots,n}$. Let $\mu_I$ be the individual
law on $\R$ of $U^j_0$; it follows that the joint law of the variables is
$\mu_I^{\otimes N}$ on $\R^N$. 
We note $P$ the law of the solution to one of the
uncoupled equations \eqref{eq:U} where we take $\theta_j$
deterministic and equal to $\bar{\theta}$ and $J_{ij}^{\olivier{N}}=0$,
$i,\,j=-n,\cdots,n$. $P$ is the law of the solution to the following
stochastic difference equation:
\begin{equation}\label{eq:Uuncoupled}
U_{t}=\gamma U_{t-1}+\bar{\theta}+B_{t-1},\,t=1,\cdots,T
\end{equation}
the law of the initial condition being $\mu_{\,I}$. This process
can be characterized exactly, as follows.

Let $\Psi:\T \rightarrow \T$ be the continuous bijection
\begin{equation}
\olivier{u=\,^t(u_0,\cdots,u_T) \to} \Psi(u) = {}^t(v_0,v_1,\ldots,v_T),\label{eq:Psidefn}
\end{equation}
where $v_0 = u_0$ and for $1\leq s\leq T$,
\begin{equation}\label{eq:Psi}
v_s  =  \Psi_s(u)=u_s - \gamma u_{s-1}-\bar{\theta}\quad s=1,\cdots, T.
\end{equation}
The following proposition is evident from equations  \eqref{eq:Uuncoupled} and \eqref{eq:Psi}.
\begin{proposition}\label{prop:psi}
The law $P$ of the solution to \eqref{eq:Uuncoupled} writes
\[
P=(\mN_{T}(0_{T},\sigma^2{\rm Id}_{T}) \otimes \mu_I) \circ \Psi ,
\]
where $0_{T}$ is the $T$-dimensional vector of coordinates equal
to 0 and ${\rm Id}_{T}$ is the $T$-dimensional identity matrix.
\end{proposition}
We later employ the convention that if $u = (u^{-n},\ldots,u^n) \in \T^N$  then $\Psi(u) = (\Psi(u^{-n}),\ldots,\Psi(u^n))$. \olivier{A similar convention applies if $u \in \T^\Z$. We also use the notation $\Psi_{1,T}$ for the mapping $\T \to \T^T$ such that $\Psi_{1,T,s}(u)=\Psi_s(u)$, $s=1\cdots T$.}

We reintroduce the coupling between the neurons, 
we note $Q^N(\J,\Theta)$ the \olivier{element of $\M_1^+(\T^N)$ which is the law of the solution to \eqref{eq:U} conditioned on $(\J,\Theta)$. We let $Q^N = \mathbb{E}^{\J,\Theta}[Q^N(\J,\Theta)]$ be the law averaged with respect to the weights and thresholds.} 

We define $\check{\M}_{1}^+(\T^N)\subset \M_{1}^+(\T^N)$ such that all $\mu^N \in \check{\M}_{1}^+(\T^N)$ have the following property. If $(v^{-n},\ldots,v^n)$ are random variables governed by $\mu^N$, then for all $|m| \leq n$, $(v^{m-n},\ldots,v^{m+n})$ has the same law as $(v^{-n},\ldots,v^n)$ (recall that the indexing is taken modulo $N$). We may thus infer that
\begin{lemma}\label{lemma:PQdag}
$P^{\otimes N}$, $Q^N$ and $(\hat{\mu}^N)^N$ (the $N^{th}$ marginal of $\hat{\mu}^N$) are in $\check{\M}_{1}^+(\T^N)$.
\end{lemma}

Since the application $\Psi$ defined in \eqref{eq:Psidefn} and \eqref{eq:Psi} plays a central role in the sequel we introduce the following definition.
\begin{definition}\label{def:mubarbar}
For each measure $\mu \in \M_1^+(\T^N)$ or $\mmeas(\T^\Z)$ we define $\undt{\mu}$ to be $\mu \circ \Psi^{-1}$. 
\end{definition}
In particular, note that
\begin{equation}\label{eq:Punder}
\undt{P}=\mN_{T}(0_{T},\sigma^2{\rm Id}_{T}) \otimes \mu_I
\end{equation}

Finally we introduce the image laws in terms of which the principal results of this paper are formulated.
\begin{definition}\label{def:PiNQN}\ \\
\begin{enumerate}
\item 
Let $\Pi^N$ be the image law of $Q^N$ through the function
$\hat{\mu}^N: \T^N \to \mM_{1,s}^+(\T^\Z)$ defined by \eqref{eq:hatmuN}.
\item
We similarly define $R^N$ to be the image law of $P^{\otimes N}$ under
$\hat{\mu}^N$.
\end{enumerate}
\end{definition}
That is, $\forall B \in \B(\mathcal{M}_{1,s}^+(\mathcal{T}^\Z))$,
\[
 \quad
\Pi^N(B)=Q^N(\hat{\mu}^{N} \in B) \quad \text{and} \quad
R^N(B)=P^{\otimes N}(\hat{\mu}^{N} \in B).
\]

The principal result of this paper is in the next theorem.
\begin{theorem}\label{theo:LDP}
$\Pi^N$ is governed by a large deviation principle with a good rate
function $H$ (to be defined in definition \ref{defn:H}). That is, if $F$ is a closed set in $\mM_{1,s}^+(\T^\Z)$, then
\begin{equation}
\lsup{N} N^{-1}\log \Pi^N(F) \leq - \inf_{\mu\in F}H(\mu).\label{eq:PiNclosed}
\end{equation}
Conversely, for all open sets $O$ in $\mM_{1,s}^+(\T^\Z)$,
\begin{equation}
\linf{N} N^{-1}\log \Pi^N(O) \geq - \inf_{\mu\in O}H(\mu).\label{eq:PiNopen}
\end{equation}
\end{theorem}
By `good rate function', we mean that $H$ is not identically $\infty$ and the sub-level sets
\[
\lbrace \mu\in \mM_{1,s}^+(\T^\Z): H(\mu) \leq c\rbrace,
\]
where $c\geq 0$, are compact.

Our proof of theorem \ref{theo:LDP} will occur in several steps. We prove in sections \ref{sect:lowerboundopensets} and \ref{section:upperboundcompact} that $\Pi^N$ satisfies a weak LDP, i.e. that it satisfies \eqref{eq:PiNclosed} when $F$ is compact and \eqref{eq:PiNopen} for all open $O$. We also prove in section \ref{sect:exponentialtightness} that $\lbrace \Pi^N\rbrace$ is exponentially tight, and we prove in section \ref{sect:Hgoodratefunction} that $H$ is a good rate function. It directly follows from these results that $\Pi^N$ satisfies a strong LDP with good rate function $H$ \cite{dembo-zeitouni:97}. Finally, in section \ref{sect:Hminimum} we prove that $H$ has a unique minimum which $\hat{\mu}^N$ converges to weakly as $N\rightarrow\infty$.
\section{The good rate function}\label{sect:PiN}
In the sections to follow we will obtain an LDP for the process with
correlations ($\Pi^N$) via the (simpler) process without correlations
($R^{N}$). However to do this we require an expression for the Radon-Nikodym derivative of $\Pi^N$ with respect to $R^{N}$, which is the main result of this section. The derivative will be expressed in terms of a function $\Gamma: \mathcal{M}_{1,s}^+(\T^{N})\to \mathbb{R}$. We will firstly define $\Gamma(\mu^N)$, demonstrating that it may be expressed in terms of a Gaussian process $G^{\mu^N}$ (to be defined below), and then use this to determine the Radon-Nikodym derivative of $\Pi^N$ with respect to $R^{N}$.
\begin{definition}\label{defn:DFT}
\olivier{
Let $p$ be a positive integer.
For $v=(v^j)_{j=-n\cdots n} \in (\R^p)^N$, we note  $\mathcal{H}^N_p(v) = v_\dag=(v_\dag^{-n},\ldots,v_\dag^n) \in (\R^p)^N$, where $v_\dag$ is defined from the Discrete Fourier Transform $\tilde{v}=(\tilde{v}^{-n},\cdots,\tilde{v}^n)$ of $v$ as follows
\begin{equation}\label{eq:DFTdefn}
\tilde{v}^k = \sum_{j=-n}^n v^j \exp\left(-\frac{2\pi ijk}{N}\right). 
\end{equation}
The inverse transform is given by $v^j = \frac{1}{N}\sum_{k=-n}^n \tilde{v}^k\exp\left(\frac{2\pi ijk}{N}\right)$.
}

\olivier{
Because $v$ is in $(\R^p)^N$ the real part of its DFT is even ($\Re(\tilde{v}^{-k})=\Re(\tilde{v}^{k})$, $k=-n,\cdots,n$) and similarly its imaginary part is odd. As a consequence we define
\begin{equation}\label{eq:wtildev}
v_\dag^k=\left\{
\begin{array}{ll}
\sqrt{2}\Im(\tilde{v}^k) & k=-n,\cdots,-1\\
\sqrt{2}\Re(\tilde{v}^k) & k=0,\cdots,n
\end{array}
\right.
\end{equation}
It is easily verified that the mapping $v \to v_\dag=\mathcal{H}^N_p(v) $ is a bijection from $(\R^p)^N$ to itself and that
\[
\sum_{k=-n}^n \| v^k_\dag\|^2=\sum_{k=-n}^n \,^t \tilde{v}^{k\,*} \tilde{v}^{k}=N \sum_{k=-n}^n \| v^k \|^2
\]
}

\olivier{
For a probability measure $\mu\in\mathcal{M}_{1}^+((\R^p)^{N})$, we define $\mu_\dag= \mu\circ(\mathcal{H}^N_p)^{-1}$ to be the image law.
}
\end{definition}
\olivier{
For simplicity we note  $\undt{\mu}_\dag$ the measure $\undt{\mu}_{1,T} \circ(\mathcal{H}^N_T)^{-1}$ (where $\undt{\mu}$ is given in definition \ref{def:mubarbar}). We note that
\begin{equation}\label{eq:tildePdistn}
\undt{P}_\dag \simeq \mathcal{N}_T\left(0_T,N\sigma^2 {\rm Id}_T\right).
\end{equation} 
}

The following lemma from Gaussian calculus \cite{moynot:99,neveu:68} which we recall
for completeness is used several times in the sequel:
\begin{lemma}\label{lemma:gauss}
Let $Z$ be a Gaussian vector of $\R^p$ with mean $c$ and covariance
matrix $K$. If $a \in \R^p$ and $b \in \R$ is such that for all
eigenvalues $\alpha$ of $K$ the relation $\alpha b > -1$ holds, we have
\begin{multline*}
\Expt{\exp\left(^taZ-\frac{b}{2} \|Z\|^2\right)}=\\
\frac{1}{\sqrt{{\rm det}\left({\rm Id}_p+bK\right)}}\times \exp\left(^tac-\frac{b}{2}\|c\|^2+\frac{1}{2}\, ^t(a-bc)K \left({\rm Id}_p+bK\right)^{-1}(a-bc)   \right)
\end{multline*}
\end{lemma}
\olivier{The same holds for} the following elementary lemma.
\begin{lemma}\label{lem:preserveeigenvalues}
Let $B$ be a symmetric block-circulant matrix with the $(j,k)$
$T \times T$ block given by $(B^{(j-k)\mod N})$,
$j,k=-n,\cdots,n$. Let $W^{(N)}$ be the $N\times N$ Unitary matrix with elements $W^{(N)}_{jk} = \frac{1}{\sqrt{N}}\exp(\frac{2\pi ijk}{N})$, $j,k=-n,\cdots,n$. Then $B$ may be `block'-diagonalised in the follow manner (where $\otimes$ is the Kronecker Product and $^*$ the complex conjugate),
\[
B = (W^{(N)}\otimes {\rm Id}_{T})\text{diag}\left(\tilde{B}^{-n},\ldots,\tilde{B}^n\right)(W^{(N)}\otimes {\rm Id}_{T})^*.
\]
Here $\tilde{B}^j$ is a $T \times T$ Hermitian matrix and is the DFT defined in \eqref{eq:DFTdefn}. We observe also that $\lambda$ is an eigenvalue of $B$ if and only if $\lambda$
is an eigenvalue of $\tilde{B}^k$ for some $k$. 
\end{lemma}
In this section we introduce a certain stationary Gaussian process attached to an element of $\mathcal{M}_{1,s}^+(\T^\Z)$.
\subsection{Gaussian process}\label{subsub:infinite}

Given $\mu$ in $\mathcal{M}_{1,s}^+(\T^\Z)$ we define a stationary Gaussian process $G^\mu$
with values in $(\R^T)^\Z$. 

For all $i$ the mean of $G^{\mu,i}_t$ is given by $c^\mu_t$, where
\begin{equation}
c^{\mu}_t = \bar{J} \int_{\T^\Z} f(u^i_{t-1})
d\mu(u),\, t=1,\cdots,T\, , i \in \Z,\label{eq:cmumean}
\end{equation}
the above integral being \olivier{well-defined because of the definition of $f$ and} independent of $i$ due to the stationarity of $\mu$.

We now define the covariance of $G^{\mu}$. We first define the following matrix-valued process.
\begin{definition}\label{def:Mmu}
Let $M^{\mu,k}$,
$k \in \Z$ be the $T \times T$ matrix defined by (for $s,t\in [1,T]$),
\begin{equation}\label{eq:Mmuinfinite}
M^{\mu,k}_{st} = \int_{\T^\Z} f(u^0_{s-1}) f(u^k_{t-1}) d\mu(u).
\end{equation}
\end{definition}
These matrixes satisfy 
\begin{equation}\label{eq:Mstationary}
{}^t M^{\mu, k} = M^{\mu,-k},
\end{equation}
because of the stationarity of $\mu$. Furthermore, they feature a spectral representation,
i.e. there exists a $T \times T$ matrix-valued measure
$\tilde{M}^{\mu}=(\tilde{M}^\mu)_{s,\,t=1,\cdots,T}$ with the following properties. Each
$\tilde{M}^\mu_{st}$ is a complex measure on $[-\pi,\,\pi[$ of finite
total variation and such that
\begin{equation}\label{eq:specdensinf}
M^{\mu,k}=\frac{1}{2\pi}\int_{-\pi}^{\pi} e^{ik \omega} \tilde{M}^\mu(d\omega).
\end{equation}
Relations \eqref{eq:Mstationary} and \eqref{eq:specdensinf} imply the following relations, for all Borelian sets $\mathcal{A}\subset [-\pi,\pi[$,
\begin{equation}
\tilde{M}^{\mu}(-\mathcal{A}) = {}^t\tilde{M}^{\mu}(\mathcal{A}) = \tilde{M}^{\mu}(\mathcal{A})^*, 
\end{equation}
from which we may infer that $\tilde{M}^{\mu}$ is Hermitian-valued. The spectral representation means that for all vectors $W \in \R^{T}$, $\,^t W \tilde{M}(d\omega) W$ is a
positive measure on $[-\pi,\pi[$. 

The covariance between the Gaussian vectors $G^{\mu,i}$ and $G^{\mu,i+k}$ is defined to be
\begin{equation}\label{eq:Kimuinfinite}
K^{\mu,k}=\theta^2 \delta_k \mone_{T} \,^t \mone_{T}+
\sum_{l=-\infty}^{\infty} \Lambda(k,l) M^{\mu,l},
\end{equation}
\olivier{where $\mone_{T} $ is the $T$-dimensional vector with all coordinates equal to 1.}
We note that the above summation converges for all $k \in \Z$ since the series
$(\Lambda(k,l))_{k,\,l \in \Z}$ is absolutely convergent 
and the elements of
$M^{\mu,l}$ are bounded by $\pm 1$ for all $l \in \Z$. 

The following lemma is necessary for the covariance function to be well-defined.
\begin{lemma}\label{lemma:Kmupos}
For $\mu \in \mathcal{M}_{1,s}^+(\T^\Z)$ and $k \in \Z$ we have
\begin{equation}\label{eq:Kstationary}
{}^t K^{\mu, k} = K^{\mu,-k},
\end{equation}
\end{lemma}
\begin{proof}
From \eqref{eq:Kimuinfinite} we have
\[
\,^t K^{\mu, -k}=\theta^2 \delta_k \mone_{T}\,^t\,\mone_{T}+
\sum_{m=-\infty}^{\infty} \Lambda(-k,m)\,^tM^{\mu,m}
\]
By \eqref{eq:Mstationary} and the evenness of $\Lambda$ we obtain \eqref{eq:Kstationary}.
\end{proof}
We next prove that
the sequence $(K^{\mu,k})_{k \in \Z}$ admits a spectral
representation (which in turn implies that $K^{\mu}$ is a well-defined covariance operator). 
\begin{proposition}\label{eq:Kmuspectral}
The sequence $(K^{\mu,k})_{k \in \Z}$ has spectral density
$\tilde{K}^\mu$ given by
\[
\tilde{K}^\mu(\omega)=\theta^2\mone_T\,^t \mone_T+\frac{1}{2\pi}\int_{-\pi}^{\pi}
\tilde{\Lambda}(\omega,-\gamma) \tilde{M}^\mu(d\gamma).
\]
That is, $\tilde{K}^\mu$ is \olivier{Hermitian} positive and satisfies $\tilde{K}^{\mu}(-\theta) = {}^t\tilde{K}^{\mu}(\theta)$ and
\begin{equation*}
K^{\mu,k}=\frac{1}{2\pi}\int_{-\pi}^{\pi} e^{ik \omega} \tilde{K}^\mu(\omega)d\omega.
\end{equation*}
\end{proposition}
\begin{proof}
First we prove that the  matrix function
\begin{equation}\label{eq:Ktildemu}
\tilde{K}^\mu(\omega)=\sum_{k=-\infty}^\infty K^{\mu,k} e^{-ik\omega}
\end{equation}
is well-defined on $[-\pi,\pi[$ and is equal to the expression in the statement of the proposition. Afterwards, we will prove that $\tilde{K}^\mu$ is positive.

From \eqref{eq:Kimuinfinite} we obtain that, for all $s,t\in [0,T]$,
\begin{equation}\label{eq:boundKmui}
|K^{\mu,k}_{st}| \leq T\theta^2
\delta_k+\sum_{l=-\infty}^{\infty} |\Lambda(k,l)|. 
\end{equation}
This shows that,
because by \eqref{eq:lambdasumassumption} the series $(\Lambda(k,l))_{k,\,l \in \Z}$ is absolutely
convergent, $\tilde{K}^\mu(\omega)$ is well-defined on $[-\pi,\pi[$. \olivier{The fact that $\tilde{K}^\mu(\omega)$ is Hermitian follows from \eqref{eq:Ktildemu} and \eqref{eq:Kstationary}.}

Combining \eqref{eq:specdensinf} \eqref{eq:Kimuinfinite} and \eqref{eq:Ktildemu} we write
\[
\tilde{K}^\mu(\omega)=\theta^2\mone_T\,^t \mone_T+
\frac{1}{2\pi}\int_{-\pi}^{\pi} \left( \sum_{m=-\infty}^{\infty}
  \sum_{k=-\infty}^{\infty}\Lambda(k,m) e^{-i(k\omega-m\gamma)}\right) \tilde{M}^\mu(d\gamma).
\]
This can be rewritten in terms of the spectral density
$\tilde{\Lambda}$ of $\Lambda$
\[
\tilde{K}^\mu(\omega)=\theta^2\mone_T\,^t \mone_T+\frac{1}{2\pi}\int_{-\pi}^{\pi} \tilde{\Lambda}(\omega,-\gamma) \tilde{M}^\mu(d\gamma).
\]
We note that $\tilde{K}^\mu(\omega)$ is positive, because for all vectors $W$ of $\R^{T}$,
\begin{equation}\label{eq:WtildeKW}
\,^t W \tilde{K}^\mu(\omega) W= \theta^2 \langle \mone_T, W \rangle^2+\frac{1}{2\pi}\int_{-\pi}^{\pi} 
\tilde{\Lambda}(\omega,-\gamma) \left( \,^t W \tilde{M}^\mu(d\gamma) W \right),
\end{equation}
the spectral density $\tilde{\Lambda}$ is
positive and the measure $\,^t W \tilde{M}^\mu(d\gamma) W$ is positive. The identity $\tilde{K}^{\mu}(-\theta) = {}^t\tilde{K}^{\mu}(\theta)$ follows from the previous lemma.
\end{proof}
If we consider the $N$th order marginal $\mu^N$ of $\mu$ we define the $N$-dimensional Gaussian process $G^{\mu^N}$ with values in $(\R^T)^N$. The mean of $G^{\mu^N,i}$, $i=-n,\cdots,n$ is given by \eqref{eq:cmumean} and the covariance between $G^{\mu^N,i}$ and $G^{\mu^N,i+k}$ is given by 
\begin{equation}\label{eq:Kimu}
K^{\mu^N,k}=\theta^2 \delta_k \mone_{T}\,^t\,\mone_{T}+
\sum_{m=-n}^{n} \Lambda(k,m)M^{\mu^N,m},
\end{equation}
for $k=-n,\cdots,n$, where the matrixes $M^{\mu^N,k}$ are defined by the finite dimensional analog of \eqref{eq:Mmuinfinite}
\begin{equation}\label{eq:Mmu}
M^{\mu^N,k}_{st} = \int_{\T^N} f(y^0_{s-1}) f(y^k_{t-1}) \mu^N(dy).
\end{equation}
Equations \eqref{eq:Mstationary} (resp. \eqref{eq:Kstationary}) hold for the matrixes $M^{\mu^N,k}$ (resp, $K^{\mu^N,k}$) for $k=-n\cdots n$. These finite sequences also have \olivier{Hermitian positive} spectral representations $\tilde{M}^{\mu^N,l}$ (resp.  $\tilde{K}^{\mu^N,l}$) for $l=-n\cdots n$ which are obtained by taking DFTs.

The finite-dimensional system `converges' to the infinite-dimensional system in the following sense. In what follows, we use the
Frobenius norm on the $T$-dimensional matrices. We write $\tilde{K}^{\mu^N}(\omega) = \sum_{k=-n}^n K^{\mu^N,k}\exp(-ik\omega)$. Note that for $|j|\leq n$, $\tilde{K}^{\mu^N}(2\pi j/N) = \tilde{K}^{\mu^N,j}$. The lemma below follows directly from the absolute convergence of $\sum_{j,k}|\Lambda(j,k)|$.
\begin{lemma}\label{lem:muNconvergenceK}
Fix $\mu\in\mM_{1,s}^+(\T^{\Z})$. For all $\varepsilon$, there exists an $N$ such that for all $M>N$ and all $j$ such that $2|j|+1 \leq M$, $\|K^{\mu^M,j}- K^{\mu,j}\|<\varepsilon$ and for all $\omega\in [-\pi,\pi[$, $\|\tilde{K}^{\mu^M}(\omega) -\tilde{K}^{\mu}(\omega)\| \leq \varepsilon$.
\end{lemma}
\begin{lemma}\label{eq:rhoK}
The eigenvalues of $\tilde{K}^{\mu^N,l}$ and $\tilde{K}^\mu(\omega)$ are upperbounded by
\begin{equation*}
\rho_K \overset{{\rm def}}{\equiv} T
\left(\theta^2+\Lambda^{sum}\right),
\end{equation*}
where $\Lambda^{sum}$ is defined in \eqref{eq:lambdasumassumption}.
\end{lemma}
\begin{proof}
Let $W\in \R^T$. We find from \eqref{eq:WtildeKW}, \eqref{eq:tildeLambdapq}, and \eqref{eq:lambdasumassumption} that
\begin{align*}
\,^t W \tilde{K}^{\mu}(\omega) W &\leq \theta^2 T\norm{W}^2+ \frac{\Lambda^{sum}}{2\pi}\int_{-\pi}^\pi \,^t W \tilde{M}^{\mu} (d\gamma)W \\
&=  \theta^2 T\norm{W}^2 + \Lambda^{sum}\,{}^t W M^{\mu,0}W.
\end{align*}
The eigenvalues of $M^{\mu,0}$ are all positive (since it is a correlation matrix), which means that each eigenvalue is upperbounded by the trace, which in turn is upperbounded by $T$. The proof in the finite dimensional case follows similarly.
\end{proof}
We note $K^{\mu^N}$ the $(NT \times NT)$ covariance matrix of the sequence of Gaussian random variables $(G^{\mu^N,-n},\cdots,G^{\mu^N,n})$. 

We let $A^{\mu^N}=K^{\mu^N}(\sigma^2 {\rm
  Id}_{NT}+K^{\mu^N})^{-1}$. This is well-defined because $K^{\mu^N}$ is \olivier{diagonalizable (being symmetric and real) and has positive eigenvalues (being a covariance matrix)}. It follows from lemma \ref{lem:preserveeigenvalues} that this is a symmetric block circulant matrix, with blocks $A^{\mu^N,k}$ ($k=-n,\cdots,n$) such that
 \[
 A^{\mu^N,-k}=\,^tA^{\mu^N,k}
 \]
and that the matrixes
\begin{equation}\label{eq:Atildel}
\tilde{A}^{\mu^N,l}=\sum_{k=-n}^n A^{\mu^N,k} e^{-\frac{2\pi i kl}{N}}=\tilde{K}^{\mu^N,l}(\sigma^2 {\rm
  Id}_T+\tilde{K}^{\mu^N,l})^{-1}.
\end{equation}
are Hermitian positive.

In the limit $N \to \infty$ we may define 
\[
\tilde{A}^{\mu}(\omega)=\tilde{K}^{\mu}(\omega)(\sigma^2 {\rm
  Id}_T+\tilde{K}^{\mu}(\omega))^{-1}
\]
as the product of two \olivier{matrix-valued} functions defined on $[-\pi,\pi[$ whose Fourier series
are absolutely convergent. The Fourier series of $(\sigma^2 {\rm
  Id}_T+\tilde{K}^{\mu}(\omega))^{-1} $ is absolutely convergent as a
consequence of Wiener's theorem because the eigenvalues of $\sigma^2 {\rm
  Id}_T+\tilde{K}^{\mu}(\omega)$ are strictly positive. Hence the Fourier series of
$\tilde{A}^{\mu}(\omega)$, i.e. $(A^{\mu,k})_{k \in \Z}$, is absolutely
convergent. We thus find that, for $l\in\Z$,
\begin{equation}\label{eq:Amulimit}
A^{\mu,l}=\frac{1}{2\pi}\int_{-\pi}^{\pi}\tilde{A}^{\mu}(\omega) e^{il\omega}d\omega = \lim_{N\rightarrow\infty}A^{\mu^N,l},
\end{equation}
and
\[
\tilde{A}^{\mu}(\omega) = \sum_{l=-\infty}^{\infty}A^{\mu,l}e^{-il\omega}.
\]
Let $\tilde{A}^{\mu^N}(\omega) = \sum_{k=-n}^n A^{\mu^N,k}\exp(-ik\omega)$ and note that for $|j| \leq n$,\newline $\tilde{A}^{\mu^N}(2\pi j/N) = \tilde{A}^{\mu^N,j}$.
\begin{lemma}\label{lemma:Lipschitz}
The map $B\rightarrow B(\sigma^2 {\rm Id}_T + B)^{-1}$ is Lipschitz continuous over
the set $\varDelta = \lbrace
\tilde{K}^{\mu^N}(\omega),\tilde{K}^{\mu}(\omega):\mu\in\M_{1,s}^+(\T^{\Z}),N>0,\omega\in
[-\pi,\pi[\rbrace$. 
\end{lemma}
\begin{proof}
The eigenvalues $\lambda$ of the matrixes in $\varDelta$ satisfy $0\leq\lambda\leq \rho_K$. Thus, both $B$ and $(\sigma^2 {\rm Id}_T + B)^{-1}$ are bounded in the operator norm (which is equal to the largest eigenvalue) for all $B\in \varDelta$. They are thus bounded over every matrix norm (as the matrix norms are all equivalent). The first term is clearly Lipschitz, and the second term is also Lipschitz because
\begin{multline*}
(\sigma^2 {\rm Id}_T + B_1)^{-1} - (\sigma^2 {\rm Id}_T + B_2)^{-1}\\= (\sigma^2 {\rm Id}_T + B_1)^{-1}\left(B_2 - B_1\right) (\sigma^2 {\rm Id}_T + B_2)^{-1}.
\end{multline*}
\end{proof}
The following lemma is a consequence of lemmas \ref{lem:muNconvergenceK} and \ref{lemma:Lipschitz}.
\begin{lemma}\label{lem:muNconvergenceA}
Fix $\mu\in\mM_{1,s}^+(\T^{\Z})$. For all $\varepsilon$, there exists an $N$ such that for all $M>N$ and all $\omega\in [-\pi,\pi[$, $\|\tilde{A}^{\mu^M}(\omega) -\tilde{A}^{\mu}(\omega)\| \leq \varepsilon$.
\end{lemma}
The above-defined matrices have the following `uniform convergence' properties.
\begin{proposition}\label{prop:Ktildeomega}
Fix $\nu \in \mM_{1,s}^+(\T^\Z)$. For all $\varepsilon > 0$, there exists an open neighbourhood $V_{\varepsilon}(\nu)$ such that for all $\mu\in V_{\varepsilon}(\nu)$, all $s,t\in [\olivier{1},T]$ and all $\omega \in [-\pi,\pi[$,
\begin{equation}\label{eq:Mbnd1}
\left|\tilde{K}^{\nu}_{st}(\omega) - \tilde{K}^{\mu}_{st}(\omega)\right| \leq \varepsilon,
\end{equation}
\begin{equation}\label{eq:Mbnd2}
\left|\tilde{A}^{\nu}_{st}(\omega) - \tilde{A}_{st}^{\mu}(\omega)\right| \leq \varepsilon,
\end{equation}
\begin{equation}\label{eq:Mbnd3}
\left|c^\nu_s - c^\mu_s \right| \leq \varepsilon,
\end{equation}
and for all $N > 0$, and for all $k$ such that $|k| \leq n$,
\begin{equation}\label{eq:Mbnd4}
\left|\tilde{K}^{\nu^N,k}_{st} - \tilde{K}^{\mu^N,k}_{st}\right| \leq \varepsilon,
\end{equation}
and
\begin{equation}\label{eq:Mbnd5}
\left|\tilde{A}_{st}^{\nu^N,k} - \tilde{A}^{\mu^N,k}_{st}\right| \leq \varepsilon.
\end{equation}
\end{proposition}
\begin{proof}
Let $\mu$ be in $\mathcal{M}_{1,s}^+(\T^\Z)$ and $\omega \in [-\pi,\pi[$. We have
\[
\tilde{K}^\mu_{st}(\omega)-\tilde{K}_{st}^\nu(\omega)=\sum_{k=-\infty}^\infty
(K^{\mu,k}_{st}-K^{\nu,k}_{st}) e^{-ik\omega}.
\]
Using \eqref{eq:Kimuinfinite} we have
\[
\tilde{K}^\mu_{st}(\omega)-\tilde{K}^\nu_{st}(\omega)=\sum_{k,l = -\infty}^\infty \Lambda(k,l)(M_{st}^{\mu,l}-M_{st}^{\nu,l}) e^{-ik\omega},
\]
hence
\begin{multline*}
\left|\tilde{K}_{st}^\mu(\omega)-\tilde{K}_{st}^\nu(\omega)\right| \leq \\
\sum_{k,l = -\infty}^\infty
|\Lambda(k,l)| \olivier{\inf_{\mathcal{L}^{2L}}}\int_{\T^{L} \times \T^L} \left| f(u^0_{s-1}) f(u^l_{t-1}) - f(v^0_{s-1})f(v^l_{t-1})\right|\mathcal{L}^{2L}(du,dv),
\end{multline*}
where $L=2|l|+1$ and $\mathcal{L}^{2L}$ has marginals $\mu^L$ and $\nu^L$. Since $|f(u^0_{s-1}) f(u^l_{t-1}) - f(v^0_{s-1}) f(v^l_{t-1})|\leq 2 (k_f d_L(\pi_L u,\pi_L v) \wedge 1)$, \olivier{where $k_f$ is the Lipschitz constant of the function $f$}, we find (through \eqref{eq:metricdefn}) that
\[
\left|\tilde{K}^{\mu}_{st}(\omega)-\tilde{K}^{\nu}_{st}(\omega)\right| \leq 2 D(\mu,\nu).
\]
Thus for \eqref{eq:Mbnd1} to be satisfied, it suffices for us to stipulate that $V_\varepsilon(\nu)$ is a ball of radius less than $\frac{\varepsilon}{2}$ (with respect to the distance metric in \eqref{eq:metricdefn}). Similar reasoning dictates that \eqref{eq:Mbnd4} is satisfied too. 

However in light of lemma \ref{lemma:Lipschitz}, it is evident that we may take the radius of $V_\varepsilon(\nu)$ to be sufficiently small that \eqref{eq:Mbnd1}, \eqref{eq:Mbnd4} and \eqref{eq:Mbnd5} are satisfied. In fact \eqref{eq:Mbnd2} is also satisfied, as it may be obtained by taking the limit as $N\rightarrow\infty$ of \eqref{eq:Mbnd5}. Since $c^\mu$ is determined by the one-dimensional marginal of $\mu$, it follows from the definition of the metric in \eqref{eq:metricdefn} that we may take the radius of $V_\varepsilon(\nu)$ to be sufficiently small that \eqref{eq:Mbnd3} is satisfied too.
\end{proof}
A direct consequence of the above proposition is that $c^\mu,\tilde{K}^{\mu^N},\tilde{K}^{\mu},\tilde{A}^{\mu^N}$ and $\tilde{A}^{\mu}$ are continuous with respect to $\mu$.

Before we close this section we define a subset of
$\mathcal{M}_{1,s}^+(\T^\Z)$ which appears naturally.
\begin{definition}\label{def:E2}
Let  $\mathcal{E}_2$ be the subset  of $\mathcal{M}_{1,s}^+(\T^\Z)$ by
\[
\mathcal{E}_2=\{ \mu \in \mathcal{M}_{1,s}^+(\T^\Z)\, | \,
\Exp^{\underline{\mu}_{1,T}}[\|v^0\|^2] < \infty\}.
\]
\end{definition}
For this set of measures, we may define the stationary process $(v^k)_{k \in \Z}$ in $\T^{\Z}_T$, where $v^k_s = \Psi_s(u^k)$, $s=1,\cdots,T$. This has a finite mean $\Exp^{\undt{\mu}_{1,T}}[v^0]$, noted $\bar{v}^\mu$. It admits the following spectral density measure,
noted $\tilde{v}^\mu$, such that
\begin{equation}\label{eq:tildemuV}
\Exp^{\undt{\mu}_{1,T}}[v^0\,^t
v^k]=\frac{1}{2\pi}\int_{-\pi}^\pi e^{ik\omega} \, \tilde{v}^\mu(d\omega).
\end{equation}
We similarly define $\mathcal{E}_2^{(N)}$ to be the subset of $\mathcal{M}_{1}^+(\T^N)$ such that for all $\mu^N$ in this subset and for all $|j|<n$,
\[
\Exp^{\undt{\mu}^N_{1,T}}[\norm{v^j}^2] < \infty 
\]
and note that if $\mu\in\mathcal{E}_2$ then $\mu^N\in\mathcal{E}_2^{(N)}$. \olivier{Also note that $\hat{\mu}^N \in \mathcal{E}_2$.}

\subsection{Definition of the functional \protect $\Gamma$}

In this section we define and study a functional $\Gamma = \Gamma_1 + \Gamma_2$, which will be used to characterise the Radon-Nikodym derivative of $\Pi^N$ with respect to $R^N$. Let $\mu\in\mathcal{M}^+_{1,s}(\mathcal{T}^{\Z})$, and let $(\mu^N)_{N \geq 1}$ be the
$N$-dimensional marginals of $\mu$ (for $N=2n+1$ odd). 
\subsubsection{\protect $\Gamma_1$}\label{Sect:Gamma1Properties}
We define
\begin{equation}\label{eq:Gamma1}
\Gamma_1(\mu^N)=-\frac{1}{2N}\log\left({\rm det}\left({\rm
  Id}_{NT}+\frac{1}{\sigma^2} K^{\mu^N}\right)\right).
\end{equation}
Because of lemma \ref{lemma:Kmupos} the spectrum of $K^{\mu^N}$ is
positive, that of ${\rm
  Id}_{NT}+\frac{1}{\sigma^2} K^{\mu^N}$ is strictly positive and the above expression has a
sense. Moreover, $\Gamma_1(\mu^N) \leq 0$.

We now define $\Gamma_1(\mu) =\lim_{N\rightarrow\infty}\Gamma_1(\mu^N)$. The following lemma indicates that this is well-defined.
\begin{lemma}\label{lemma:Gamma1}
When $N$ goes to infinity the limit of \eqref{eq:Gamma1} is given by
\begin{equation}\label{eq:Gamma1infinity}
\Gamma_1(\mu)=-\frac{1}{4\pi} \int_{-\pi}^{\pi} \log \left( {\rm det}\left({\rm
  Id}_{T}+\frac{1}{\sigma^2} \tilde{K}^\mu(\omega)\right)\right)\,d\omega
\end{equation}
for all $\mu \in \mathcal{M}_{1,s}^+(\T^\Z)$.
\end{lemma}
\begin{proof}
Through lemma \ref{lem:preserveeigenvalues}, we have that
\begin{equation}\label{eq:Gamma1tildeKN}
\Gamma_1(\mu^N) = -\frac{1}{2N}\sum_{l=-n}^n\log\left(\det\left({\rm Id}_{T} + \frac{1}{\sigma^2}\tilde{K}^{\mu^N}\left(\frac{2\pi l}{N}\right)\right)\right),
\end{equation}
where we recall that $\tilde{K}^{\mu^N}\left(\frac{2\pi l}{N}\right) = \tilde{K}^{\mu^N,l}$. Since, by lemma \ref{lem:muNconvergenceK}, $\tilde{K}^{\mu^N}(\omega)$ converges uniformly to $\tilde{K}^{\mu}(\omega)$, it is evident that the above expression converges to the desired result.
\end{proof}
\begin{proposition}\label{prop:Gamma1b}
$\Gamma_1$ is 
bounded below and continuous on both $\mathcal{M}_{1}^+(\T^N)$ and $\mathcal{M}_{1,s}^+(\T^\Z)$.
\end{proposition}
\begin{proof}
Applying lemma \ref{lemma:gauss} in the case of
$Z=(G^{\mu^N,-n}-c^{\mu^N},\cdots,G^{\mu^N,n}-c^{\mu^N})$, $a=0$, $b=\sigma^{-2}$,  we write
\[
\Gamma_1(\mu^N)=\frac{1}{N}\log \Expt { \exp \left(- \frac{1}{2\sigma^2}
    \sum_{k=-n}^{n} \| G^{\mu^N,k}-c^{\mu^N} \|^2 \right) }.
\]
Using Jensen's inequality we have
\[
\Gamma_1(\mu^N) \geq -\frac{1}{2N\sigma^2} \Expt { \sum_{k=-n}^{n} \|
  G^{\mu^N,k}-c^{\mu^N} \|^2 }=-\frac{1}{2\sigma^2} \Expt{\|G^{\mu^N,0}-c^{\mu^N} \|^2}.
\]
By definition of $K^{\mu^N,0}$, the righthand side is equal to
$-\frac{1}{2 \sigma^2} {\rm Trace}(K^{\mu^N,0})$. From \eqref{eq:Kimu}, we find that
\[
{\rm Trace}(K^{\mu^N,0})=T\theta^2+\sum_{m=-n}^n \Lambda(0,m) {\rm Trace}(M^{\mu^N,m}).
\]
It follows from the definition \eqref{eq:Mmu} that
\[
0 \leq |{\rm Trace}(M^{\mu^N,m})| \leq T.
\]
We obtain
\[
{\rm Trace}(K^{\mu^N,0}) \leq T\left(\theta^2+ \sum_{m=-n}^n
  \left|\Lambda(0,m)\right|\right) \leq T\left(\theta^2+\Lambda^{sum}\right)
\]
Hence
\[
\Gamma_1(\mu^N) \geq -\beta_1,
\]
where 
\begin{equation}\label{eq:Gamma1lb}
\beta_1 = \frac{T}{2\sigma^2} \left(\theta^2+ \Lambda^{sum}\right).
\end{equation}
It follows from lemma \ref{lemma:Gamma1} that  $-\beta_1$ is a lower
bound for $\Gamma_1(\mu)$ as well. 

The continuity (over both $\mathcal{M}_{1}^+(\T^N)$ and $\mathcal{M}_{1,s}^+(\T^\Z)$) follows from the expressions \eqref{eq:Gamma1} and \eqref{eq:Gamma1infinity}, continuity of the applications $\mu^N
\to \tilde{K}^{\mu^N}$ and $\mu
\to \tilde{K}^{\mu}$ (proposition \ref{prop:Ktildeomega}) and the continuity of the determinant.
\end{proof}

\subsubsection{\protect $\Gamma_2$}\label{Sect:Gamma2Properties}
We define
\begin{equation}\label{eq:Gamma2muN}
\Gamma_2(\mu^N)=\int_{\T_T^N}\phi^N(\mu,v)\undt{\mu}^N_{1,T}(dv)
\end{equation} 
where
\begin{equation}\label{eq:phiNdeftn}
\phi^N(\mu,v) = \frac{1}{2\sigma^2}\Bigg(\frac{1}{N}\sum_{j,k=-n}^{n} 
\,^t (v^j- c^{\mu})A^{\mu^N,\,k-j}
   (v^k-
c^{\mu} )+\\ \frac{2}{N}\sum_{j=-n}^n \langle c^{\mu}, v^j\rangle -
  \|c^{\mu} \|^2\Bigg).
\end{equation}
This quantity is finite in the subset $\mathcal{E}_2^N$ of
$\mM_{1}^+(\T^N)$ defined in definition \ref{def:E2}. If
$\mu^N\notin\mathcal{E}_2^N$, then we set $\Gamma_2(\mu^N) =
\infty$. Equivalently, we note that $\Gamma_2(\mu^N) = \int_{\T_T^N}\phi_\olivier{\dag}^N(\mu,\olivier{v_\dag}) \olivier{\undt{\mu}_\dag^N}(\olivier{dv_\dag)}$, where
\begin{multline}\label{eq:phiNdag}
\phi_\dag^N(\mu,\olivier{v_\dag}) = \frac{1}{2N^2\sigma^2}\sum_{l=-n}^n {}^t\tilde{v}^{l,*}\tilde{A}^{\mu^N,-l}\tilde{v}^{l} \\+\frac{1}{N\sigma^2}{}^t\tilde{v}^{0}\left({\rm Id}_{T}-\tilde{A}^{\mu^N}(0)\right)c^{\mu} \olivier{-} \frac{1}{2\sigma^2}{}^tc^{\mu}\left({\rm Id}_T-\tilde{A}^{\mu^N}(0)\right)c^{\mu},
\end{multline}
and $\tilde{v}$ is implicitly given by \eqref{eq:wtildev} as a function of $v_{\dag}$.
We have used definition \ref{defn:DFT} and the DFT diagonalisation of Lemma \ref{lem:preserveeigenvalues}. We note that, since $\tilde{A}^{\mu^N,l}$ is Hermitian \olivier{positive}, ${}^t\tilde{v}^{l,*}\tilde{A}^{\mu^N,l}\tilde{v}^{l}$ is real and positive.

We define
\begin{equation*}
\Gamma_2(\mu) = \lim_{N\rightarrow\infty}\Gamma_2(\mu^N),
\end{equation*}
where $\mu^N$ is the $N$-dimensional marginal of $\mu$. If $\mu\notin \mathcal{E}_2$ then $\mu^N\notin\mathcal{E}_2^{N}$ and $\Gamma_2(\mu)=\infty$. We assume throughout the rest of this section that $\mu\in\mM_{1,s}^+(\T^\Z)$ is in $\mathcal{E}_2$. This means that the spectral measure $\tilde{v}^\mu$ (as given in (\ref{eq:tildemuV})) exists. The following proposition indicates that $\Gamma_2(\mu)$ is well-defined.

\begin{proposition}\label{prop:E2Gamma2}
If the measure $\mu$ is in $\mathcal{E}_2$, i.e. if $\Exp^{\undt{\mu}_{1,T}}[\| v^0 \|^2] < \infty$, then $\Gamma_2(\mu)$ is finite and writes
\begin{multline*}\label{eq:Gamma2spectral}
\Gamma_2(\mu)=\frac{1}{2\sigma^2}
\left(\frac{1}{2\pi}\int_{-\pi}^{\pi} \tilde{A}^\mu(-\omega) :
  \tilde{v}^\mu(d\omega)\right.\\ \left.+\,^t c^\mu
  (\tilde{A}^\mu(0)-{\rm Id}_{T}) c^\mu+2\Exp^{\undt{\mu}_{1,T}}\left[{}^{t}
   {v^0}({\rm Id}_{T}-\tilde{A}^\mu(0))c^\mu\right]\right).
\end{multline*}
The ``:'' symbol indicates the double
contraction on the indexes. One also has\\
\begin{multline*}
\Gamma_2(\mu)=\frac{1}{2\sigma^2}\left(\lim_{n \to \infty} \sum_{k=-n}^n \int_{\T_T^\Z} \,^t
(v^0- c^{\mu}) A^{\mu,\,k} (v^{k}-c^{\mu})\,
d\undt{\mu}_\olivier{{1,T}}(v)\right.\\ \left.+2\Exp^{\undt{\mu}_\olivier{{1,T}}}[\langle c^\mu, v^0\rangle]-\|c^\mu\|^2\right).
\end{multline*}
\end{proposition}
\begin{proof}
Using \eqref{eq:tildemuV}, \eqref{eq:Gamma2muN} the stationarity of $\mu$ and the fact that $\sum_{k=-n}^n A^{\mu^N,k}=\tilde{A}^{\mu^N}(0)$, we have
\begin{multline}
\Gamma_2(\mu^N) = \frac{1}{4\pi\sigma^2}\int_{-\pi}^{\pi}\sum_{k=-n}^n \exp(ik\omega)A^{\mu^N,k}:\tilde{v}^\mu(d\omega) \\+ \frac{1}{\sigma^2}\int_{\T_T^N} \langle c^\mu, v^0\rangle -{}^t c^\mu\tilde{A}^{\mu^N}(0)v^0 d\undt{\mu}_{1,T}^N(v) + \frac{1}{2\sigma^2}{}^t c^\mu \left({\rm Id}_{T} - \tilde{A}^{\mu^N}(0)\right)c^\mu. 
\end{multline}
From the spectral representation of $A^{\mu^N}$ we find that
\begin{multline}
\Gamma_2(\mu^N) = \frac{1}{4\pi\sigma^2}\int_{-\pi}^\pi\tilde{A}^{\mu^N}(-\omega):\tilde{v}^{\mu}(d\omega) \\+ \frac{1}{\sigma^2}E^{\undt{\mu}_{1,T}}\left[{}^t {v^0}({\rm Id}_{T} - \tilde{A}^{\mu^N}(0))c^\mu\right]+\frac{1}{2\sigma^2} {}^t c^\mu \left({\rm Id}_{T} - \tilde{A}^{\mu^N}(0)\right)c^\mu.
\end{multline}
Since (according to proposition \ref{lem:muNconvergenceA}) $\tilde{A}^{\mu^N}(\omega)$ converges uniformly to $\tilde{A}^{\mu}(\omega)$ as $N\rightarrow\infty$, it follows by dominated convergence that $\Gamma_2(\mu^N)$ converges to the expression in the proposition.

The second expression for $\Gamma_2(\mu)$ follows analogously, although this time we make use of the fact that the partial sums of the Fourier Series of $\tilde{A}^{\mu}$ converge uniformly to $\tilde{A}^{\mu}$ (because the Fourier Series is absolutely convergent).
\end{proof}
We next obtain more information about the eigenvalues of the matrices
$\tilde{A}^{\mu^N,k} = \tilde{A}^{\mu^N}(\frac{2k\pi}{N})$ (where ${k=-n,\ldots,n}$) and $\tilde{A}^{\mu}(\omega)$.
\begin{lemma}\label{lem:alphabound}
There exists $\olivier{0 <}\, \alpha < 1$, such that for all $N$, $\mu$ and $\omega$,
the eigenvalues of $\tilde{A}^{\mu^N,k}$, $\tilde{A}^{\mu}(\omega)$ and $A^{\mu^N}$ are less than or equal to $\alpha$.
\end{lemma}
\begin{proof}
By lemma \ref{eq:rhoK}, the eigenvalues of $\tilde{K}^{\mu}(\omega)$ are positive and upperbounded by $\rho_K$. Since $\tilde{K}^{\mu}(\omega)$ and $\left(\sigma^2 {\rm Id}_T + \tilde{K}^{\mu}(\omega)\right)^{-1}$ are coaxial (because $\tilde{K}^{\mu}$ is Hermitian and therefore diagonalisable), we may take
\begin{equation*}
\alpha = \frac{\rho_K}{\sigma^2+\rho_K}.
\end{equation*}
This upperbound also holds for $\tilde{A}^{\mu^N,k}$, and for the eigenvalues of $A^{\mu^N}$ because of lemma \ref{lem:preserveeigenvalues}.
\end{proof}
We wish to prove that $\Gamma_2(\mu^N)$ is lower semicontinuous. A consequence of this will be that $\Gamma_2(\mu^N)$ is measureable with respect to $\mathcal{B}(\mathcal{M}_{1}(\T^N))$. In
order to do this, we must first prove that its integrand $\phi^N(\mu,v)$ possesses a lower bound. Letting $\olivier{\tilde{w}^j}= \tilde{v}^j$ for all $j$, except that $\tilde{w}^0 = \tilde{v}^0 - Nc^{\mu}$, we may write the integrand as $\olivier{\undt{\phi}_\dag}^N(\mu,\olivier{w_\dag}) = \olivier{\phi_\dag^N(\mu,v_\dag)}$, where \\
\begin{equation}\label{eq:phiw}
 \olivier{\undt{\phi}_\dag}^N(\mu,\olivier{w_\dag}) = \frac{1}{2N^2\sigma^2}\sum_{l=-n}^n {}^t \olivier{\tilde{w}^{l,*}}\tilde{A}^{\mu^N,-l}\olivier{\tilde{w}}^l + \frac{1}{N\sigma^2}\langle c^{\mu},\tilde{w}^0\rangle + \frac{1}{2\sigma^2}\norm{c^\mu}^2.
\end{equation}
Note that the correspondence between $\tilde{w}$ and $w_{\dag}$ is given by \eqref{eq:wtildev}.

Thus in order that the integrand possesses a lower bound, it suffices to prove, \olivier{since the matrixes  $\tilde{A}^{\mu^N,l}$ are Hermitian positive} that there exists a lower bound for
\begin{equation}\label{eq:Gamma2reduced}
\frac{1}{N^2} \,^t
\tilde{w}^0\tilde{A}^{\mu^N,0}\tilde{w}^{0}+\frac{2}{N} \ \langle
\tilde{w}^0 , c^{\mu^N}\rangle,
\end{equation}
We have made use of the fact that $\tilde{w}^0$ and $\tilde{A}^{\mu^N,0}$ are real (since they are each a sum of real variables). Let $\tilde{K}^{\mu^N,0}$ = $O^{\mu^N} D^{\mu^N} \,^t O^{\mu^N}$, where $D^{\mu^N}$ is diagonal and $O^{\mu^N}$ is orthonormal. We define $X = ^t O^{\mu^N}\olivier{\tilde{w}^0}$, so that (\ref{eq:Gamma2reduced}) is equal to
\begin{equation}
\frac{1}{N^2} \,^t X D^{\mu^N}(\sigma^2{\rm Id}_{T}+D^{\mu^N})^{-1}X + \frac{2}{N}\sum_{t=1}^T \langle {}^tO^{\mu^N}_t, c^{\mu^N}\rangle X_t,\label{eq:Gamma2reduced2}
\end{equation}
where $O^{\mu^N}_t$ is the $t$-th column vector of $O^{\mu^N}$.
In order that (\ref{eq:Gamma2reduced2}) is bounded below, we require that the coefficient of $X$ converges to zero when $D^{\mu^N}$ does. The following lemma is sufficient. 
\begin{lemma}\label{lem:boundOc}
For each $1 \leq t\leq T$,
\begin{equation*}
\langle c^{\mu^N}, O^{\mu^N}_t \rangle^2\leq\frac{\bar{J}^2}{\tilde{\Lambda}^{\text{min}}}D^{\mu^N}_{tt},
\end{equation*}
where $\tilde{\Lambda}^{\text{min}}$ is given in proposition \ref{prop:lambdabehaviour}.
\end{lemma}
\begin{proof}
If $\bar{J} = 0$ the conclusion is evident, thus we assume throughout this proof that $\bar{J}\neq 0$. 
Since $D_{tt}^{\mu^N} = {}^t \bar{O}^{\mu^N}_t \tilde{K}^{\mu^N,0} O^{\mu^N}_t$, we find from the definition that
\begin{equation*}
D_{tt}^{\mu^N} = \theta^2 \langle\mone_T,O^{\mu^N}_{t}\rangle^2
+\sum_{k,m=-n}^n \Lambda^N(k,m) \,^t
O^{\mu^N}_t M^{\mu^N,{\olivier m}} O^{\mu^N}_t.
\end{equation*}
We introduce the matrixes $(L^{\mu^N,k})_{ k=-n,\cdots,n}$, where for $1\leq s,t\leq T$,
\[
L^{\mu^N,k}_{st}=M^{\mu^N,k}_{st}- \bar{c}^\mu_s \bar{c}^\mu_t =\int_{\T^N}
(f(u^0_{s-1})-\bar{c}^{\mu}_{s-1}) (f(u^k_{t-1})-\bar{c}^{\mu}_{t-1})\,\mu^N(du)
\]
where $\bar{c}^{\mu}=\frac{1}{\bar{J}}c^{\mu^N}$. 

These matrices have the same properties as the matrixes $M^{\mu^N,k}$,
in particular the discrete Fourier Transform
$(\tilde{L}^{\mu^N,l})_{l=-n,\cdots,n}$ is \olivier{Hermitian} positive.
Using this spectral representation we
write
\[
D_{tt}^{\mu^N} = \theta^2 \langle \mone_T, O^{\mu^N}_{t}\rangle^2+\tilde{\Lambda}^N(0,0)\langle\bar{c}^\mu,
O^{\mu^N}_t \rangle^2
+ \frac{1}{N}\sum_{l=-n}^n \tilde{\Lambda}^N(0,-l) \,^t
O^{\mu^N}_t \tilde{L}^{\mu^N,l} O^{\mu^N}_t,
\]
and since $\tilde{\Lambda}^N(0,-l)$ is positive for all
$l=-n,\cdots ,n$ and $\,^t
O^{\mu^N}_t \tilde{L}^{\mu^N,l} O^{\mu^N}_t$ is positive for all $t=1,\cdots,T$,
we have
\[
D_{tt}^{\mu^N} \geq \frac{\tilde{\Lambda}^N(0,0)}{\bar{J}^2}\langle c^{\mu^N}, O^{\mu^N}_t\rangle^2,
\]
and the conclusion follows from assumption \eqref{eq:Lambdabound1}.\\
\end{proof}
We may use the previous lemma to obtain a lower-bound for the quadratic form (\ref{eq:Gamma2reduced2}). We recall the easily-proved identity from the calculus of quadratics that, for all $x\in\mathbb{R}$,
\begin{equation*}
ax^2 + 2bx \geq -\frac{b^2}{a}.
\end{equation*}
We therefore find, through lemma \ref{lem:boundOc}, that (\ref{eq:Gamma2reduced2}) is greater than or equal to
\begin{equation}\label{eq:betadefn}
-\frac{\bar{J}^2}{\tilde{\Lambda}^{\text{min}}}\left(T\sigma^2 +
  \sum_{t=1}^T
  D^{\mu^N}_{tt}\right)=-\frac{\bar{J}^2}{\tilde{\Lambda}^{\text{min}}}\left(T\sigma^2
  + {\rm Trace}(\tilde{K}^{\mu^N,0}))\right).
\end{equation}
Since $\tilde{K}^{\mu^N,0}=\sum_{k=-n}^n K^{\mu,k}$ and
$K^{\mu,k}=\theta^2 \delta_k \mone_T \,^t \mone_T+\sum_{m=-n}^n
\Lambda(k,m) M^{\mu,m},$ it follows that
\[
{\rm Trace}(\tilde{K}^{\mu^N,0}) \leq T\left(\theta^2+\Lambda^{sum}\right).
\]
Putting all this together we find that $\phi^N(\mu,v)$ is greater
than $-\beta_2$, where
\begin{equation}\label{eq:beta}
\beta_2=\frac{T \bar{J}^2}{2\sigma^2\tilde{\Lambda}^{\text{min}}} \left(\sigma^2 +\theta^2+\Lambda^{sum}\right).
\end{equation}
This is a `universal' constant which depends only on the model
parameters and not on the particular measure $\mu$.\\

We then have the following proposition.
\begin{proposition}\label{prop:Gamma2lsc}
$\Gamma_2(\mu^N)$ is lower-semicontinuous.
\end{proposition}
\begin{proof}
We define $\phi^{N,M}(\mu^N,v) = 1_{B_M}(v)\left(\phi^{N}(\mu^N,v) + \beta_2\right)$, where $v\in B_M$ if $ N^{-1}\sum_{j=-n}^{n} \norm{v^j}^2 \leq M$. We have just seen that $\phi^{N,M} \geq 0$. We also define
\begin{equation*}
\Gamma_2^M(\mu^N) = \int_{\T_T^{N}}\phi^{N,M}(\mu,v)\olivier{\undt{\mu}}_{1,T}^N(dv) -\beta_2.
\end{equation*}
Suppose that $\nu^N\rightarrow\mu^N$ with respect to the weak topology. Observe that
\begin{multline*}
\left|\Gamma_2^M(\mu^N) - \Gamma_2^M(\nu^N)\right|\leq \left|\int_{\T_T^{N}}\phi^{N,M}(\mu^N,v)\olivier{\undt{\mu}}^N_{1,T}(dv) - \int_{\T_T^{N}}\phi^{N,M}(\mu^N,v)\olivier{\undt{\nu}}^N_{1,T}(dv)\right| \\+ \left|\int_{\T_T^{N}}\phi^{N,M}(\mu^N,v)\olivier{\undt{\nu}}^N_{1,T}(dv) - \int_{\T_T^{N}}\phi^{N,M}(\nu^N,v)\olivier{\undt{\nu}}^N_{1,T}(dv)\right|.
\end{multline*}
We may infer from the above expression that $\Gamma_2^M(\mu^N)$ is
continuous (with respect to $\mu^N$) for the following reasons. The
first term on the right hand side converges to zero because
$\phi^{N,M}$ is continuous and bounded (with respect to $v$). The
second term converges to zero because $\phi^{N,M}(\mu^N,v)$ is a
continuous function of $\mu^N$, see proposition \ref{prop:Ktildeomega}.

Since $\Gamma_2^M(\mu^N)$ grows to $\Gamma_2(\mu^N)$ as $M\rightarrow\infty$, we may conclude that $\Gamma_2(\mu^N)$ is lower semicontinuous with respect to $\mu^N$. 
\end{proof}
We define $\Gamma(\mu^N)=\Gamma_1(\mu^N)+\Gamma_2(\mu^N)$. We may conclude from propositions \ref{prop:Gamma1b} and \ref{prop:Gamma2lsc} that $\Gamma$ is measureable. 

\subsection{The Radon-Nikodym derivative}\label{sect:average}
In this section we determine the Radon-Nikodym derivative of $\Pi^N$ with respect to $R^N$. However in order for us to do this, we must first
compute the Radon-Nikodym derivative of $Q^N$ with respect to
$P^{\otimes N}$.  We do this 
in the next proposition where, and we will use the same notation
throughout the paper, the usual inner product of two vectors $u$ and $v$ of
$\R^{T}$ is noted $\langle u, v \rangle$.
\begin{proposition}\label{prop:RNderiv1}
The Radon-Nikodym  derivative
of $Q^N$ with respect to $P^{\otimes N}$ is given by
the following expression.
\begin{multline}\label{eq:dQNdPN}
\frac{dQ^N}{dP^{\otimes N}}(u^{-n},\cdots,u^{n})=\\
\Expt{\exp\left(\frac{1}{\sigma^2} \left(\sum_{j=-n}^n \langle \Psi_{1,T}(u^j), G^j \rangle-
\frac{1}{2}\| G^j \|^2\right)\right)},
\end{multline}
the expectation being taken against the $N$ $T$-dimensional Gaussian processes $(G^i)$, $i=-n,\cdots,n$ given by
\begin{equation}
\label{equation:Gprocdef}
G^i_t  =  \sum_{j=-n}^{n} J_{ij}^{\olivier{N}} f(u^j_{t-1})+\theta_i-\bar{\theta},
\quad t=1,\cdots,T ,
\end{equation}
and the function $\Psi$ being defined by \eqref{eq:Psidefn} and \eqref{eq:Psi}.
\end{proposition}
\begin{proof}
For fixed $\olivier{(\J^{\olivier{N}}, \Theta)}$, we let $R_{J^N,\Theta} : \R^{N(T+1)} \to \R^{N(T+1)}$ be the mapping $u\rightarrow y$, i.e.
$R_{J^N,\Theta}(u^{-n},\cdots,u^{n})=(y^{-n},\cdots,y^n)$, where for $j=-n,\cdots,n$,
\[
\left\{
\begin{array}{lcl}
y^j_0 & = & u^j_0\\
y^j_t & = & u^j_t-\gamma u^j_{t-1}-G^j_t \quad t=1,\cdots,T.
\end{array}
\right.
\]
The determinant of the Jacobian of $R_{J^N,\Theta}$ is $1$ for the following reasons. Since $\frac{dy^j_s}{du^k_t} = 0$ if $t>s$, the determinant is $\prod_{s=0}^T D_s$, where $D_s$ is the Jacobian of the map $(u^{-n}_s,\ldots,u^{n}_s)\to (y^{-n}_s,\ldots,y^{n}_s)$ induced by $R_{\olivier{\J^{\olivier{N}}},\Theta}$. However $D_s$ is evidently $1$. Similar reasoning implies that $R_{\olivier{\J^{\olivier{N}}},\Theta}$ is a bijection.

It may be seen that the random vector $Y = R_{\olivier{\J^{\olivier{N}}},\Theta}(U)$ is such that $Y^j_0 = U^j_0$ and $Y^j_t  =  B^j_{t-1}+\bar{\theta}$ where $|j|\leq n$ and $t=1,\cdots,T.$ Therefore
\begin{equation*}
Y^j \simeq
\mathcal{N}_{T}(\bar{\theta} \mone_{T},\sigma^2 {\rm Id}_{T}) \otimes \mu_I,
\quad j=-n,\cdots,n.
\end{equation*}
Since the determinant of the Jacobian of $R_{J^N,\Theta}$ is one, we obtain the law of $Q^N(\J^{\olivier{N}},\Theta)$ by applying the inverse of $R_{\olivier{\J^{\olivier{N}}},\Theta}$ to the above distribution, i.e.
\begin{align*}
Q^N(\J^{\olivier{N}},\Theta)(du) = \left(2\pi\sigma^2\right)^{-\frac{NT}{2}}\exp\left(-\frac{1}{2\sigma^2}\norm{R_{\J^{\olivier{N}},\Theta}(u)}^2\right)\prod_{j=-n}^n\mu_I(du^j_0) \prod_{t=1}^T du^j_t.
\end{align*}
Recalling that $P^{\otimes N} = Q^N(0,\bar{\theta}\mone_N)$, we therefore find that
\begin{equation*}
\frac{dQ^N(\J^{\olivier{N}},\Theta)}{dP^{\otimes N}}(u) = \exp\left(-\frac{1}{2\sigma^2}\left(\norm{R_{\J^{\olivier{N}},\Theta}(u)}^2-\norm{R_{0,\bar{\theta}\mone_N}(u)}^2\right)\right).
\end{equation*}
Taking the expectation of this with respect to $\J^{\olivier{N}}$ and $\Theta$ yields the result.
\end{proof}

We now prove that the Gaussian system $(G^i_s)_{i=-n,\ldots,n,s=1,\ldots,T}$ has the same law as the system $G^{(\hat{\mu}^N)^N}$, as defined in \eqref{eq:Kimu} and afterwards.

\begin{proposition}\label{prop:Kmuhatpos}
Fix $u \in \T^N$. The covariance of the Gaussian system
$(G^i_s)$, where $i=-n,\ldots,n$ and $s=1,\ldots,T$ writes $K^{(\hat{\mu}^N(u))^N}$, where $(\hat{\mu}^N(u))^N$ is the $N$-dimensional marginal of $\hat{\mu}^N(u)$. For each $i$, the mean of $G^i$ is $c^{\hat{\mu}^N(u)}$.
\end{proposition}
\begin{proof}
The proof follows from the definition \eqref{equation:Gprocdef}.

The mean of $G^i_t$  is equal to
\[
\Expt{G^i_t}=\frac{\bar{J}}{N}\sum_{j=-n}^{n}
f(u^j_{t-1})=\bar{J}\int_{\T^{\Z}}
f(y^0_{t-1})\,d\hat{\mu}^{N}(u)(y),
\]
for $t=1,\cdots,T$. This is indeed independent of the index $i$.

Let us now examine the covariance function $K$ of these $N$ Gaussian
processes. It is an $NT \times NT$ matrix which has a block structure, each block $K^{ik}$,
$i,\,k=-n,\cdots,n$, being the $T \times T$ covariance matrix of the
two processes $G^i$ and $G^k$. We have
\[
K^{ik}_{ts}=cov(G^i_tG^k_s)=\sum_{j,\,l=-n}^{n} cov(J_{ij}^{\olivier{N}}J_{kl}^{\olivier{N}})
f(u^j_{t-1}) f(u^l_{s-1})+\theta^2 \delta_{i-k},\,s,t=1,\cdots,T.
\]
Because of our definition \eqref{eq:cov} of the covariance structure
we have
\begin{equation}
K^{ik}_{ts}=\sum_{m=-n}^{n} \Lambda(k-i,m)\left(\frac{1}{N}\sum_{j=-n}^{n} f(u^j_{t-1}) f(u^{j+m}_{s-1})\right) +\theta^2 \delta_{i-k}\label{eq:Kikintermediate}.
\end{equation}
Since $K^{ik}$ depends only on $(k-i)$, it can be seen that $K$ is a block circulant matrix, and we may write
\[
K^{ik} \overset{\rm def}{\equiv} K^{(k-i)\mod N},
\]
where we recall that $j\mod N$ lies between $\pm n$. 
It may be inferred from \eqref{eq:Kikintermediate} that 
\[
K^i_{ts}=\theta^2 \delta_i+
\sum_{m=-n}^{n} \Lambda(i,m)\int_{\T^N} f(v^0_{t-1})f(v^{m}_{s-1})\,(\hat{\mu}^N(u))^N(dv).
\]
\olivier{After a comparison of this with \eqref{eq:Kimu}, we find that $K = K^{(\hat{\mu}^N(u))^N}$.}
\end{proof}
We obtain an alternative expression for the Radon-Nikodym derivative in \eqref{eq:dQNdPN} by applying lemma \ref{lemma:gauss}. That is, we substitute \newline$Z=(G^{-n},\cdots,G^{n})$,
$a=\frac{1}{\sigma^2}({v}^{-n},\cdots,{v}^{n})$, and
$b=\frac{1}{\sigma^2}$ into the formula in lemma \ref{lemma:gauss}. After noting proposition \ref{prop:Kmuhatpos} we thus find that
\begin{proposition}\label{prop:radon-nikodym}
The Radon-Nikodym derivatives write as 
\begin{align*}
\frac{dQ^N}{dP^{\otimes N}}(u^{-n},\cdots,u^{n}) &=\exp(N\Gamma((\hat{\mu}^N(u^{-n},\cdots,u^n))^N)),\\
\frac{d\Pi^N}{dR^N}(\mu) &= \exp(N\Gamma(\mu^N)).
\end{align*}
Here $\mu\in \M_{1,s}^{+}(\T^\Z)$, 
$\Gamma(\mu)=\Gamma_1(\mu)+\Gamma_2(\mu)$
and the expressions for $\Gamma_1$ and
$\Gamma_2$ have been defined in equations \eqref{eq:Gamma1} and \eqref{eq:Gamma2muN}. 
\end{proposition}
The second expression in the above proposition follows from the first one because $\Gamma$ is measureable.
\section{The large deviation principle}\label{Sect:GoodRateFunction}
In this section we prove the principal result of this paper (Theorem \ref{theo:LDP}), that the image laws $\Pi^N$ satisfy an LDP with good rate function $H$ (to be defined below). We do this by firstly establishing an LDP for the image law with uncoupled weights ($R^N$), \olivier{see definition \ref{def:PiNQN}}, and then use the Radon-Nikodym derivative of corollary \ref{prop:radon-nikodym} to establish the full LDP for $\Pi^N$. Therefore our first task is to write the LDP governing $R^N$.

Let $\mu,\nu$ be probability measures over a Polish Space $\Omega$ with respect to the Borelian $\sigma$-algebra $\mathcal{B}(\Omega)$. The K\"ullback-Leibler
divergence of $\mu$ relative to $\nu$ is
\[
I^{(2)}(\mu, \nu)=\int_{\Omega} \log \left(\frac{d
  \mu}{d \nu}\right) d\mu
\]
if $\mu$ is absolutely continuous with respect to $\nu$,
and $I^{(2)}(\mu, \nu)=\infty$ otherwise.

Let $\nu^{\Z}$ be the infinite product measure on $\Omega^{\Z}$ induced by $\nu$. If $\mu$ is a stationary probability measure over $\Omega^{\Z}$, then the process-level entropy \olivier{of $\mu$ with respect to $\nu^\Z$} is defined to be
\begin{equation}\label{eq:I3limit}
I^{(3)}(\mu,\nu^{\Z})=\lim_{N\to \infty} \frac{1}{N} I^{(2)}(\mu^N,
\nu^{\otimes N}).
\end{equation}

$R^N$ is governed by the following large deviation principle \cite{donsker-varadhan:85,baxter-jain:93}. 
\begin{theorem}\label{theorem:Rnexptight}
If $F$ is a closed set in $\mmeas$, then
\begin{equation*}
\lsup{N} N^{-1}\log R^N(F) \leq - \inf_{\mu\in F}I^{(3)}(\mu,P^\Z),
\end{equation*}
and for all open sets $O$
\begin{equation*}
\linf{N} N^{-1}\log R^N(O) \geq - \inf_{\mu\in O}I^{(3)}(\mu,P^\Z).
\end{equation*}
Here
\begin{equation}\label{eq:I3R}
I^{(3)}\left(\mu,P^{\Z}\right) = I^{(3)}\left(\mu_0,\mu_{I}^{\Z}\right) + \int_{\R^\infty}I^{(3)}\left(\mu_{u_0},P^{\Z}_{u_0}\right)d\mu_0(u_0),
\end{equation}
where $\mu_{u_0}\in\M_{1,s}^+(\T_T^{\Z})$ is the conditional probability distribution of $\mu$ given $u_0$ \olivier{in $\R^\Z$}.
$I^{(3)}$ is a good rate function (i.e. its level sets are compact). In addition, the set of measures $\lbrace R^N\rbrace$ is exponentially tight. This means that, for all $0 \leq a < \infty$, there exists a compact set $K_a \subset M_{1,s}^+(\T^{\Z})$ such that for all $N$
\begin{equation*}
\lsup{N}N^{-1}\log R^N\left(K_a^c\right) < -a.
\end{equation*}
\end{theorem}
\begin{proof}
$R^N$ satisfies an LDP with good rate function $I^{(3)}(\mu,P^{\Z})$ \cite{ellis:85}.  In turn, a sequence of probability measures (such as $\lbrace R^N\rbrace$) over a Polish Space satisfying a large deviations upper bound with a good rate function is exponentially tight \cite{dembo-zeitouni:97}. 

It is an identity in \cite{donsker-varadhan:83} that \[I^{(2)}(\mu^N,P^{\otimes N}) = I^{(2)}\left(\mu^N_0,\mu_{I}^{\otimes N}\right) + \int_{\R^N}I^{(2)}\left(\mu_{u_0}^N,P^{\otimes N}_{u_0}\right)\mu_{0}^{N}(du_0^{-n}\cdots du_0^n).\]
We divide by $N$ and then take $N\to\infty$ to obtain \eqref{eq:I3R}.
\end{proof}
Because $\Psi$ is bijective and continuous, it may be easily shown that 
\begin{align}
I^{(2)}\left(\mu^N,P^{\otimes N}\right) &= I^{(2)}\left(\undt{\mu}^N,\undt{P}^{\otimes N}\right)\label{eq:I2equality}\\
I^{(3)}\left(\mu,P^{\Z}\right) &= I^{(3)}\left(\undt{\mu},\undt{P}^{\Z}\right).\label{eq:I3equality} 
\end{align}
Before we move to a statement of the LDP governing $\Pi^N$, we prove the following relationship between the set
$\mathcal{E}_2$ (see definition \ref{def:E2}) and the set of stationary measures which have a finite
K\"ullback-Leibler information or process level entropy with respect to $P^\Z$.
\begin{lemma}\label{lemma:E2enough}
We have
\[
\{\mu \in \mathcal{M}_{1,s}^+(\T^\Z),\, I^{(3)}(\mu,P^\Z) < \infty
\} \subset \mathcal{E}_2.
\]
\end{lemma}
\begin{proof}
Let $\mu \in \mathcal{M}_{1,s}^+(\T^\Z)$. We use the classical result
that
\begin{equation}
I^{(2)}\left(\undt{\mu}^N,\undt{P}^{\otimes N}\right)=\sup_{\varphi \in {\rm C}_b(\T^N)} \left(
  \int_{\T^N} \varphi\,d\undt{\mu}^N-\log \int_{\T^N} \exp(\varphi)\,d\undt{P}^{\otimes N}\right).\label{eq:FenchelLegendre}
\end{equation}
We let $\rho(y)=\sum_{s=1}^T\sum_{k=-n}^n (y^k_s)^2$ and $\varphi(y) = a\rho(y)$, where $a>0$. The function $\rho_M(x)=\rho(x)\mone_{\varphi(x) \leq M}$ is in
${\rm C}_b(\T^N)$, hence for all $a  > 0$
\[
a \int_{\T^N} \rho_M \,d\undt{\mu}^N \leq \log \int_{\T^N} \exp (a
\rho_M)\,d\undt{P}^{\otimes N}+I^{(2)}(\undt{\mu}^N,\undt{P}^{\otimes N}).
\]
According to proposition  \ref{prop:psi}, $\underline{P}_{1,T}\simeq\mN(0_{T},\sigma^2{\rm Id}_{T})$. Hence, as soon as $1-2a\sigma^2 >0$, 
we obtain using an easy Gaussian computation that
\[
\log \int_{\T^N} \exp (a
\rho)\,d\undt{P}^{\otimes N}=-\frac{NT}{2} \log(1-2a\sigma^2).
\]
Hence, since $\int_{\T^N} \rho \,d\undt{\mu}^N=N \Exp^{\undt{\mu}_{1,T}}[\|v^0\|^2]$, after taking $M\to\infty$ and applying the dominated convergence theorem we have
\[
a\Exp^{\undt{\mu}_{1,T}}[\|v^0\|^2] \leq -\frac{T}{2} \log(1-2a\sigma^2)+\frac{1}{N}I^{(2)}(\undt{\mu}^N,\undt{P}^{\otimes N}).
\]
By taking the limit $N \to \infty$ we obtain the result.
\end{proof}
We are now in a position to define what will be the rate function of the LDP governing $\Pi^N$.
\begin{definition}\label{defn:H}
Let $H$ be the function $\mathcal{M}_{1,s}^+(\T^\Z) \to \R \cup \{+\infty\}$ defined by
\[
H(\mu)=\left\{
\begin{array}{l}
+\infty \quad \text{if} \quad I^{(3)}(\mu,P^\Z)=\infty\\
I^{(3)}(\mu,P^\Z)-\Gamma(\mu) \quad \text{otherwise}.
\end{array}
\right.
\]
\end{definition}
Here $\Gamma(\mu)=\Gamma_1(\mu)+\Gamma_2(\mu)$
and the expressions for $\Gamma_1$ and
$\Gamma_2$ have been defined in equations \eqref{eq:Gamma1} and \eqref{eq:Gamma2muN}. Note that because of proposition \ref{prop:E2Gamma2} and lemma \ref{lemma:E2enough}, whenever
$I^{(3)}(\mu,P^\Z)$ is finite, so is $\Gamma(\mu)$. 

\subsection{Lower bound on the open sets}\label{sect:lowerboundopensets}
We prove the second half of proposition \ref{theo:LDP}.
\begin{lemma}\label{lemma:lbopen}
For all open sets $O$,
\begin{equation*}
\linf{N} N^{-1}\log\Pi^N(O) \geq -\inf_{\mu\in O}H(\mu).
\end{equation*}
\end{lemma}
\begin{proof}
From the expression for the Radon-Nikodym derivative in corollary \ref{prop:radon-nikodym} we have
\begin{equation*}
\Pi^N(O) = \int_{O}\exp\left(N\Gamma(\mu^N)\right)dR^N(\mu).
\end{equation*}
If $\mu \in O$ is such that $I^{(3)}(\mu,P^\Z) = \infty$, then $H(\mu) = \infty$ and evidently
\begin{equation}
\linf{N} N^{-1}\log\Pi^N(O) \geq -H(\mu).\label{eq:LDP1}
\end{equation}
We now prove (\ref{eq:LDP1}) for all  $\mu\in O$ such that $I^{(3)}(\mu,P^\Z) < \infty$. Let $\varepsilon > 0$ and $Z^N_\varepsilon(\mu)\subset O$ be an open
neighbourhood 
containing $\mu$ such that \newline$\inf_{\gamma\in Z^N_\varepsilon(\mu) }\Gamma(\gamma^N)\geq
\Gamma(\mu^N)-\varepsilon$. Such $\lbrace Z^N_\varepsilon(\mu) \rbrace$ exist for all $N$
because of the lower semi-continuity of $\Gamma(\mu^N)$ (see proposition \ref{prop:Gamma2lsc}) and the fact that the projection $\mu\rightarrow \mu^N$ is clearly continuous. Then
\begin{align*}
\linf{N}N^{-1}\log\Pi^N(O) &= \linf{N} N^{-1}\log\int_O\exp(N\Gamma(\gamma^N))dR^N(\gamma) \\
&\geq \linf{N} N^{-1}\log\left(R^N(Z^N_\varepsilon\olivier{(\mu)})\times \inf_{\gamma\in Z^N_\varepsilon(\mu)}\exp(N\Gamma(\gamma^N))\right) \\
&\geq -I^{(3)}(\mu,P^\Z) + \linf{N}\inf_{\gamma\in Z^N_\varepsilon(\mu)}\Gamma(\gamma^N) \\
&\geq -I^{(3)}(\mu,P^\Z) + \linf{N}\Gamma(\mu^N) - \varepsilon \\ 
&= -I^{(3)}(\mu,P^\Z) + \Gamma(\mu) - \varepsilon.
\end{align*}
The last equality follows from lemma \ref{lemma:Gamma1} and proposition \ref{prop:E2Gamma2}. Since $\varepsilon$ is arbitrary, we may take the limit as $\varepsilon\rightarrow 0$ to obtain \eqref{eq:LDP1}. Since (\ref{eq:LDP1}) is true for all $\mu\in O$ the lemma is proved.
\end{proof}

\subsection{Exponential Tightness of \protect $\Pi^N$}\label{sect:exponentialtightness}
We recall that if $\mu\in\mathcal{M}_{1,s}^+(\T^{\Z})$ but $\mu\notin\mathcal{E}_2$, then $I^{(3)}(\mu,P^{\Z}) = \Gamma(\mu) =\infty$. We begin with the following technical lemma. 
\begin{lemma}\label{lemma:expNc}
There exist positive constants $c>0$ and $a > 1$ such that, for all $N$,
\begin{equation*}
\int_{\T^N}\exp\left(aN\phi^N(\hat{\mu}^N(u),\Psi(u))\right)\,P^{\otimes N}(du) \leq \exp(Nc),
\end{equation*}
where $\phi^N$ is defined in \eqref{eq:phiNdeftn}.
\end{lemma}
\begin{proof}
We have from \eqref{eq:phiw} that $\phi^N\left(\mu,v\right) = \olivier{\undt{\phi}_\dag}^N(\mu,w_\dag)$, where \olivier{$w^j_\dag = v_\dag^j$} for all $j$, except that \olivier{$w^0_\dag = v_\dag^0 - Nc^{\mu}$}. Since (by \eqref{eq:tildePdistn}) the distribution of the variables $\olivier{v_\dag}$ under $\olivier{\undt{P}_\dag}^{\otimes N}$ is \olivier{$\mathcal{N}_T\left(0_T,N\sigma^2{\rm Id}_T\right)^{\otimes N}$}, the distribution of \olivier{$w_\dag$} under \olivier{$\undt{P}_\dag^{\otimes N}$} is \olivier{$\mathcal{N}_T\left(Nc^{\mu},N\sigma^2\rm{Id}_T\right)^{\otimes N}$}. By lemma \ref{lem:alphabound}, the eigenvalues of $\tilde{A}^{\mu^N,j}$ are upperbounded by $0 < \alpha <1$, for all $j$. Thus
\begin{equation}
N\olivier{\undt{\phi}_\dag}^N\left(\mu,w_\olivier{\dag}\right) \leq \frac{\alpha}{2N\sigma^2}\sum_{l=-n}^n\norm{w_\olivier{\dag}^l}^2  + \frac{1}{\sigma^2}\langle c^{\mu},w_\olivier{\dag}^0\rangle + \frac{N}{2\sigma^2}\norm{c^\mu}^2.
\end{equation}
Hence we find that 
\begin{multline*}
\int_{\T^{N}}\exp\left(aN\phi^N(\hat{\mu}^N(u),\Psi(u))\right)P^{\otimes N}(du)
\leq \left(\sqrt{2\pi
    N\sigma^2}\right)^{-NT}\times \\ \int_{\T_T^{N-1}}\mathcal{G}_1\exp\left(\frac{1}{2N\sigma^2}\left[\sum_{|j|=1}^n a\alpha \norm{y^j}^2-\norm{y^j}^2\right]\right) \prod_{|j|=1}^{n}\prod_{t=1}^T dy^j_t,
\end{multline*}
 where
\begin{multline*}
\mathcal{G}_1 =\int_{\T_T}\exp\Bigg[\frac{1}{2N\sigma^2}\times \\ \left[a\alpha \norm{y^0}^2 +2aN\langle
  c^{\hat{\mu}^N}, y^0\rangle+aN^2\|c^{\hat{\mu}^N}\|^2-\norm{y^0+Nc^{\hat{\mu}^N}}^2\right]\Bigg] \prod_{t=1}^T dy^0_t. 
\end{multline*}
We note the dependency of $\mathcal{G}_1$ on $(y^j)$ (for all $|j|\neq n$) via $c^{\hat{\mu}^N}$. After diagonalisation, we find that
\begin{multline*}
\mathcal{G}_1 =\int_{\T_T}\exp\left[\frac{N\norm{c^{\hat{\mu}^N}}^2 a(a-1)(1-\alpha)}{2(1-a\alpha)\sigma^2}\right]\times \\
\exp\left[-\frac{(1-a\alpha)}{2N\sigma^2}\sum_{s=1}^T\left(y^0_s - \frac{Nc^{\hat{\mu}^N}_s(a-1)}{1-a\alpha}\right)^2\right]\prod_{s=1}^T dy^0_s. 
\end{multline*}
We assume that $a>1$ is such that $1-a\alpha \olivier{>} 0$. To bound this expression, we note the identity that if $\mathcal{A}:\mathbb{R}\to\mathbb{R}$ satisfies $|\mathcal{A}|\leq \mathcal{B} > 0$ and $\gamma_c>0$, then
\[
\int_{\R}\exp\left(-\frac{1}{2\gamma_c}\left(t-\mathcal{A}(t)\right)^2\right)dt \leq 2\mathcal{B} + \sqrt{2\pi\olivier{\gamma_c}}.
\]
Since $|c^{\hat{\mu}^N}_s| \leq \olivier{|\bar{J}|}$, \olivier{and hence $\|c^{\hat{\mu}^N}\|^2 \leq T \bar{J}^2$}, we therefore find that $\mathcal{G}_1 \leq \mathcal{G}_1^c$, where
\begin{equation*}
\mathcal{G}_1^c =\exp\left[\frac{N \olivier{T\bar{J}^2}a(a-1)(1-\alpha)}{2\sigma^2(1-a\alpha)}\right]\left(\frac{2N\olivier{|\bar{J}|}(a-1)}{1-a\alpha} + \sqrt{\frac{2\pi N\sigma^2}{1-a\alpha}}\right)^T.
\end{equation*}
Thus
\[
\int_{\T^{N}}\exp\left(aN\phi^N(\hat{\mu}^N(u),\Psi(u))\right)P^{\otimes N}(du)
\leq \mathcal{G}_1^c (1-a\alpha)^{-\frac{T(N-1)}{2}}\left(2\pi N\sigma^2\right)^{-\frac{T}{2}},
\]
which yields the lemma.
\end{proof}
\begin{proposition}\label{prop:exptight}
The family $\lbrace \Pi^N\rbrace$ is exponentially tight.
\end{proposition}
\begin{proof}
Let $B\in\B(\mM_{1,s}^+(\T^\Z))$. We have
\[
\Pi^N(B)=\int_{(\hat{\mu}^N)^{-1}(B)} \exp N
\Gamma(\hat{\mu}^N(u))\,P^{\otimes N}(du).
\]
Through H\"older's Inequality, we find that for any $a > 1$ such that $1-a\alpha > 0$:
\begin{equation*}
\Pi^N(B) \leq R^N(B)^{(1-\frac{1}{a})}\left(\int_{(\hat{\mu}^N)^{-1}(B)} \exp\left(a N\Gamma(\hat{\mu}^N(u))\right)P^{\otimes N}(du)\right)^{\frac{1}{a}},
\end{equation*}
Now it may be observed that 
\begin{multline*}
\int_{\T^{N}}\exp\left(aN\Gamma(\hat{\mu}^N(u))\right)P^{\otimes N}(du) \\= \int_{\T^N}\exp\left(aN\phi^N(\hat{\mu}^N(u),\Psi(u)) + aN\Gamma_1(\hat{\mu}^N(u))\right)P^{\otimes N}(du).
\end{multline*}
Since $\Gamma_1 \leq 0$, it follows from lemma \ref{lemma:expNc} that
\begin{equation}
\Pi^N(B) \leq R^N(B)^{(1-\frac{1}{a})}\exp\left(\frac{Nc}{a}\right). \label{eq:PinB}
\end{equation}
By the exponential tightness of $\lbrace R^N\rbrace$ (as proved in lemma \ref{theorem:Rnexptight}), for each $L > 0$, there exists a compact set $K_L$ such that
\begin{equation*}
\lsup{N} N^{-1}\log(R^N(K_L^c)) \leq -L.
\end{equation*}
It thus suffices for us to choose
\begin{equation*}
B = K_{\frac{a}{a-1}(L + \frac{c}{a})}^c.
\end{equation*}
\end{proof}
We finish with another technical lemma which will be of use later on.
\begin{lemma}\label{lem:acbound}
There exist constants $a>1$ and $c>0$ such that for all $\mu\in\mathcal{M}_{1,s}^+(\T^{\Z})\cap\mathcal{E}_2$,
\begin{equation*}
\Gamma(\mu) \leq \frac{\left(I^{(3)}(\mu,P^{\Z}) + c\right)}{a}. 
\end{equation*}
\end{lemma}
\begin{proof}
We have (from (\ref{eq:I3limit})) that
\begin{equation*}
I^{(3)}(\mu,P^{\Z}) = \lim_{N\rightarrow\infty} N^{-1} I^{(2)}\left(\mu^N,P^{\otimes N}\right).
\end{equation*}
We recall that $I^{(2)}$ may be expressed using the Fenchel-Legendre transform as
\begin{multline}
I^{(2)}\left(\mu^N,P^{\otimes N}\right) = \\ 
=\sup_{\varsigma^N \in C_b(\T^N)}\left(\int_{\T^{N}} \varsigma^N(u)\mu^N(du) - \log\int_{\T^{N}}\exp(\varsigma^N(u))P^{\otimes N}(du)\right),\label{eq:I2bound}
\end{multline}
where $\varsigma^N$ is a continuous, bounded function on $\T^N$. We let $\varsigma^N_M = \olivier{a}1_{B_M}\varsigma^N_*$, where $\varsigma^N_*(u) = \olivier{N}(\phi^N(\mu,\Psi(u)) + \Gamma_1(\mu^N))$, and $u\in B_M$ only if either $\norm{\Psi\olivier{(u)}} \leq NM$ or $(\phi^N(\mu,\Psi(u)) + \Gamma_1(\mu^N)) \leq 0$. We proved in section \ref{Sect:Gamma2Properties} that $\phi^N(\mu,\Psi(u))$ possesses a lower bound, which means that $\varsigma^N_M$ is continuous and bounded. Furthermore $\varsigma^N_M$ grows to $\varsigma^N_*$, so that after substituting $\varsigma^N_M$ into \eqref{eq:I2bound} and taking $M\to\infty$ (i.e. applying the dominated convergence theorem), we obtain

\begin{equation}
a\int_{\T^N}\varsigma^N_*(u) \mu^N(du)
\leq\log\int_{\T^N}\exp\left(a\varsigma^N_*(u)\right)\,P^{\otimes N}(du) +
I^{(2)}(\mu^N,P^{\otimes N}).\label{eq:phiNM}
\end{equation}
It can be easily shown, similarly to lemma \ref{lemma:expNc}, that \newline$\log\int_{\T^N}\exp\left(a\varsigma^N_*(u)\right)\,P^{\otimes N}(du)\leq Nc$. We may thus divide both sides by $aN$ and let $N \to \infty$ to obtain the required result.
\end{proof}
\subsection{Upper Bound on the Compact Sets}\label{section:upperboundcompact}
In this section we obtain an upper bound on the compact sets, i.e. the first half of theorem \ref{theo:LDP} for $F$ compact. Our
method is to obtain an LDP for a simplified Gaussian system (with fixed
$A^{\nu}$ and $c^\nu$), and then prove that this converges to the
required bound as $\nu\rightarrow\mu$. 
\subsubsection{An LDP for a Gaussian measure}  
\label{section:Qndefinition}
We linearise $\Gamma$ in the following manner. Fix $\nu\in\M_{1,s}^+(\T^{\Z})$ and assume for the moment that $\mu \in \mathcal{E}_2$. Let
\begin{equation}\label{eq:Gamma2nuN}
\Gamma_{2}^\nu(\mu^N)=\int_{\T_T^N}\phi^N_{\infty}(\nu,v)d\undt{\mu}^N_{1,T}(v), \text{ where }
\end{equation}
\begin{equation}\label{eq:phiNInfinitydefn}
\phi^N_{\infty}(\nu,v) = \frac{1}{2\sigma^2}\Bigg(\frac{1}{N}\sum_{j,k=-n}^{n} 
\,^t (v^\olivier{j}- c^{\nu})A^{\nu,\,k-j}
   (v^k-
c^{\nu} )+\\ \frac{2}{N}\sum_{j=-n}^n \langle c^{\nu}, v^j\rangle -
  \|c^{\nu} \|^2\Bigg).
\end{equation}
Let us also define
\[
\Gamma_1^N(\nu)=-\frac{1}{2N} \log \det \left({\rm Id}_{NT}+\frac{1}{\sigma^2} K^{\nu,N}\right),
\]
where $K^{\nu,N}$ is the $NT\times NT$ matrix with elements given by $K^{\nu,l}_{st}$, \olivier{$l=-n\cdots n$}. \olivier{Its $T \times T$ blocks are noted $K^{\nu,N,l}$}. We define
\begin{equation*}
 \Gamma^\nu(\mu^N) = \Gamma_1^N(\nu) + \Gamma_2^\nu(\mu^N), \text{ and }
\end{equation*}
\olivier{$\Gamma^\nu_2(\mu) = \lim_{N\rightarrow\infty}\Gamma^\nu_2(\mu^N)$.} We
find, using the first identity in proposition \ref{prop:E2Gamma2}, that 
\begin{multline}\label{eq:Gamma2nuspectral}
\Gamma_2^\nu(\mu)=\frac{1}{2\sigma^2}
\left(\frac{1}{2\pi}\int_{-\pi}^{\pi} \tilde{A}^\nu(-\omega) :
  \tilde{v}^\mu(d\omega)\right.\\ \left.-2\,{}^t c^\nu\tilde{A}^\nu(0)
  \bar{v}^\mu+\,^tc^\nu \tilde{A}^\nu(0) c^\nu+2\langle c^\nu, \bar{v}^\mu \rangle - \|c^\nu\|^2\right),
\end{multline}
where $\bar{v}^\mu= \Exp^{\undt{\mu}_{1,T}}[v^0]$, and $\tilde{v}^\mu$ is the spectral measure defined in \eqref{eq:tildemuV}. We recall that $:$ denotes double contraction on the indices.

Similarly to lemma \ref{lemma:Gamma1}, we find that 
\begin{equation}\label{eq:Gamma1nu}
\lim_{N \to \infty} \Gamma_1^N(\nu)=-\frac{1}{4\pi} \int_{-\pi}^\pi
\left(\log \det \left({\rm Id}_{T}+\frac{1}{\sigma^2} \tilde{K}^\nu(\omega)\right)\right)\,d\omega=\Gamma_1(\nu).
\end{equation}
For $\mu\in\mathcal{E}_2$, we define $H^\nu(\mu) = I^{(3)}(\mu,P^\Z) - \Gamma^{\nu}(\mu)$; for $\mu\notin\mathcal{E}_2$, we define $\Gamma_2^\nu(\mu) = \Gamma^\nu(\mu)= \infty$ and $H^\nu(\mu) = \infty$. In fact it will be seen that $H^{\nu}$ is the rate function for the Gaussian Stationary Process $Q^{\nu}$ which we now define.

\begin{definition}
Let $\undt{Q}^{\nu}\in\M_{1,s}^+\left(\T^{\Z}\right)$ with $N$-dimensional marginals $\undt{Q}^{\nu,N}$ given by  
\begin{equation}\label{defn:Qnu}
\undt{Q}^{\nu,N}(B) =
\int_{B}\exp\left(N\Gamma^\nu(\undt{\hat{\mu}}^N(v))\right)\,\undt{P}^{\otimes N}(dv),
\end{equation}
where $B\in\mathcal{B}(\olivier{\T^N})$. This defines a law $Q^{\nu}\in\M_{1,s}^+\left(\T^{\Z}\right)$ according to the
correspondence in definition \ref{def:mubarbar}.
\end{definition}
It is easily shown that $\undt{Q}^{\nu}_{1,T}$ is Gaussian, with covariance operator
$\sigma^2 {\rm Id} + K^{\nu}$ and mean $c^{\nu}$. The spectral density of the covariance is $\sigma^2
{\rm Id}_{T} + \tilde{K}^{\nu}$. \olivier{In addition, 
\begin{equation}\label{eq:muIassumption}
\undt{Q}^{\nu}_0=\mu_{I}^{\Z}.
\end{equation}
}
\begin{definition}
Let $\undt{\Pi}^{\nu,N}$ be the image law of $\undt{Q}^{\nu, \olivier{N}}$ under $\undt{\hat{\mu}}^N$, i.e. for $B\in\B(\M_{1,s}^+(\T^\Z))$,
\begin{equation*}
\undt{\Pi}^{\nu,N}(B) = \undt{Q}^{\nu, \olivier{N}} \left(\undt{\hat{\mu}}^N  \in B\right).
\end{equation*}
\end{definition}
\begin{lemma}\label{lemma:PinuNLDP}
The image law $\undt{\Pi}^{\nu,N}$ satisfies a strong LDP (in the manner of Theorem \ref{theo:LDP}) with good rate function
\begin{equation}\label{eq:Hnuratefunction}
\undt{H}^{\nu}(\undt{\mu}) = I^{(3)}\left(\undt{\mu},\undt{P}^{\Z}\right) - \Gamma^{\nu}(\mu).
\end{equation}
\end{lemma}
This result is proved in appendix \ref{section:proofLemmaLDP}. For $B\in\B(\M_{1,s}^+(\T^{\Z}))$, we define the image law
\begin{equation*}
\Pi^{\nu,N}(B) = Q^{\nu,\olivier{N}}(\hat{\mu}^N\in B)=\undt{Q}^{\nu, \olivier{N}}(\hat{\undt{\mu}}^N \in
B).
\end{equation*}
It follows from the contraction principle that if we write $H^\nu(\mu) := \undt{H}^\nu(\undt{\mu})$, then
\begin{corollary}\label{Cor:PiNnu}
The image law $\Pi^{\nu,N}$ satisfies a strong LDP with good rate function
\begin{equation}\label{eq:Hnuratefunction2}
H^{\nu}(\mu) = I^{(3)}(\mu,P^\Z) - \Gamma^{\nu}(\mu). 
\end{equation}
\end{corollary}
\subsubsection{An upper bound for \protect $\Pi^N$ over compact sets}
In this section we derive an upper bound for $\Pi^N$ over compact sets using the LDP of the previous section. Before we do this, we require some lemmas governing the `distance' between $\Gamma^\nu$ and $\Gamma$. Let $\tilde{K}^{\mu,N}$ be the DFT of $\left(K^{\mu,j}\right)_{j=-n}^n$, and similarly $\tilde{A}^{\mu,N}$ is the DFT of $(A^{\mu,j})_{j=-n}^n$. We define
\begin{equation}
C^\nu_N = \sup_{M\geq N,(2|l|+1)\leq M} \lbrace\norm{\tilde{A}^{\nu^M,l}-\tilde{A}^{\nu,M,l}}, \norm{\tilde{K}^{\nu^M,l}-\tilde{K}^{\nu,M,l}}\rbrace,\label{defn:CnuN}
\end{equation}
where we have taken the operator norm. 
\begin{lemma}\label{lem:cnuNasymptote}
For all $\nu\in \M_{1,s}^+(\T^{\Z})$, $C^{\nu}_N$ is finite and 
\begin{equation*}
C^\nu_N \rightarrow 0 \text{ as } N\rightarrow\infty.
\end{equation*}
\end{lemma}
\begin{proof}
We recall from proposition \ref{prop:Ktildeomega} that $\tilde{K}^{\nu^M}_{st}(\omega)$ converges uniformly (in $\omega$) to $\tilde{K}^{\nu}_{st}(\omega)$. The same holds for $\tilde{K}_{st}^{\nu,M,l}$, because this represents the partial summation of an absolutely converging Fourier Series. That is, for fixed $\omega = 2\pi l_M / M$, $\tilde{K}_{st}^{\nu,M,l_M}\to\tilde{K}_{st}^\nu(\omega)$ as $M\rightarrow\infty$. The result then follows from the equivalence of matrix norms. The proof for $\tilde{A}^\nu$ is analogous.
\end{proof}
\begin{lemma}\label{lemma:Vsquared}
There exists a constant $C_0$ such that for all $\nu$ in $\mathcal{M}_{1,s}^+(\T^\Z)$, all $\varepsilon > 0$ and all $\mu \in V_\varepsilon(\nu)\cap\mathcal{E}_2$, 
\begin{equation*}
\left|\Gamma(\mu^N)-\Gamma^{\nu}(\mu^N)\right|\leq C_0 (C^\nu_N + \varepsilon)(1+\Exp^{\undt{\mu}_{1,T}}[\| v^0 \|^2]).
\end{equation*}
Here $V_\varepsilon(\nu)$ is the open neighbourhood defined in proposition \ref{prop:Ktildeomega}, and $\undt{\mu}$ is given in definition \ref{def:mubarbar}.
\end{lemma}
\begin{proof}
We firstly bound $\Gamma_1$.
\begin{multline*}
\left|\Gamma_1(\mu^N) - \Gamma_1^N(\nu)\right| \leq \\ \frac{1}{2N}\sum_{l=-n}^n \left|\log\det\left({\rm Id}_{T} + \sigma^{-2}\tilde{K}^{\mu^N,l}\right) - \log\det\left({\rm Id}_{T}+\sigma^{-2}\tilde{K}^{\nu^N,l}\right)\right| \\ +\frac{1}{2N}\sum_{l=-n}^n \left|\log\det\left({\rm Id}_{T} + \sigma^{-2}\tilde{K}^{\nu^N,l}\right) - \log\det\left({\rm Id}_{T}+\sigma^{-2}\tilde{K}^{\nu,N,l}\right)\right|.
\end{multline*}
It thus follows from proposition \ref{prop:Ktildeomega} and lemma \ref{lem:cnuNasymptote} that
\begin{equation*}
 \left|\Gamma_1(\mu^N) - \Gamma_1^N(\nu)\right| \leq C_0^* (C^\nu_N + \varepsilon),
\end{equation*}
for some constant $C_0^*$ which is independent of $\nu$ and $N$.

We define $\phi^N_{\infty,\dag}(\nu,\olivier{v_\dag}) = \phi^N_{\infty}(\nu,(\mathcal{H}_T^N)^{-1}(\olivier{v}))$, where $\mathcal{H}_T^{N}$ is given in definition \ref{defn:DFT} and $\phi^N_\infty$ is given in \eqref{eq:phiNInfinitydefn}, and find that
\begin{multline}
\phi^N_{\infty,\dag}(\mu,\olivier{v_\dag}) = \frac{1}{2N^2\sigma^2}\sum_{l=-n}^n {}^t\tilde{v}^{l,*}\tilde{A}^{\nu,-l}\tilde{v}^{l} \\+\frac{1}{N\sigma^2}{}^t\tilde{v}^0\left({\rm Id}_{T}-\tilde{A}^{\nu}(0)\right)c^{\nu} + \frac{1}{2\sigma^2}{}^tc^{\nu}\left({\rm Id}_T-\tilde{A}^{\nu}(0)\right)c^{\nu}.
\end{multline}
This means that
\begin{equation}
\Gamma_2^\nu(\mu^N) - \Gamma_2(\mu) = \int_{\T_T^N}\phi^N_{\infty,\dag}(\nu,\olivier{v_\dag}) - \phi^N_{\dag}(\mu,\olivier{v_\dag})\undt{\mu}_{\dag}^N(\olivier{dv_\dag}).
\end{equation}
Upon expansion of the above expression, we find that
\begin{multline*}
 \left|\phi^N_{\infty,\dag}(\nu,\olivier{v_\dag}) - \phi^N_{\dag}(\mu,\olivier{v_\dag})\right|\leq \\
\frac{1}{2\sigma^2} \Bigg(\frac{1}{N^2} \sum_{l=-n}^n \| \tilde{A}^{\mu^N,-l}-\tilde{A}^{\nu,N,-l}\|
\| \tilde{v}^l \|^2 +\frac{2}{N} \| d_{\nu,\mu} \| \| \tilde{v}^0 \| +|e_{\nu,\mu} |
\Bigg),
\end{multline*}
where $d_{\nu,\mu}=c^\mu-c^\nu+\tilde{A}^{\nu,N,0}
c^\nu-\tilde{A}^{\mu^N,0} c^\mu$and $e_{\nu,\mu}=\,^t
c^\mu \tilde{A}^{\mu^N,0} c^\mu -\| c^\mu \|^2-\,^t
c^\nu \tilde{A}^{\nu,N,0} c^\nu +\| c^\nu \|^2$. 
It follows from proposition \ref{lem:muNconvergenceA} and lemma \ref{lem:cnuNasymptote} that the (Euclidean) norm each of the above
terms is bounded by $C^*(C^\nu_N + \varepsilon)$ for some constant
$C^*$. 

The lemma now follows after consideration of the fact that
$\int_{\T_T^\Z}\| v^k\|^2\undt{\mu}_\olivier{1.T}(dv) = \Exp^{\undt{\mu}_{1,T}}[\| v^0
\|^2]$, $\| \tilde{v}^0 \|^2 \leq N \sum_{k=-n}^n \| v^k \|^2$ (\olivier{Cauchy-Schwarz}) and, because of the properties of the DFT, $\sum_{l=-n}^n \| v^l \|^2 = N\sum_{k=-n}^n \| \tilde{v}^k \|^2.$
\end{proof}
We are now ready to begin the proof of the upper bound on compact sets.
\begin{proposition}\label{prop:ubcompact}
Let $K$ be a compact subset of $\M_{1,s}(\T^{\Z})$. Then \newline
$\lsup{N} N^{-1}\log(\Pi^N(K))\leq-\inf_K H$.
\end{proposition}
\begin{proof}
Fix $\varepsilon > 0$. Let $V_\varepsilon(\nu)$ be the open neighbourhood of $\nu$ defined in proposition \ref{prop:Ktildeomega}, and let $\bar{V}_\varepsilon(\nu)$ be its closure. Since $K$ is compact and $\lbrace V_\varepsilon(\nu)\rbrace_{\nu \in K}$ is an open cover, there exists an $r$ and
$\lbrace \nu_i\rbrace_{i=1}^r$ such that $K\subset\bigcup_{i=1}^r
V_\varepsilon(\nu_i)$. We find that
\begin{multline*}
\lsup{N} N^{-1}\log\left(\Pi^N\left(\bigcup_{i=1}^r V_\varepsilon(\nu_i)\cap K\right)\right)\\ \leq \sup_{1\leq i\leq r}\lsup{N} N^{-1}\log\left(\Pi^N\left(\bar{V}_\varepsilon(\nu_i)\cap K\right)\right).
\end{multline*}
It follows from the fact that $\hat{\mu}^N\in\mathcal{E}_2$, lemma \ref{lemma:Vsquared} and the definition of $\Pi^N$ that
\begin{multline}
\Pi^N(\bar{V}_\varepsilon(\nu_i)\cap K) \leq \int_{\hat{\mu}^N(u)\in
  \bar{V}_\varepsilon(\nu_i)\cap K}\exp\Bigg(N\Gamma^{\nu_i}(\hat{\mu}^N(u))+\\
\left.NC_0 (\varepsilon + C^{\nu_i}_N) \left(1+\frac{1}{N}\sum_{j=-n}^n\|\Psi_{{1,T}}(u^j) \|^2\right)\right)P^{\otimes N}_{1,T}(du),\label{eq:PiN1}
\end{multline}
where \olivier{if $u \in \T$,} $\Psi_{1,T}(u) = (\Psi(u)_1,\ldots,\Psi(u)_T)$. From the definition of $Q^{\nu,N}$ in (\ref{defn:Qnu}) and H\"older's Inequality, for $p,q$ such that $\frac{1}{p} + \frac{1}{q} = 1$, we have
\begin{equation}
\Pi^N(\bar{V}_\varepsilon(\nu_i)\cap K) \leq \left(Q^{\nu_i,N}(\hat{\mu}^N(u)\in \bar{V}_\varepsilon(\nu_i)\cap K)\right)^{\frac{1}{p}}D^{\frac{1}{q}},\label{eq:PiN2}
\end{equation}
where
\begin{multline*}
D = \int_{\hat{\mu}^N(u)\in \bar{V}_\varepsilon(\nu_i)\cap
  K}\exp\left(qNC_0 (\varepsilon + C^{\nu_i}_N)\left(1+\frac{1}{N}\sum_{j=-n}^n
    \|\Psi_{1,T}(u^j)\|^2\right)\right)Q^{\nu_i,N}_{1,T}(du)\\
\olivier{\leq} \exp qNC_0(\varepsilon + C^{\nu_i}_N)\times\\ \int_{\undt{\hat{\mu}}^N(v)\in \Psi(\bar{V}_\varepsilon(\nu_i)\cap
  K)}\exp  \left(qC_0(\varepsilon + C^{\nu_i}_N) \left( \sum_{j=-n}^n
    \|v^j\|^2\right)\right)\undt{Q}^{\nu_i,N}_{1,T}(dv).
\end{multline*}
We note from lemma \ref{eq:rhoK} that the eigenvalues of the covariance of $\undt{Q}^{\nu_i,N}$ are upperbounded by $\sigma^2 + \rho_K$. Thus for this integral to converge it is sufficient that
\begin{equation}
qC_0(\varepsilon + C^{\nu_i}_N) \leq \frac{1}{2(\sigma^2 + \rho_K)}.\label{eq:integralconvergecond}
\end{equation}
This condition will always be satisfied for sufficiently small $\varepsilon$ and sufficiently large $N$ (since $C^{\nu_i}_N\rightarrow 0$ as $N\rightarrow\infty$). By corollary \ref{Cor:PiNnu},
\begin{equation}\label{eq:BndBnui}
 \lsup{N}N^{-1} \log \left(Q^{\nu_i,N}(\hat{\mu}^N(u)\in \bar{V}_\varepsilon(\nu_i)\cap K)\right) \leq -\inf_{\mu\in \bar{V}_\varepsilon(\nu_i)\cap K} H^{\nu_i}(\mu),
\end{equation}
where we defined $\undt{Q}^{\nu_i,N}$ in the previous section. We apply lemma \ref{lemma:gauss} to find
\begin{multline*}
\int_{\undt{\hat{\mu}}^N(v)\in \Psi (\bar{V}_\varepsilon(\nu_i)\cap
  K)}\exp q C_0(\varepsilon + C^{\nu_i}_N) \left( \sum_{j=-n}^n
    \|v^j\|^2\right)\undt{Q}^{\nu_i,N}_{1,T}(dv) \leq \\
\left(\det \left((1-2qC_0 (\varepsilon + C^{\nu_i}_N) \sigma^2){\rm
    Id}_{NT}-2qC_0(\varepsilon + C^{\nu_i}_N)K^{\nu_i,N}\right)\right)^{-\frac{1}{2}}\times \\
\exp\left(2C_0^2 q^2((\varepsilon + C^{\nu_i}_N)^2) {}^t (1_{NT} 
  c^{\nu_i})B(\olivier{({\rm Id}_T \otimes \mone_N)} c^{\nu_i})+NqC_0 (\varepsilon + C^{\nu_i}_N)\norm{c^{\nu_i}}^2\right)
\end{multline*}
where $\olivier{{\rm Id}_T \otimes \mone_N}$ is the $NT \times T$ block matrix with each block ${\rm Id}_{T}$ and 
\begin{equation*}
B = (\sigma^2 {\rm Id}_{NT} +K^{\nu_i,N})((1-2C_0 q(\varepsilon + C^{\nu_i}_N) \sigma^2){\rm
    Id}_{NT}-2C_0 q(\varepsilon + C^{\nu_i}_N) K^{\nu_i,N})^{-1}
\end{equation*}
is a symmetric block circulant matrix. 

We note $B^k$, $k=-n,\cdots,n$
its $T \times T$ blocks. We have
\[
{}^t (\olivier{({\rm Id}_T \otimes \mone_N)} 
  c^{\nu_i})B(\olivier{({\rm Id}_T \otimes \mone_N)} c^{\nu_i})=N \,^t c^{\nu_i} \left(\sum_{k=-n}^n B^k\right)
  c^{\nu_i}=N\,^t c^{\nu_i} \tilde{B}^0 c^{\nu_i},
\]
where $\tilde{B}^0$ is the 0th component of the spectral
representation of the sequence $(B^k)_{k=-n,\cdots,n}$.
Let $v_{m}$ be the largest eigenvalue of $B$. Since (by lemma \ref{lem:preserveeigenvalues}) the eigenvalues of $\tilde{B}^0$ are a subset of the eigenvalues of $B$, we have 
\[
{}^t (\olivier{({\rm Id}_T \otimes \mone_N)} 
  c^{\nu_i})B(\olivier{({\rm Id}_T \otimes \mone_N)} c^{\nu_i}) \leq N v_m \|c^{\nu_i} \|^2.
\]
From the definition of $B$ and through lemma \ref{eq:rhoK} we have
\[
v_m \leq \frac{\sigma^2+\rho_K}{1-2C_0 q(\varepsilon + C^{\nu_i}_N)(
  \sigma^2+\rho_K)}.
\]
Hence we have, since $\| c^{\nu_i} \|^2 \leq T\bar{J}^2$
\begin{multline*}
\exp\left(2C_0^2( q^2(\varepsilon + C^{\nu_i}_N)^2 {}^t (\olivier{({\rm Id}_T \otimes \mone_N)} 
  c^{\nu_i})B(\olivier{({\rm Id}_T \otimes \mone_N)} c^{\nu_i})\right) \leq \\ \exp\left(  NT \times \frac{2C_0^2 q^2(\varepsilon + C^{\nu_i}_N)^2(\sigma^2+\rho_K)\bar{J}^2}{1-2C_0  q(\varepsilon + C^{\nu_i}_N)(
  \sigma^2+\rho_K)}\right).
\end{multline*}
Since the determinant is the product of the eigenvalues, we similarly find that
\begin{multline*}
   \left(\det \left((1-2C_0 q(\varepsilon + C^{\nu_i}_N) \sigma^2){\rm
    Id}_{NT}-2C_0 q(\varepsilon + C^{\nu_i}_N) K^{\nu_i,N}\right)\right)^{-\frac{1}{2}}
\leq \\ \left(  1-2C_0 q(\varepsilon + C^{\nu_i}_N)(
  \sigma^2+ \rho_K) \right)^{-\frac{NT}{2}}.
\end{multline*}
Upon collecting the above inequalities, and noting that $\norm{c^\nu}^2\leq T\bar{J}^2$, we find that
\begin{equation}
D \leq \exp(N s^{\nu_i}_N(q,\varepsilon)),\label{eq:Dtemporary}
\end{equation}
where
\begin{multline*}
s^{\nu_i}_N(q,\varepsilon) =T\left( -\frac{1}{2} \log \left( 1-2C_0 q(\varepsilon + C^{\nu_i}_N)(
  \sigma^2+ \rho_K)\right)\right.\\ \left. + \frac{2C_0^2 q^2(\varepsilon + C^{\nu_i}_N)^2(\sigma^2+\rho_K)\bar{J}^2}{1-2C_0 q(\varepsilon + C^{\nu_i}_N)(
  \sigma^2+\rho_K)}+qC_0(\varepsilon + C^{\nu_i}_N) \left(\frac{1}{T}+\bar{J}^2\right) \right).
\end{multline*}
We let $s(q,\varepsilon) = \lsup{N} s^{\nu_i}_N(q,\varepsilon)$, and find through lemma \ref{lem:cnuNasymptote} that
\begin{multline*}
s(q,\varepsilon) =T\left( -\frac{1}{2} \log \left( 1-2C_0 q\varepsilon(
  \sigma^2+ \rho_K)\right)\right.\\ \left. + \frac{2C_0^2 q^2\varepsilon^2(\sigma^2+\rho_K)\bar{J}^2}{1-2C_0 q\varepsilon(
  \sigma^2+\rho_K)}+qC_0 \varepsilon \left(\frac{1}{T}+\bar{J}^2\right) \right).
\end{multline*}
Notice that $s(q,\varepsilon)$ is independent of $\nu_i$ and that $s(q,\varepsilon)\rightarrow 0$ as $\varepsilon\rightarrow 0$. Using \eqref{eq:PiN2}, \eqref{eq:BndBnui} and \eqref{eq:Dtemporary} we thus find that
\begin{equation*}
\lsup{N} N^{-1}\log(\Pi^N(K)) \leq \sup_{1\leq i\leq r}-\frac{1}{p}\inf_{\mu\in K\cap \bar{V}_\varepsilon(\nu_i)}H^{\nu_i}(\mu) - \frac{1}{q}s(q,\varepsilon).
\end{equation*}
Recall that $H^{\nu}(\mu)=\infty$ for all $\mu\notin\mathcal{E}_2$. Thus if $K\cap \mathcal{E}_2 = \emptyset$, we may infer that $\lsup{N} N^{-1}\log(\Pi^N(K)) = -\infty$ and the proposition is evident. Thus we may assume without loss of generality that $\inf_{\mu\in K}H^{\nu_i}(\mu) = \inf_{\mu\in K\cap\mathcal{E}_2}H^{\nu_i}(\mu)$. Furthermore it follows from proposition \ref{prop:Gamma2nu} (below) that there exists a constant $C_I$ such that for all $\mu\in \bar{V}_\varepsilon(\nu_i)\cap\mathcal{E}_2$, 
\[
H^{\nu_i}(\mu)\geq I^{(3)}(\mu,P^\Z) - \Gamma(\mu) - C_I\varepsilon(1+I^{(3)}(\mu,P^\Z)).
\]
We thus find that 
\begin{multline*}
\lsup{N} N^{-1}\log(\Pi^N(K)) \leq \\-\frac{1}{p}\inf_{K\cap\mathcal{E}_2}\left(I^{(3)}(\mu,P^\Z)(1-C_I\varepsilon)-\Gamma(\mu)\right) - \frac{s(q,\varepsilon)}{q}+\frac{\varepsilon}{p}C_I,
\end{multline*}
We take $\varepsilon \rightarrow 0$ and find, through the use of lemma \ref{lem:acbound}, that
\begin{equation*}
\lsup{N} N^{-1}\log(\Pi^N(K)) \leq -\frac{1}{p}\inf_{K}\left(I^{(3)}(\mu,P^\Z)-\Gamma(\mu)\right).
\end{equation*}
The proof may thus be completed by taking $p\rightarrow 1$.
\end{proof}
\begin{proposition}\label{prop:Gamma2nu}
There exists a positive constant $C_I$ such that, for all $\nu$ in $\mathcal{M}_{1,s}^+(\T^\Z)\cap\mathcal{E}_2 $, all $\varepsilon >
0$ and all $\mu \in \bar{V}_\varepsilon(\nu)\cap\mathcal{E}_2$ (where $\bar{V}_\varepsilon(\nu)$ is the neighbourhood defined in proposition \ref{prop:Ktildeomega}),
\begin{equation}
\left|\Gamma^\nu(\mu)-\Gamma^\mu(\mu)\right|\leq C_I \varepsilon \left(1+I^{(3)}(\mu,P^\Z)\right).\label{eq:secondbound}
\end{equation}
\end{proposition}
The proof is very similar to that of lemma \ref{lemma:Vsquared} and we leave it to the reader.
\subsection{End of the proof of theorem \ref{theo:LDP}}
\label{sect:Hgoodratefunction}
\begin{lemma}\label{lemma:Hsemicontinuous}
$H(\mu)$ is lower-semi-continuous.
\end{lemma}
\begin{proof}
Fix $\mu$ and let $(\mu_{m})_{m \geq 0}$ converge weakly to $\mu$ as $m\rightarrow\infty$. We let $(\mu_{p_m})$ be a subset such that $\linf{m}H\left(\mu_{m}\right) = \limunder{m}H(\mu_{p_m})$. Suppose firstly that 
\begin{equation}
\lsup{m} I^{(3)}\left(\mu_{p_m},P^{\Z}\right) = \infty. \label{eq:assmpt1Hgoodfunction}
\end{equation}
From lemma \ref{lem:acbound} we have that, if $\mu_{p_m}\in\mathcal{E}_2$, then \newline$H(\mu_{p_m})\geq \left(1-\frac{1}{a}\right)I^{(3)}\left(\mu_{p_m}\right)-\frac{c}{a}$, where $a > 1$ and $c>0$ are constants. Otherwise, if $\mu_{p_m}\notin\mathcal{E}_2$ then (through lemma \ref{lemma:E2enough}) $H(\mu_{p_m})=\infty$. In either case, we find that $\lsup{m}H\left(\mu_{p_m}\right) = \limunder{m}H(\mu_{p_m}) = \infty$, so that in this instance $H$ is lower-semicontinuous at $\mu$.

In the second instance, we assume that (\ref{eq:assmpt1Hgoodfunction}) does not hold, so that there exists an $M$ such that for all $m\geq M$, $\lbrace I^{(3)}\left(\mu_{p_m},P^{\Z}\right)\rbrace$ is upperbounded (and by lemma \ref{lemma:E2enough}, $\mu_{p_m}\in\mathcal{E}_2$). We then find that
\begin{align*}
\linf{m}H\left(\mu_{p_m}\right) &= \linf{m}\left(I^{(3)}(\mu_{p_m},P^{\Z})-\Gamma(\mu_{p_m})\right) \\
&\geq \linf{m} H^{\mu}(\mu_{p_m}) + \linf{m} (\Gamma^\mu - \Gamma)(\mu_{p_m}).
\end{align*}
Recall that $\Gamma(\mu_{p_m}) = \Gamma^{\mu_{p_m}}(\mu_{p_m})$. It follows from proposition \ref{prop:Gamma2nu} and the boundedness of  $I^{(3)}\left(\mu_{p_m}\right)$ that the second term is zero. However, from lemma \ref{lem:hnusemicont}, $H^{\mu}$ is lower-semi-continuous, which allows us to conclude that $\linf{m} H^{\mu}(\mu_{p_m}) \geq H^{\mu}(\mu) = H(\mu)$ as required.
\end{proof}
Because $\lbrace\Pi^N\rbrace$ is exponentially tight and satisfies the weak LDP with rate function $H(\mu)$, the following corollary is immediate \cite[Lemma 2.1.5]{deuschel-stroock:89}. 
\begin{corollary}\label{coro:goodrate}
$H(\mu)$ is a good rate function, i.e. the sets $\lbrace \mu : H(\mu) \leq \delta\rbrace$ are compact for all $\delta\in\R^+$, \olivier{and it satisfies the first condition of theorem \ref{theo:LDP}.}
\end{corollary}
\noindent
\olivier{This allows us to complete the proof of theorem \ref{theo:LDP}}:
\begin{proof}
\olivier{
By combining lemmas \ref{lemma:Hsemicontinuous} and \ref{lemma:lbopen}, proposition \ref{prop:exptight}, and corollary 
\ref{coro:goodrate}, we complete the proof of theorem \ref{theo:LDP}.
}
\end{proof}
\section{The unique minimum of the rate function}\label{sect:Hminimum}
We first prove that there exists a unique minimum $\mu_e$ of the rate function. We finish by providing explicit equations for $\mu_e$ which would facilitate its numerical simulation.
\begin{lemma}
For $\mu,\nu\in\mM_{1,s}^+(\T^\Z)$, $H^\nu(\mu) = 0$ if and only if $\mu = Q^{\nu}$.
\end{lemma}
\begin{proof}
Using the correspondences in section \ref{section:Qndefinition}, it suffices to
prove that $\undt{H}^\nu\left(\undt{\mu}\right) = 0$ if and only if $\undt{\mu} =
\undt{Q}^{\nu}$. We have from theorem \ref{lemma:PiNLDP1} that if $\mu\in\mathcal{E}_2$, then $H^{\nu}(\undt{\mu}) = I^{(3)}\left(\undt{\mu}\circ(\tau^{\nu})^{-1},\undt{P}^{\Z}\right)$. In turn a contraction principle \cite{ellis:85} dictates that 
\begin{equation*}
I^{(3)}\left( \undt{\mu}\circ (\tau^\nu)^{-1},\undt{P}^{\Z}\right) \geq 0,
\end{equation*}
with equality if and only if $\undt{\mu}\circ (\tau^\nu)^{-1} = \undt{P}^{\Z}$. However since $\undt{P}^{\olivier{\Z}} = \undt{Q}^{\nu}\circ (\tau^\nu)^{-1}$, it follows from lemma \ref{lemma:inversesigmatau} that $\undt{H}^{\nu}(\undt{\mu})$ is zero if and only if $\undt{\mu} = \undt{Q}^{\nu}$.
\end{proof}
\begin{proposition}\label{prop:unique}
There is a unique distribution $\mu_e \in \mathcal{M}_{1,s}^+(\T^\Z)$ which minimises $H$. This distribution satisfies $H(\mu_e)=0$. 
\end{proposition}
\begin{proof}
By the previous lemma, it suffices to prove that there is a unique $\mu_e$ such that 
\begin{equation}
Q^{\mu_e} = \mu_e.\label{eq:Qmumu}
\end{equation}
 We define the mapping $L : 
\mathcal{M}_{1,s}^+(\T^\Z) \to \mathcal{M}_{1,s}^+(\T^\Z)$ by
\[
\mu \to L(\mu)=Q^\mu.
\]
It follows from \eqref{eq:muIassumption} that
\begin{equation}\label{eq:nu0}
Q^\mu_{\olivier{0}}=\mu_I^\Z,
\end{equation}
which is independent of $\mu$. 

It may be inferred from the definitions in Section \ref{subsub:infinite} that the marginal of $L(\mu)=Q^\mu$ over $\F_t$ only depends upon the
marginal of $\mu$ over $\F_{t-1}$. This follows from the fact that $\undt{Q}^\mu_{\olivier{s}}$ (which determines $Q^\mu_\olivier{s}$) is completely determined by the means $\lbrace c^{\mu}_t;t=0,\ldots,s-1\rbrace$ and covariances $\lbrace K^{\mu,j}_{uv}; j\in\Z,u,v\in [0,s-1]\rbrace$.  In turn, it may be observed from \eqref{eq:cmumean} and  \eqref{eq:Kimuinfinite} that these variables are determined by $\mu_\olivier{s-1}$. Thus for any $\mu,\nu\in\mathcal{M}_{1,s}^+(\T^\Z)$ and $t\in [1,T]$, if
\[
\mu_\olivier{t-1}=\nu_\olivier{t-1},
\]
then
\[
L(\mu)_\olivier{t}=L(\nu) _\olivier{t}.
\]
It follows from repeated application of the above identity that for any $\nu$ satisfying $\nu_\olivier{0} = \mu_I^\Z$,
\begin{equation}\label{eq:fixed}
L^T(\nu)_\olivier{T}=L(L^T(\nu))_\olivier{T}.
\end{equation}
Defining
\begin{equation}\label{eq:mufixeddefn}
\mu_e=L^T(\nu),
\end{equation}
it follows from \eqref{eq:fixed} that $\mu_e$ satisfies \eqref{eq:Qmumu}.

Conversely if $\mu=L(\mu)$ for some $\mu$, then we have that $\mu =
L^2(\nu)$ for any $\nu$ such that $\nu_\olivier{T-2} =
\mu_\olivier{T-2}$. Continuing this reasoning, we find that $\mu =
L^T(\nu)$ for any $\nu$ such that $\nu_\olivier{0} = \mu_\olivier{0}$. But
by \eqref{eq:nu0}, since $Q^\mu= \mu$, we have $\mu_\olivier{0}=
\mu_I^\Z$. But we have just seen that any $\mu$ satisfying $\mu =
L^T(\nu)$, where $\nu_\olivier{0} = \mu_{I}^\Z$, is uniquely defined by \eqref{eq:mufixeddefn}, which means that $\mu=\mu_e$.
\end{proof}

We may use the proof of proposition \ref{prop:unique} to characterize the unique measure ${\mu_e}$ such that ${\mu_e}=Q^{\mu_e}$ in terms of its image $\undt{\mu_e}$. This characterization allows one to directly numerically calculate $\mu_e$. We characterize $\undt{\mu_e}$ recursively (in time), by providing a method of determining $\undt{\mu_{e}}_t$ in terms of $\undt{\mu_e}_{t-1}$. However we must firstly outline explicitly the bijective correspondence between ${\mu_e}_t$ and $\undt{\mu_e}_t$, as follows. For $v\in\T$, we write $\Psi^{-1}(v) = (\Psi^{-1}(v)_0,\ldots,\Psi^{-1}(v)_T)$. We recall from \eqref{eq:Psidefn} that $\Psi^{-1}(v)_0 = \olivier{v_0}$. The coordinate $\Psi^{-1}(v)_{t}$ is the affine function of $v_s$, $s=0\cdots t$ obtained from equations \eqref{eq:Psidefn} and
\eqref{eq:Psi}
\begin{equation}\label{eq;psimoins1}
\Psi^{-1}(v)_{t}=\sum_{i=0}^{t} \gamma^i v_{t-i}+\bar{\theta} \frac{\gamma^{t}-1}{\gamma-1}.
\end{equation}

We have the following proposition.
\begin{proposition}\label{prop:effective}
The mean of the limit law $\mu_e$ can be computed iteratively by
\begin{equation}\label{eq:mean}
c^{\mu_e}_t=\bar{J} \int_{\R^{t}}\left(f\circ \Psi^{-1}(v)_{t-1}\right)\,\undt{{\mu^1_{e}}}{}_{t-1}(dv),
\end{equation}
where the measure $\undt{{\mu^1_{e}}}{}_{t-1}$ is given by
\[
\undt{{\mu_e}}{}_{\olivier{t-1}}^{1}(dv)=\mu_I(dv_0)\otimes \mathcal{N}_{t-1}(c^{\mu_e}_{(t-1)},K^{{\mu_e},0}_{(t-1,t-1)})dv_1 \cdots dv_{t-1}.
\]
$K^{{\mu_e},l}_{(t-1,s-1)}$ is the $\olivier{(t-1)} \times \olivier{(s-1)}$ submatrix of $K^{\mu_e,l}$ composed of the rows from times $\olivier{1}$ to $(t-1)$
and the columns from times $\olivier{1}$ to $(s-1)$, and
\[
c^{\mu_e}_{(t-1)} = \,^t(c^{\mu_e}_\olivier{1},\ldots,c^{\mu_e}_{t-1}).
\]
The covariance $K^{\mu_e}$ of the limit law $\mu_e$ can be computed from equation \eqref{eq:Kimuinfinite}  and
\begin{equation}\label{eq:Mmue0}
M^{{\mu_e},0}_{ts} = 
\int_{\R^{\max(t,s)}} \left(f \circ \Psi^{-1}(v)_{t-1}\right) \times \left(f \circ \Psi^{-1}(v)_{s-1}\right)\, \undt{{\mu_e}}{}_{\olivier{\max(t-1,s-1)}}^{1}(dv),
\end{equation}
and for $l \neq 0$
\begin{equation}\label{eq:Mmue}
M^{{\mu_e},l}_{ts} = 
\int_{\R^{t}\times \R^{s}}\left(f \circ \Psi^{-1}(v^0)_{t-1}\right) \times \left(f \circ \Psi^{-1}(v^l)_{s-1}\right)\, \undt{{\mu_e}}{}_{\olivier{t-1,s-1}}^{0,l}(dv^0dv^l) 
\end{equation}
for $s=1\cdots t$.
\olivier{
Moreover (with a slight abuse of notation):
\begin{multline*}
\undt{{\mu_e}}{}_{t-1,s-1}^{(0,l)}(dv^0dv^l)=\mu_I(dv^0_0)\otimes \mu_I(dv^l_0 )\otimes \\
\mathcal{N}_{(t-1)+(s-1)}((c^{\mu_e}_{(t-1)},c^{\mu_e}_{(s-1)}),K^{{\mu_e},(0,l)}_{(t-1,s-1)})dv^0_1 \cdots dv^0_{t-1}dv^l_1 \cdots dv^l_{s-1},
\end{multline*}}
where
\[
K^{{\mu_e},(0,l)}_{(t-1,s-1)}=
\left[
\begin{array}{cc}
K^{{\mu_e},0}_{(t-1,t-1)} & K^{{\mu_e},l}_{(t-1,s-1)}\\
\,^t K^{{\mu_e},l}_{(t-1,s-1)} & K^{{\mu_e},0}_{(s-1,s-1)}
\end{array}
\right]
.
\]
\end{proposition}
\begin{proof}
In the course of the previous proof we saw that $\mu_e{}_{0} = \undt{\mu_e}{}_0 = \mu_I^{\Z}$. It remains for us to explicitly outline how we determine $\undt{\mu_e}{}_{t}$ from $\undt{\mu_e}{}_{t-1}$ for each $t\geq 1$. We saw in the previous proof that $\undt{\mu_e}{}_{t} = \undt{\mu}_I^{\olivier{\Z}}\otimes\undt{\mu_e}_{1,t}$, where $\undt{\mu_e}_{1,t}$ is a Gaussian Process (for $t\geq 1$). Hence it suffices for us to provide expressions for $c^{\mu_e}_t$ and $\lbrace K^{\mu_e,j}_{ts}: s=\olivier{1},\ldots,t,j\in\Z\rbrace$, \olivier{since $K^{\mu_e,j}_{st}=K^{\mu_e,-j}_{ts}$}, in terms of $\undt{\mu_e}{}_{t-1}$. The other components of the mean and covariance of $\undt{\mu_e} {}_{t}$ are the same as their analogues in $\undt{\mu_e}{}_{t-1}$. The mean is given, through the change of variable $v=\Psi(u)$, using definition \ref{def:mubarbar}, by 
\begin{equation*}
c^{\mu_e}_t=\bar{J}\int_{\mathbb{R}} f(u_{t-1})\,\mu^1_e{}_{t-1}(du) = \bar{J} \int_{\R^{t}}\left(f\circ \Psi^{-1}(v)_{t-1}\right)\,\undt{{\mu^1_{e}}}{}_{t-1}(dv),
\end{equation*}
where $\mu_e^1$ is the marginal distribution over one neuron and
\olivier{
\[
\undt{{\mu^1_{e}}}{}_{t-1}(dv)=\mu_I(dv_0) \otimes \undt{{\mu^1_{e}}}{}_{1,t-1}(dv_1\cdots dv_{t-1}),
\]
where $\undt{{\mu^1_{e}}}{}_{1,t-1}$ is a $(t-1)$-dimensional Gaussian measure described below.
}

The formula for $K^{{\mu_e},j}$ can be obtained from equations \eqref{eq:Kimuinfinite} and \eqref{eq:Mmuinfinite}. Indeed, we have
\[
K^{{\mu_e},j}=\bar{\theta}^2 \olivier{\delta_j} \olivier{\mone}_{T} \,^t \olivier{\mone}_{T}+
\sum_{l=-\infty}^{\infty} \Lambda(j,l) M^{{\mu_e},l},
\]
and 
\[
M^{{\mu_e},l}_{ts} = \int_{\T^\Z} \olivier{f(u^0_{t-1})} \olivier{f(u^l_{s-1})} d{\mu_e}(u)\quad s =1\cdots t.
\]
\olivier{This can be rewritten, again through the change of variable $v=\Psi(u)$, using definition \ref{def:mubarbar}, as}
\begin{equation*}
M^{{\mu_e},l}_{ts} = \left\{
\begin{array}{lr}
\int_{\R^{\max(t,s)}} \left(f \circ \Psi^{-1}(v)_{t-1}\right) \times \left(f \circ \Psi^{-1}(v)_{s-1}\right)\, \undt{{\mu_e}}{}_{\olivier{\max(t-1,s-1)}}^{1}(dv) & l =0\\
\int_{\R^{t}\times \R^{s}}\left(f \circ \Psi^{-1}(v^0)_{t-1}\right) \times \left(f \circ \Psi^{-1}(v^l)_{s-1}\right)\, \undt{{\mu_e}}{}_{\olivier{t-1,s-1}}^{0,l}(dv^0dv^l) & l \neq 0
\end{array}
\right.
\end{equation*}
for $s=1\cdots t$.
\olivier{
For $l\neq 0$ (with a slight abuse of notation):
\begin{multline*}
\undt{{\mu_e}}{}_{t-1,s-1}^{(0,l)}(dv^0dv^l)=\mu_I(dv^0_0)\otimes \mu_I(dv^l_0 )\otimes \\
\mathcal{N}_{(t-1)+(s-1)}((c^{\mu_e}_{(t-1)},c^{\mu_e}_{(s-1)}),K^{{\mu_e},(0,l)}_{(t-1,s-1)})dv^0_1 \cdots dv^0_{t-1}dv^l_1 \cdots dv^l_{s-1},
\end{multline*}}
and
\[
c^{\mu_e}_{(t-1)} = \,^t(c^{\mu_e}_\olivier{1},\ldots,c^{\mu_e}_{t-1}),
\]
\[
K^{{\mu_e},(0,l)}_{(t-1,s-1)}=
\left[
\begin{array}{cc}
K^{{\mu_e},0}_{(t-1,t-1)} & K^{{\mu_e},l}_{(t-1,s-1)}\\
\,^t K^{{\mu_e},l}_{(t-1,s-1)} & K^{{\mu_e},0}_{(s-1,s-1)}
\end{array}
\right]
,
\]
and $K^{{\mu_e},l}_{(t-1,s-1)}$ is the $\olivier{(t-1)} \times \olivier{(s-1)}$ submatrix of $K^{\mu_e,l}$ composed of the rows from times $\olivier{1}$ to $(t-1)$
and the columns from times $\olivier{1}$ to $(s-1)$. Obviously, if either $t-1$ or $s-1$ is equal to 0, the corresponding matrixes are empty.

\olivier{
For $l=0$ we have
\[
\undt{{\mu_e}}{}_{\olivier{t-1}}^{1}(dv)=\mu_I(dv_0)\otimes \mathcal{N}_{t-1}(c^{\mu_e}_{(t-1)},K^{{\mu_e},0}_{(t-1,t-1)})dv_1 \cdots dv_{t-1}
\]
}
\end{proof}
\olivier{
Note that the integral in the righthand side of equation \eqref{eq:mean} can be reduced through a change of variable to an integral over at most $\R^2$. Similarly, the integrals in the righthand side  of equations \eqref{eq:Mmue0}  and \eqref{eq:Mmue} can be reduced to integrals computed over at most $\R^4$.}

\section{Conclusion}\label{sect:conclusion}
In this section we  sketch out some important consequences of our work and possible generalizations.
\subsection{Important consequences}
We note that the LDP of Moynot and Samuelides \cite{moynot:99,moynot-samuelides:02} may be obtained from ours by stipulating that $\Lambda(a,b)$ is nonzero if and only if $a=b=0$. Their LDP may then be obtained by applying a contraction principle to our LDP through taking the $1$-dimensional marginal of $\hat{\mu}^N$. More generally, for any $d\in\Z^+$ one may obtain a process-level LDP governing the interaction of each neuron with its $d$ neighbours by applying a contraction principle to the $d-$dimensional marginal of the empirical measure. 


We state some important consequences of our results, culminating in an analog of the Ergodic Theorem. We recall that $Q^N(J,\Theta)$ is the conditional law of $N$ neurons for given $J$ and $\Theta$.
\begin{theorem}\label{theo:weakconv}
$\Pi^N$ converges weakly to $\delta_{\mu_e}$, i.e., for all $\Phi \in \mathcal{C}_b(\mathcal{M}_{1,s}^+(\T^\Z))$,
\[
\lim_{N \to \infty} \int_{\T^N} \Phi(\hat{\mu}^N(u))\,Q^N(du)=\Phi(\mu_e).
\]
Similarly,
\[
\lim_{N \to \infty} \int_{\T^N} \Phi(\hat{\mu}^N(u))\,Q^N(J,\Theta)(d\olivier{u})=\Phi(\mu_e)
\]
\end{theorem}
\begin{proof}
The proof of the first result follows directly from the existence of an LDP for the measure $\Pi_N$, see theorem \ref{theo:LDP}, and is a straightforward adaptation of the one in \cite[Theorem 2.5.1]{moynot:99}. The proof of the second result uses the same method, making use of theorem \ref{theo:asLDP} below.
\end{proof}
We can in fact obtain the following quenched convergence analogue of \eqref{eq:PiNclosed}. 
\begin{theorem}\label{theo:asLDP}
For each closed set $F$ of $\mmeas(\T^\Z)$ and for almost all $(J,\Theta)$
\[
\lsup{N} \frac{1}{N} \log \left[Q^N(J,\Theta) (\hat{\mu}^N \in F)\right]  \leq -\inf_{\mu \in F} H(\mu).
\]
\end{theorem}
\begin{proof}
The proof is a combination of Tchebyshev's inequality and the Borel-Cantelli lemma and is a straightforward adaptation of the one in \cite[Theorem 2.5.4, Corollary 2.5.6]{moynot:99}.
\end{proof}
We define $\check{Q}^N(J^N,\Theta) = \frac{1}{N}\sum_{j=-n}^n Q^N(J^N,\Theta)\circ S^{-\olivier{j}}$, where we recall the shift operator $S$ defined at the start of section \ref{subsection:laws-uncoupled-coupled}. 
\begin{corollary}\label{cor:weakconv}
Fix $M$ and let $N> M$. For almost every $J$ and $\Theta$, and all $h \in \mathcal{C}_b(\T^M)$,
\begin{align*}
\lim_{N \to \infty} \int_{\T^M}h(u)\,\check{Q}^{N,M}(J^{\olivier{N}},\Theta)(du)&=\int_{\T^M} h(u)\,d\mu_e^M(u).\\
\lim_{N \to \infty} \int_{\T^M}h(u)\,Q^{N,M}(du)&=\int_{\T^M} h(u)\,d\mu_e^M(u).
\end{align*}
That is, the $M^{th}$ marginals $\check{Q}^{N,M}(J,\Theta)$ and $Q^{N,M}$ converge weakly to $\mu_e^M$ as $N\to\infty$. 
\end{corollary}
\begin{proof}
It is sufficient to apply theorem \ref{theo:weakconv} in the case where $\Phi$ in $\mathcal{C}_b(\mathcal{M}_{1,s}^+(\T^\Z))$ is defined by
\[
\Phi(\mu)=\int_{\T^M} h\,d\mu^M
\]
and to use the fact that $Q^N,\check{Q}^N(J,\Theta) \in \check{\M}_{1}^+(\T^N)$ (lemma \ref{lemma:PQdag}).
\end{proof}
We have the following analogue of the Ergodic Theorem. We may represent the ambient probability space by $\mathfrak{W}$, where $\omega\in\mathfrak{W}$ is such that $\omega = (J_{ij},\theta_j,B_{st}^j,u^j_0)$, where $i,j\in\Z$ and $1\leq s,t \leq T$. We denote the probability measure governing $\omega$ by $\mathfrak{P}$. Let $u^{(N)}(\omega)\in \T^N$ be defined by \eqref{eq:U}. As an aside, we may then understand $Q^N(J,\Theta)$ to be the conditional law of $\mathfrak{P}$ on $u^{(N)}(\omega)$, for given $(J_{ij},\theta_i)_{i,j=-n}^n$. 
\begin{theorem}
Fix $M > 0$ and let $h\in C_b(\T^M)$. For $u^{(N)}(\omega)\in\T^{N}$ (where $N>M$) and $|j|\leq n$, let $\check{u}^{(N),j}(\omega) = (u^{(N),j}(\omega),u^{(N),j+1}(\omega),\ldots,u^{(N),j+M-1}(\omega))$, the indexing being taken modulo $N$. Then $\mathfrak{P}$ almost surely,
\begin{equation}\label{eq:Mthmarginalconverge}
\lim_{N\to\infty} \frac{1}{N} \sum _{j=-n}^n h\left(\check{u}^{(N),j}(\omega)\right) = \int_{\T^M}h(u)d\mu_e^M(u).  
\end{equation}
That is, $\hat{\mu}^N(u^{\olivier{(N)}}(\omega))$ converges $\mathfrak{P}$-almost-surely to $\mu_e$.
\end{theorem}
\begin{proof}
Our proof is an adaptation of \cite{moynot:99}. We may suppose without loss of generality that $\int_{\T^M}h(u) d\mu_e(u) = 0$. For $p>1$ let 
\begin{equation*}
F_p = \left\lbrace \mu\in\M_{1,s}^+(\T^{\Z})| \left|\int_{\T^M}h(u)d\mu_e^M(u)\right| \geq \frac{1}{p} \right\rbrace.
\end{equation*}
Since $\mu_e\notin F_p$, but it is the unique zero of $H$, it follows that $\inf_{F_p}H = m > 0$. Thus by theorem \ref{theo:LDP} there exists an $N_0$, such that for all $N>N_0$,
\begin{equation*}
Q^N\left(\hat{\mu}^N\in F_p\right) \leq \exp\left(-mN\right). 
\end{equation*}
However
\begin{equation*}
\mathfrak{P}\left(\omega |\hat{\mu}^N(u^{(N)}(\omega))\in F_p\right) = Q^N\left( u | \hat{\mu}^N(u)\in F_p\right).
\end{equation*}
Thus 
\begin{equation*}
\sum_{N=1}^{\infty}\mathfrak{P}\left(\omega |\hat{\mu}^N(u^{(N)}(\omega))\in F_p\right) < \infty.
\end{equation*}
We may thus conclude from the Borel-Cantelli Lemma that $\mathfrak{P}$ almost surely, for every $\omega\in\mathfrak{W}$, there exists $N_p$ such that for all $N\geq N_p$,
\[
\left|\frac{1}{N}\sum_{j=-n}^n h\left(\check{u}^{(N),j}(\omega)\right)\right| \leq \frac{1}{p}. 
\]
This yields \eqref{eq:Mthmarginalconverge} because $p$ is arbitrary. The convergence of $\hat{\mu}^N(u^{\olivier{(N)}}(\omega))$ is a direct consequence of \eqref{eq:Mthmarginalconverge}, since this means that each of the $M^{th}$ marginals converge.
\end{proof}

\subsection{Possible extensions}
Our results hold true if we assume that equation \eqref{eq:U} is replaced by the more general equation
\[
U^j_{t}=\sum_{k=1}^l \gamma_k U^j_{t-k}+\sum_{i=-n}^{n} J_{ji}^N
f(U^i_{t-1})+\theta_j+B^j_{t-1}, \quad j=-n,\ldots,n \quad t=l,\ldots,T,
\]
where $l$ is a positive integer strictly less than $T$ (in practice much smaller). This equation accounts for a more complicated "intrinsic" dynamics of the neurons, i.e. when they are uncoupled. The parameters $\gamma_k$, $k=1\cdots l$ must satisfy some conditions to ensure stability of the uncoupled dynamics. 

This result can be straightforwardly extended to the case when the noise is correlated but stationary Gaussian, that is ${\rm cov}(B^{j}_{s},B^{k}_{t})$ is some function of $s,t$ and $(k-j)$. It can also be easily extended to the case that the initial distribution is correlated but mixing, using the Large Deviation Principle in \cite{chiyonobu-kusuoka:88}.

The hypothesis that the synaptic weights are Gaussian is somewhat unrealistic from the biological viewpoint. In his PhD thesis \cite{moynot:99}, Moynot has obtained some preliminary results in the case of uncorrelated weights. We think that this is also a promising avenue.

Moynot again, in his thesis, has extended the uncorrelated weights case, to include two populations with different (Gaussian) statistics for each population. This is also an important practical problem in neuroscience. Extending Moynot's result to the correlated case is probably a low hanging fruit.

Last but not least, the solutions of the equations for the mean and covariance operator of the measure minimizing the rate function derived in section \ref{sect:Hminimum} and their numerical simulation are very much worth investigating and their predictions confronted to biological measurements.

\subsection{Discussion}

In recent years there has been a lot of effort to mathematically justify neural-field models, through some sort of asymptotic analysis of finite-size neural networks. Many, if not most, of these models assume / prove some sort of thermodynamic limit, whereby if one isolates a particular population of neurons in a localised area of space, they are found to fire increasingly asynchronously as the number in the population asymptotes to infinity.\footnote{We noted in the introduction that this is termed propagation of chaos by some.} Indeed this was the result of Moynot and Samuelides. However our results imply that there are system-wide correlations between the neurons, even in the asymptotic limit. The key reason why we do not have propagation of chaos is that the Radon-Nikodym derivative $\frac{dQ^N}{dP^N}$ of the average laws in proposition \ref{prop:RNderiv1} cannot be tensored into $N$ i.i.d. processes; whereas the simpler assumptions on the weight function $\Lambda$ in Moynot and Samuelides allow the Radon-Nikodym derivative to be tensored. A very important implication of our result is that the mean-field behaviour is insufficient to characterise the behaviour of a population. Our limit process $\mu_e$ is system-wide and ergodic. Our work challenges the assumption held by some that one cannot have a `concise' macroscopic description of a neural network without an assumption of asynchronicity at the local population level.

The utility of this paper extends well beyond the identification of the limit law $\mu_e$. The LDP provides a powerful means of assessing how quickly the empirical measure converges to its limit. In particular, it provides a means of assessing the probability of finite size effects. For example if it could be shown that the rate function $H$ is sharply convex everywhere, then one would be more confident that the system converges quickly to its limit law. The rate functions of many classical LDPs, such as the one in lemma \ref{lemma:PinuNLDP}, are indeed convex (in fact the rate function $H^\nu(\cdot)$, for fixed $\nu$, is affine). However it is not clear whether our rate function $H$ is convex. Indeed if it could be shown that the rate function $H$ is not sharply convex, and in particular that it has a local minimum at another point $\mu_m$, then perhaps if $N$ is not too great there could be a reasonable probability that the empirical measure lies close to $\mu_m$. The upshot of this discussion is that further exploration of the topology of the rate function $H$ could be a very fruitful avenue of research for assessing the probability of finite-size effects. It would be of interest to compare our LDP with other analyses of the rate of convergence of neural networks to their limits as the size asymptotes to infinity. This includes the system-size expansion of Bressloff \cite{bressloff:09}, the path-integral formulation of Buice and Cowan \cite{buice-cowan:07} and the systematic expansion of the moments by (amongst others) \cite{ginzburg-sompolinsky:94,elboustani-destexhe:09,buice-cowan-etal:10}. 

\vspace{0.5cm}

\noindent
{\bf Acknowledgements}\\
Many thanks to Bruno Cessac whose suggestion to look at process-level empirical measures and entropies has been very useful and whose insights into the physical interpretations of our results have been very stimulating.

This work was partially supported by the European Union Seventh Framework Programme (FP7/2007-
2013) under grant agreement no. 269921 (BrainScaleS), no. 318723 (Mathemacs), and by the ERC advanced grant NerVi no. 227747.
\appendix

\section{Proof of Lemma \ref{lemma:PinuNLDP}}\label{section:proofLemmaLDP}

We note $\undt{R}^N $ the image law of $\undt{P}^{\otimes N}$ under $\undt{\hat{\mu}}^N$. This satisfies a strong LDP with good rate function \olivier{see \eqref{eq:I3R})}
\begin{equation}\label{eq:I3R2}
I^{(3)}\left(\undt{\mu},\undt{P}^{\Z}\right) = I^{(3)}\left(\undt{\mu}_0,\undt{\mu}_I^{\Z}\right) + \int_{\R^\infty}I^{(3)}\left(\undt{\mu}_{u_0},\undt{P}^{\Z}_{1,T}\right)d\undt{\mu}_{\olivier{0}}(u_0).
\end{equation}
This may be obtained by applying the contraction principle to the result in Theorem \ref{theorem:Rnexptight}. We recall that $\undt{\mu}_{u_0}$ is considered to be a probability measure on $\mM_{1,s}^+(\T_T^{\Z})$ (the definition of the latter is at the start of section \ref{subsection:laws-uncoupled-coupled}), and we note from \eqref{eq:Punder} that $\undt{P}^{\Z}_{u_0} = \undt{P}^{\Z}_{1,T}$, i.e. it is independent of $u_0$.

We obtain the LDP governing $\undt{\Pi}^{\olivier{\nu},N}$ by applying a contraction principle to the LDP governing $\undt{R}^N$. A proof of an LDP for stationary Gaussian processes over $(\R^d)^{\Z}$ has already been obtained in \cite{baxter-jain:93}, however these authors do not provide an explicit expression for the rate function. We therefore adapt their proof to our problem, which has a non-Gaussian initial condition, and we derive the required expression for the rate function. We assume for simplicity throughout this appendix that $c^{\nu}=0$; the results may be easily generalised. For $\omega\in [-\pi,\pi[$, let
\begin{equation}
\sqrt{{\rm Id}_T + \sigma^{-2}\tilde{K}^{\nu}(\omega)} = \sum_{j=-\infty}^\infty F^j\exp\left(-i j\omega\right).
\end{equation}
The $T \times T$ matrixes $F^j$ are the coefficients of the absolutely converging Fourier series of the positive square root $\sqrt{{\rm Id}_T + \sigma^{-2} \tilde{K}^{\nu}(\omega)}$.
Define $\tau^\nu:\T^\Z\to\T^\Z$ and $\tau^\nu_{(M)}:\T^\Z\to\T^\Z$ as follows. We let $\left(\tau^\nu(u)\right)^k_0 = \left(\tau^\nu_{(M)}(u)\right)^k_0 =u^k_0$ and (for $1\leq s\leq T$) 
\begin{align}
\left(\tau^\nu(u)\right)^k_s &= \sum_{j=-\infty}^\infty\sum_{t=1}^T F^j_{ts} u^{k-j}_t ,\nonumber\\ 
\left(\tau^\nu_{(M)}(u)\right)^k_s &= \sum_{|j|\leq m}\sum_{t=1}^T F^j_{ts}\left(1-\frac{|j|}{M}\right)u^{k-j}_t,\label{eq:tauFejer}
\end{align}
where $M=2m+1$. We note that $\tau^\nu_{(M)}$ is a continuous map, but $\tau^\nu$ is not continuous (in general). We note that $(\undt{P}^{\Z}\circ(\tau^\nu)^{-1})_{1,T}$ has spectral density $\sigma^2\rm{Id}_T + \tilde{K}^\nu$, and the spectral density of $(\undt{P}^{\Z}\circ(\tau^\nu_{(M)})^{-1})_{1,T}$ is defined to be $h_{(M)}(\omega)$. We write $\epsilon^\nu(M) = \sup_{\omega\in [-\pi,\pi[}\norm{\sqrt{\sigma^2 \rm{Id}_T + \tilde{K}^\nu(\omega)}-\sqrt{h_{(M)}(\omega)}}^2$ (this is the Euclidean norm). By Fejer's Theorem, $\epsilon^{\nu}(M)\to 0$ as $M\to\infty$.

\begin{theorem}\label{lemma:PiNLDP1}$\undt{\Pi}^{\nu,N}$ satisfies a strong LDP with good rate function 
\begin{equation}\label{eq:Hnudefnappend}
\undt{H}^{\nu}(\undt{\mu}) = I^{(3)}\left(\undt{\mu}\circ(\tau^{\nu})^{-1}\right),
\end{equation}
for $\mu\in\mathcal{E}_2$. If $\mu\notin\mathcal{E}_2$, then $\undt{H}^{\nu}(\undt{\mu}) = \infty$.
\end{theorem}
Our proof makes use of results in \cite{donsker-varadhan:85} and \cite{baxter-jain:93}. Before we prove this theorem, we require some preliminary lemmas. We define, in analogy to \eqref{eq:Gamma1nu} and \eqref{eq:Gamma2nuspectral},

\begin{align*}
\Gamma_{1,(M)}(\nu) &=-\frac{1}{4\pi} \int_{-\pi}^\pi
\left(\log \det \left({\rm Id}_{T}+\frac{1}{\sigma^2} \tilde{K}_{(M)}^\nu(\omega)\right)\right)\,d\omega \\
\Gamma_{2,(M)}^\nu(\mu) &=\frac{1}{2\sigma^2}
\frac{1}{2\pi}\int_{-\pi}^{\pi} \tilde{A}^\nu_{(M)}(-\omega) :
  \tilde{v}^\mu(d\omega)
 \end{align*}
where $\tilde{A}^{\nu}_{(M)} = \tilde{K}^{\nu}_{(M)}\left(\sigma^2\rm{Id}_T + \tilde{K}^{\nu}_{(M)}\right)^{-1}$ and $\tilde{K}^{\nu}_{(M)}(\omega) = h_{(M)}(\omega) - \sigma^2 \rm{Id}_T$. Since we have assumed that $c^{\nu}=0$, it may be shown that $\Gamma_{1,(M)}(\nu)$ converges to $\Gamma_{1}(\nu)$ and, for all $\mu\in\mathcal{E}_2$, $\Gamma^{\nu}_{2,(M)}(\mu)$ converges to $\Gamma^{\nu}_{2}(\mu)$ as $M\to\infty$.

Let $(\zeta^j)$ be i.i.d random variables in $\T$ governed by $\undt{P}^{\Z}$. Let $(\undt{\Pi}^N_{(M)})$ be the image laws of the empirical measures generated by the stationary  sequence $\tau_{(M)}(\zeta)$. Since $\tau_{(M)}$ is continuous, an application of the contraction principle to \eqref{eq:I3R2} dictates that $(\undt{\Pi}^N_{(M)})$ satisfies a strong LDP with good rate function given by
\begin{equation}\label{eq:HMLDP}
\undt{H}_{(M)}^{\nu}\olivier{(\undt{\mu})} = \inf_{\undt{\nu}: \undt{\mu} = \undt{\nu}\circ\tau_{(M)}^{-1}} I^{(3)}(\undt{\nu},\undt{P}^{\Z}).
\end{equation}

\begin{lemma}\label{lem:hnusemicont}
If $\mu\in\mathcal{E}_2$, then $\undt{H}^{\nu}(\undt{\mu})$ (as defined in \eqref{eq:Hnudefnappend}) is equal to 
\[
I^{(3)}\left(\undt{\mu},\undt{P}^{\Z}\right) - \Gamma^{\nu}(\mu).
\]
$\undt{H}^{\nu}(\undt{\mu})$ is lower-semi-continuous (as a function of $\undt{\mu}$).
\end{lemma}
\begin{proof}
We use results from the following section. From lemma \ref{lemma:lastinpaper}, we find that
\begin{equation*}
\undt{H}^{\nu}(\undt{\mu}) = I^{(3)}\left(\mu_0,\mu_{I}^{\olivier{\Z}}\right) + \int_{\R^\olivier{\Z}}I^{(3)}\left(\undt{\mu}_{u_0},\undt{P}^{\Z}_{1,T}\right) - \Gamma^{\Omega}(\mu_{u_0})\,\olivier{d\mu_{0}}(u_0),
\end{equation*}
where $\Gamma^{\Omega}(\mu_{u_0})$ is defined in \eqref{eq:GammaOmegaDefn}, and we substitute $\mathcal{K} = \rm{Id}_T + \sigma^{-2}\tilde{K}^{\nu}$. Now, after noting \eqref{eq:Gamma1nu} and \eqref{eq:Gamma2nuspectral}, and recalling our assumption that $c^{\nu} =0$, we find that\begin{align*}
\int_{\R^{\olivier{\Z}}}\Gamma^\Omega\left(\mu_{u_0}\right)\,\olivier{d\mu_{0}}(u_0) &= \Gamma^{\nu}_1 + \lim_{N\to\infty}\frac{1}{N}\int_{\T^N}\phi^N_{\infty}(\nu,v)\olivier{d\undt{\mu}_{v_0}}(v)\,\olivier{d\mu_0}(v_0), \\
&= \Gamma^\nu(\mu),
\end{align*}
from which the expression in the lemma follows. 

It remains for us to prove that if $\undt{\mu}^{(j)}\to\undt{\mu}$ then $\linf{j} \undt{H}^{\nu}(\undt{\mu}^{(j)}) \geq \undt{H}^{\nu}(\undt{\mu})$. We may assume without loss of generality that \[\linf{j} \undt{H}^{\nu}(\undt{\mu}^{(j)}) = \lim_{j\to\infty} \undt{H}^{\nu}(\undt{\mu}^{(j)}) \in \mathbb{R}\cup\infty.\]  

Suppose firstly that $\Exp^{\undt{\mu}^{(p_j)}}\left[ \norm{v^0}^2\right]\to\infty$ for some subsequence $(p_j)$. It follows from lemma \ref{lemma:E2enough} that $I^{(3)}(\undt{\mu}^{(p_j)},\undt{P}^{\Z})\to\infty$. It may be proved very similarly to lemma \ref{lem:acbound} that there exist constants $a>1$ and $c>0$ such that $\Gamma^{\nu}(\mu) \leq \frac{1}{a}\left( I^{(3)}(\undt{\mu},\undt{P}^{\Z}) + c\right)$. This means that $\undt{H}^{\nu}(\undt{\mu}^{(p_j)})\to\infty$ as well, and therefore $\lim_{j\to\infty} \undt{H}^{\nu}(\undt{\mu}^{(j)}) = \infty$ , satisfying the requirements of the lemma.

Otherwise we may assume that there exists a constant $l$ such that \newline$\Exp^{\undt{\mu}^{(j)}}\left[ \norm{u^0}^2\right]\leq l$ for all $j$. We therefore have that, for all $M$,
\begin{equation*}
\linf{j}\undt{H}^{\nu}(\undt{\mu}^{(j)}) = \linf{j}\left( \undt{H}^{\nu}_{(M)}(\undt{\mu}^{(j)}) + \Gamma^{\nu}_{(M)}\left(\undt{\mu}^{(j)}\right) - \Gamma^{\nu}\left(\undt{\mu}^{(j)}\right)\right).
\end{equation*}
Now 
\begin{align*}
\Gamma^{\nu}_{2,(M)}\left(\undt{\mu}^{(j)}\right) - \Gamma^{\nu}_2\left(\undt{\mu}^{(j)}\right) =& \frac{1}{4\pi \sigma^{2}}\int_{-\pi}^{\pi}\left(\tilde{A}^{\nu}_{(M)}(-\omega)-\tilde{A}^{\nu}(-\omega)\right):d\tilde{v}^{\mu^{(j)}}(\omega).\\
\leq & \frac{1}{4\pi\sigma^2}\int_{-\pi}^{\pi}\tilde{A}_{\rm{eig}}(-\omega)\sum_{s=1}^T d\tilde{v}^{\mu}_{ss}(\omega),
\end{align*}
where $\tilde{A}_{\rm{eig}}(\omega)$ is $\max\lbrace|\lambda|\rbrace$, where $\lambda$ is an eigenvalue of $\left(\tilde{A}^{\nu}_{(M)}(\omega)-\tilde{A}^{\nu}(\omega)\right)$. Since \newline$\sum_{s=1}^T\tilde{v}_{ss}^{\mu^{(j)}}\left( [-\pi,\pi[\right) = 2\pi \Exp^{\undt{\mu}^{(j)}}\left[ \norm{v^0}^2\right]$, we may write
\begin{equation*}
\left|\Gamma^{\nu}_{2,(M)}\left(\undt{\mu}^{(j)}\right) - \Gamma^{\nu}_2\left(\undt{\mu}^{(j)}\right)\right| \leq l \epsilon^*_{(M)},
\end{equation*}
where $\epsilon^*_{(M)} = \frac{T}{2\sigma^2} \sup_{\omega\in [-\pi,\pi[}\norm{\tilde{A}^{\nu}_{(M)}(\omega)-\tilde{A}^{\nu}(\omega)}\to 0$ as $M\to \infty$ because of the Fejer approximation in \eqref{eq:tauFejer}.
Thus
\begin{equation*}
\linf{j}\undt{H}^{\nu}(\mu^{(j)}) \geq \linf{j}\left( \undt{H}^{\nu}_{(M)}(\undt{\mu}^{(j)}) - l\epsilon^*_{(M)} -\left|\Gamma^{\nu}_{1,(M)}- \Gamma^{\nu}_1\right|\right).
\end{equation*}
However it follows from the Fejer approximation that $\left|\Gamma^{\nu}_{1,(M)}- \Gamma^{\nu}_1\right|\to 0$ as $M\to \infty$. In addition, $\linf{j}\undt{H}^{\nu}_{(M)}(\undt{\mu}^{(j)}) \geq \undt{H}^{\nu}_{(M)}(\mu)$ due to the lower semi-continuity of $\undt{H}^{\nu}_{(M)}$. On taking $M\to \infty$, we therefore find that
\begin{equation*}
\linf{j}\undt{H}^{\nu}(\undt{\mu}^{(j)}) \geq \undt{H}^{\nu}(\undt{\mu}^{(j)}).
\end{equation*}
\end{proof}
 \begin{lemma}\label{lem:tempend}
If $0 < \lambda < \frac{1}{2\varepsilon^\nu(M)}$ then for all odd $N=2n+1$
\[
\frac{1}{N}\log\Exp^{\undt{P}^{\Z}}\left[\exp\left(\frac{\lambda}{\sigma^2}\sum_{k=-n}^n\norm{\tau^\nu_{(M)}(u){}^k - \tau^\nu(u){}^k}^2\right)\right]\leq -\frac{1}{2}\log\left(1-2\lambda\varepsilon^\nu(M)\right).
\]
\end{lemma}
The proof is almost identical to that in \cite{donsker-varadhan:85}. We are now ready to prove theorem \ref{lemma:PiNLDP1}.
\begin{proof}
It follows from the above lemma, the LDP for $\undt{\Pi}^{\nu,N}_{(M)}$ in \eqref{eq:HMLDP} and \cite[Theorem 4.9]{baxter-jain:93} that $\undt{\Pi}^{\nu,N}$ satisfies a strong LDP with good rate function
\begin{equation}\label{eq:LDPtmp3}
\lim_{\delta\to 0}\linf{M}\inf_{\gamma\in B^\delta(\mu)}\undt{H}^{\nu}_{(M)}(\undt{\gamma}),
\end{equation}
where $B^{\delta}(\mu) = \lbrace\gamma: \check{d}(\mu,\gamma) \leq \delta\rbrace$. Here $\check{d}$ is the Prohorov metric over $\M_{1,s}^+(\T^{\Z})$, induced by the metric $\grave{d}(x,y) = \sum_{j=-\infty}^{\infty} 2^{-|j|} \min\left(\norm{x^j-y^j},1\right)$ over $\T^{\Z}$. Note that the above expression with a closed ball is equivalent to the expression with the open ball in \cite{baxter-jain:93}. 

It follows, very similarly to  \cite{donsker-varadhan:85}, that if $F$ is a compact set, then
\begin{equation}\label{eq:tmp5}
\linf{M}\inf_{\gamma\in F} \undt{H}^{\nu}_{(M)}(\undt{\gamma}) = \inf_{\gamma\in F} \undt{H}^{\nu}(\undt{\mu}).
\end{equation} 
The theorem now follows since $\undt{H}^{\nu}$ is lower semicontinuous, by lemma \ref{lem:hnusemicont}.
\end{proof}

\subsection{Properties of the Entropy}\label{sect:EntropyProperties}

Let $\undt{\xi}$ be a zero-mean stationary measure on $\M_{\olivier{1,s}}^\olivier{+}((\R^T)^{\Z})$. Let $\mathcal{K}:[-\pi,\pi]\to \R^{\olivier{T \times T}}$ possess an absolutely convergent Fourier Series, i.e. $\mathcal{K}(-\pi) = \mathcal{K}(\pi)$, and be such that the eigenvalues of $\mathcal{K}(\omega)$ are strictly greater than zero for all $\omega$. We require that $\mathcal{K}$ is the density of a stationary sequence, which means that we must also assume that for all $\omega$
\[
\mathcal{K}(-\omega) = {}^t \mathcal{K}(\omega) = \mathcal{K}(\omega)^*.
\]
This means, in particular, that $\mathcal{K}(\omega)$ is Hermitian. We write
\begin{align}
\left(\Omega(u)\right)^k_s &= \sum_{j=-\infty}^\infty\sum_{t=1}^T \mathcal{R}^j_{ts} u^{k-j}_t, \text{ where } \\
\sum_{j=-\infty}^{\infty} \mathcal{R}^j\exp\left(- ij\omega\right) &= \mathcal{K}^{\frac{1}{2}},
\end{align}
Here $\mathcal{K}^{\frac{1}{2}}$ is understood to be the positive Hermitian square root of $\mathcal{K}$. In this section, we determine a general expression for
$I^{(3)}\left(\undt{\xi}\circ\Omega^{-1},\undt{P}_{1,T}\right)$. We are generalising the result for $T=1$ given in \cite{donsker-varadhan:85}. These results are necessary for the proofs in the previous section.

We similarly write that
\begin{align}
\left(\Delta(u)\right)^k_s &= \sum_{j=-\infty}^\infty\sum_{t=1}^T S^j_{ts} u^{k-j}_t , \text{ where } \\
\sum_{j=-\infty}^{\infty} S^j\exp\left(- ij\omega\right) &= \mathcal{K}^{-\frac{1}{2}}.
\end{align}
As previously, $\mathcal{K}^{-\frac{1}{2}}$ is understood to be the positive Hermitian square root. The Fourier Series of $\mathcal{K}^{-\frac{1}{2}}$ is absolutely convergent as a consequence of Wiener's Theorem. We note that $\mathcal{R}^{-j} = {}^t \mathcal{R}^j$ and $S^{-j} = {}^t S^j$. Similarly to definition \ref{def:E2}, we let $\mathcal{E}_{2,T}$ be the subset  of $\mathcal{M}_{1,s}^+((\mathbb{R}^T)^\Z)$ such that
\[
\mathcal{E}_{2,T}=\{ \underline{\mu} \in \mathcal{M}_{1,s}^+((\mathbb{R}^T)^{\Z})\, | \,
\Exp^{\underline{\mu}}[\|v^0\|^2] < \infty \}.
\]
\begin{lemma}\label{lemma:inversesigmatau}
For all $\olivier{\undt{\xi}}\in\mathcal{E}_{2,T}$, $\undt{\xi}\circ\Delta^{-1}$ and $\undt{\xi}\circ\Omega^{-1}$ are in $\mathcal{E}_{2,T}$ and
\[\undt{\xi}\circ\Delta^{-1}\circ\Omega^{-1} = \undt{\xi}\circ\Omega^{-1}\circ\Delta^{-1} = \undt{\xi}.\]
\end{lemma}
\begin{proof}
We make use of the following standard lemma from \cite{shiryaev:96}, to which the reader is referred for the definition of an orthogonal stochastic measure.
Let $(U^j)\in \mathbb{R}^T$ be a zero-mean stationary sequence governed by $\undt{\xi}\in\mathcal{E}_{2,T}$. Then there exists an orthogonal $\R^T$-valued stochastic measure $Z^{\undt{\xi}}=Z^{\undt{\xi}}(\Delta)$ ($\Delta \in \mathcal{B}([-\pi,\pi[)$, such that for every $j\in\Z$ ($\undt{\xi}$ a.s.) 
\begin{equation}\label{eq:stochmeasuredefn}
U^j = \frac{1}{2\pi}\int_{-\pi}^{\pi}\exp(i\omega j)Z^{\undt{\xi}}(d\omega).
\end{equation}
Conversely any orthogonal stochastic measure defines a zero-mean stationary sequence through \eqref{eq:stochmeasuredefn}.
It may be inferred from this representation that
\begin{align*}
Z^{\undt{\xi}\circ\Delta^{-1}}(d\omega) &= {}^t\mathcal{K}^{-\frac{1}{2}}(\omega)Z^{\undt{\xi}}(d\omega),\\
Z^{\undt{\xi}\circ\Omega^{-1}}(d\omega) &= {}^t\mathcal{K}^{\frac{1}{2}}(\omega)Z^{\undt{\xi}}(d\omega).
\end{align*}
The proof that this is well-defined makes use of the fact that $\mathcal{K}^{\frac{1}{2}}$ and $\mathcal{K}^{-\frac{1}{2}}$ are uniformly continuous, since their Fourier Series' each converge uniformly. This gives us the lemma. We note for future reference that, if $\undt{\xi}$ has spectral density $\mathcal{X}(\theta)$, then the spectral density of $\undt{\xi}\circ\Omega^{-1}$ is
\[
\mathcal{K}^{\frac{1}{2}}(\omega)\mathcal{X}(\omega)\mathcal{K}^{\frac{1}{2}}(\omega).
\]
\end{proof}
It remains for us to determine a specific expression for $I^{(3)}(\undt{\xi}\circ\Omega^{-1},\undt{P}^{\Z})$ when $\undt{\xi}\in\mathcal{E}_{2,T}$. 
We define
\begin{equation}\label{eq:GammaOmegaDefn}
\Gamma^{\Omega}(\undt{\xi}) = \frac{1}{2}\left(\Exp^{\undt{\xi}}\left[\sigma^{-2}\norm{v^0}^2\right]-\Exp^{\undt{\xi}\circ\Omega^{-1}}\left[\sigma^{-2}\norm{v^0}^2\right] \right) -\frac{1}{4\pi}\int_{-\pi}^{\pi}\log\det\left(\mathcal{K}(\omega)\right)d\omega.
\end{equation}
\begin{lemma}
For all $\undt{\xi}\in\mM_{1,s}^+(\T^{\Z})$,
\begin{equation*}
I^{(3)}\left(\undt{\xi}\circ \Omega^{-1},\undt{P}_{1,T}^{\Z}\right) \leq I^{(3)}\left(\undt{\xi},\undt{P}^{\Z}_{1,T}\right) - \Gamma^{\Omega}(\undt{\xi}).
\end{equation*}
\end{lemma}
\begin{proof}
We assume for now that there exists a $q$ such that $\mathcal{R}^{j}_{st} = 0$ for all $s,t$ and $j \geq q$, denoting the corresponding map by $\Omega_q$. Let $\mathcal{N}_q^L$ be the $TL\times TL$ block-circulant matrix, with the $j^{th}$ block given by $\mathcal{R}^j$. Let $\grave{\undt{\xi}}_{q,N} = \undt{\xi}^{N}_{u_0}\circ(\mathcal{N}_q^N)^{-1}$. It follows from this assumption that the $L=2l+1$ dimensional marginals of $\grave{\undt{\xi}}_{q,N}$ and $(\undt{\xi}\circ\Omega_q)^{-1})_{u_0}$ are the same, as long as $l\leq n-q$ (where $N=2n+1$). Thus
\begin{align}
I^{(2)}\left((\undt{\xi}\circ(\Omega_q)^{-1})^L_{u_0},\undt{P}^{\otimes L}_{1,T}\right) =& I^{(2)}\left((\grave{\undt{\xi}}_{q,N})^{L},\undt{P}^{\otimes L}_{1,T}\right) \nonumber\\
\leq & I^{(2)}\left(\grave{\undt{\xi}}_{q,N},\undt{P}^{\otimes N}_{1,T}\right).\label{eq:I2inequality}
\end{align}
This last inequality follows from a property of the K\"ullback-Leibler Divergence $I^{(2)}$, namely that it is nondecreasing as we take a `finer' $\sigma$-algebra. If $\undt{\xi}^N$ does not have a density for some $N$ then $I^{(3)}\left(\undt{\xi},\undt{P}^{\Z}_{1,T}\right)$ is infinite and the lemma is trivial. If otherwise, we may readily evaluate $I^{(2)}\left(\grave{\undt{\xi}}_{q,N},\undt{P}^{\otimes N}_{1,T}\right)$ using a change of variable to find that
\begin{multline*}
 I^{(2)}\left(\grave{\undt{\xi}}_{q,N},\undt{P}^{\otimes N}_{1,T}\right) = I^{(2)}\left(\undt{\xi}^N,\undt{P}^{\otimes N}_{1,T}\right) + \frac{1}{2\sigma^2}\Exp^{\undt{\xi}^N}\left[\norm{\mathcal{N}_q^N u}^2-\norm{u}^{2}\right]\\+\frac{1}{2}\log\det\left(\mathcal{N}_q^N\right).
\end{multline*}
We divide \eqref{eq:I2inequality} by $L$, substitute the above result, and finally take $L\to\infty$ (while fixing $N=L+2q$) to find that
\begin{align*}
&I^{(3)}\left((\undt{\xi}\circ(\Omega_q)^{-1}),\undt{P}_{1,T}^{\Z}\right) \leq I^{(3)}\left(\undt{\xi},\undt{P}^{\Z}_{1,T}\right)\nonumber\\ &+\frac{1}{2\sigma^2}\left( \Exp^{\undt{\xi}\circ(\Omega)^{-1}}\left[\norm{u^0}^2\right] - \Exp^{\undt{\xi}}\left[\norm{u^0}^2\right]\right) +\frac{1}{4\pi}\int_{-\pi}^{\pi}\log\det\left(\mathcal{K}(\omega)\right)d\omega\nonumber \\
&= I^{(3)}\left(\undt{\xi},\undt{P}^{\Z}_{1,T}\right) - \Gamma^{\Omega}_{q}(\undt{\xi}).
\end{align*}
Here $\Gamma^{\Omega}_{q}(\undt{\xi})$ is equal to $\Gamma^{\Omega}(\undt{\xi})$ as defined above, subject to the above assumption that $\mathcal{R}^j = 0$ for $j>q$. On taking $q\to\infty$, it may be readily seen that $\Gamma^{\Omega}_{q}\to \Gamma^\Omega$ pointwise. Furthermore the lower semicontinuity of $I^{(3)}$ dictates that
\begin{equation*}
I^{(3)}\left((\undt{\xi}\circ\Omega^{-1}),\undt{P}_{1,T}^{\Z}\right) \leq \linf{q}
I^{(3)}\left((\undt{\xi}\circ\Omega_q^{-1}),\undt{P}_{1,T}^{\Z}\right),
\end{equation*}
which gives us the lemma.
\end{proof}
\begin{lemma}\label{lemma:lastinpaper}
If $\undt{\xi}\in\mathcal{E}_{2,T}$, then
$I^{(3)}\left(\undt{\xi}\circ\Omega^{-1},\undt{P}^{\Z}_{1,T}\right) = I^{(3)}\left(\undt{\xi},\undt{P}^{\Z}_{1,T}\right) - \Gamma^{\Omega}\left(\undt{\xi}\right)$.
\end{lemma}
\begin{proof}
We find, similarly to the previous lemma, that if $\gamma\in\mM_{1,s}(\T^{\Z})$ then
\begin{multline}\label{eq:I3temp4}
I^{(3)}\left((\undt{\gamma}\circ\Delta^{-1}),\undt{P}_{1,T}^{\Z}\right) \leq I^{(3)}\left(\undt{\gamma},\undt{P}^{\Z}_{1,T}\right) +  \\ \frac{1}{2\sigma^2}\left[ \Exp^{(\undt{\gamma}\circ\Delta^{-1})}\left[ \norm{u^0}^2\right] - \Exp^{\undt{\gamma}}\left[\norm{u^0}^2\right]\right] - \frac{1}{4\pi}\int_{-\pi}^{\pi}\log\det\mathcal{K}(\omega)d\omega.
\end{multline}
We substitute $\undt{\gamma} = \undt{\xi}\circ\Omega^{-1}$ into the above and, after noting lemma \ref{lemma:inversesigmatau}, we find that
\begin{align}
I^{(3)}\left(\undt{\xi},\undt{P}_{1,T}^{\Z}\right) &\leq I^{(3)}\left((\undt{\xi}\circ \Omega^{-1}),\undt{P}_{1,T}^{\Z}\right) +\nonumber\\ &\frac{1}{2\sigma^2}\left( \Exp^{\undt{\xi}}\left[ \norm{u^0}^2\right] - \Exp^{(\undt{\xi}\circ\Omega^{-1})}\left[\norm{u^0}^2\right]\right) - \frac{1}{4\pi}\int_{-\pi}^{\pi}\log\det\left(\mathcal{K}(\omega)\right)d\omega\nonumber\\
&=I^{(3)}\left((\undt{\xi}\circ\Omega^{-1}),\undt{P}^{\Z}_{1,T}\right) + \Gamma^{\Omega}(\undt{\xi}).\label{eq:I3temp5}
\end{align}
The result now follows from the previous lemma and \eqref{eq:I3temp5}.
\end{proof}

\end{document}